%% file: qi2.tex
\pdfoutput=1
\documentclass[a4paper, 12pt]{article}

\usepackage{standalone}
\usepackage{header}
\usepackage{abbrev}
\usepackage{abbrev_qi}

\usepackage{enumerate}

\title{The Interface of the FK-representation of the Quantum Ising Model Converges to the $\SLE_{16/3}$}
\author{Jhih-Huang Li}
\date{}

\begin{document}

\maketitle

\vspace{1cm}

\begin{abstract}
We study the interface in the FK-representation of the 1D quantum Ising model and show that in the limit, it converges to the $\SLE_{16/3}$ curve.
\end{abstract}

\vspace{1cm}

\tableofcontents

\input{intro2.tex}
\input{prem2.tex}
\input{sdca2.tex}
\input{obs2.tex}
\input{cvg2.tex}

\appendix

\input{residue2.tex}

\bibliographystyle{alpha}
\bibliography{qi}

\Address

\end{document}

%% file: intro2.tex
\section{Introduction}

\begin{killcontents}
The conformal invariance of models of planar statistical mecanics at criticality was first suggested by Belavin, Polyakov and Zamolodchikov in \cite{BPZ84-1, BPZ84-2}.
This gave rise to Conformal Field Theory, consisting of the study of scaling limits of quantum fields, which should be conformally invariant.
However, mathematically, the notion of conformal invariance needed to be defined, and it turned out that the interface of the model was a good candidate, encoding important information on the model itself.
Since then, the interface of some models have been proven to be an conformally invariant curve in the sense discussed below.

The family of SLE curves, or \emph{Schramm-Löwner Evolution}, introduced by Schramm in \cite{Schramm-SLE}, is a family of conformally invariant continuous curves which do not intersect themselves.
These curves are fractal objects which are pretty well-understood.
Thus, proving the conformal invariance of a model to an SLE curve leads to the knowledge of these exponents.

In 2000, Kenyon used \emph{discrete holomorphicity} to show the conformal invariance of the dimer model \cite{Kenyon-dimer-conformal}.
Later on, Smirnov proved Cardy's formula for site-percolation on the triangular lattice \cite{Cardy, Smirnov-perco}.
Later on, Camia and Newman proved the convergence of the interface to $\SLE_6$ \cite{CamNew-perco-SLE}.
Schramm and Scheffield constructed the harmonic explorer and showed its convergence to $\SLE_4$ in \cite{SS-harm-expl}.
An extended notion of \emph{preholomorphicity}, or \emph{s-holomorphicity}, was introduced in \cite{Smirnov-towards, Smirnov-conformal} to show the conformal invariance of FK- and spin- Ising interfaces on the square lattice \cite{Smirnov-conformal, DS-conf-inv} and isoradial lattice \cite{CS-universality}.
Additionally, these conformally invariant limits were identified to be $\SLE_{16/3}$ in the FK-Ising case and $\SLE_3$ in the spin-Ising case \cite{Ising-SLE-convergence}.

To extract information from the model, the interface is a good object to study.
The techniques behind are based on complex analysis and the study of the so-called \emph{(para-)fermionic observable}, which encodes the interface.
Establishing its discrete holomorphicity and its convergence in the limit to a holomorphic function are the main points of study.
Moreover, the same observable was also used to show the constant connectivity of hexagonal lattice is $\sqrt{2 + \sqrt2}$ in \cite{DS-conn-const}.

The family of SLE curves is the first collection of curves which are conformally invariant, thus giving the first step towards the conformal field theory.
Recently, Benoist and Hongler showed that the collection of critical Ising loops converges to CLE(3), Conformal Loop Ensemble \cite{BH-Ising-CLE}.
\end{killcontents}

The quantum Ising model on $\bbZ$ is an exactly solvable one-dimensional quantum model \cite{Pfeuty-QI}.
It can be represented as the space-time evolution of a spin configuration, first introduced in \cite{AKN-perco-QI}.
Specifically, consider an initial configuration on the base graph $\mathbb{Z}$, and let it evolve according to the following quantum Hamiltonian,
$$
H = - \mu \sum_{(x, y) \in E} \Pauli{x}{3} \Pauli{y}{3} 
- \lambda \sum_{x \in V} \Pauli{x}{1}
$$
where the Pauli matrices are denoted by
\begin{equation}
\Pauli{}{1} = \left(
\begin{array}{cc}
0 & 1 \\ 1 & 0
\end{array}
\right),\quad
\Pauli{}{3} = \left(
\begin{array}{cc}
1 & 0 \\ 0 & -1
\end{array}
\right),
\end{equation}
which act on the spin configuration at each coordinate, indicated by the subscript.

The quantum Hamiltonian acts on the Hilbert space $\bigotimes_\bbZ \bbC^2$ where a Hilbert space of spin configuration $\bbC^2 \cong {\rm Vect} ( \ket{+}, \ket{-} )$ is associated to each $x \in \bbZ$.
We may represent $\ket{+} = (1, 0)$ and $\ket{-} = (0, 1)$ for example.
In the Hamiltonian, $\lambda$ and $\mu$ are two positive parameters where $\mu$ is the interaction term between particles at neighboring sites and $\lambda$ is the intensity of the external field.

\bigskip

The model can also be seen as a space-time evolution, consisting in developping the exponential operator $e^{-\beta H}$ where $\beta$ is interpreted as time \cite{AKN-perco-QI}.
In particular, we obtain the FK-representation, loop representation and random-current (random-parity) representation.
Readers may have a look at \cite{Ioffe-QI} for a nice and complete exposition on this topic.
These representations are useful in interpreting results from the classical Ising model \cite{GOS-entanglement, BG-QI-sharp, Bjornberg-infrared}.

\bigskip

Let us briefly describe the FK-representation of the model, a bit more details being given in Section \ref{sec:prem}.
In $\bbR^2$, we consider the collection of vertical real lines indexed by $\bbZ$, which we denote by $\bbZ \times \bbR$.
It is called \emph{(primal) semi-discrete lattice} below.
We put independent Poisson point processes of parameter $\lambda$ on each of these real lines whose points are called \emph{death points}.
Similarly, we consider the dual of $\bbZ \times \bbR$ which is $(\bbZ + \frac12) \times \bbR$.
It can also be seen as a collection of real lines, this time indexed by $\bbZ + \frac12$.
We also put independent Poisson point processes on each of the lines, but this time of parameter $\mu$.
These points are called \emph{bridges} and we draw at the same level a horizontal segment connecting the two neighboring vertical (primal) lines.
See Figure \ref{fig:connectivity} for an example.
Moreover, these two families of Poisson point processes are taken to be independent one of each other.

Consider a random configuration given by the above Poisson point processes and define the notion of \emph{connectivity}.
Two points $x$ and $y$ in the semi-discrete lattice are said to be connected if one can go from one to the other using only primal lines and horizontal bridges without crossing any death points.
The \emph{cluster} of a point $x$ is the largest subgraph containing $x$ which is connected in the sense mentioned above.
See Figure \ref{fig:connectivity} for an illustration.

\begin{figure}[htb]
  \centering
    \includegraphics[scale=1]{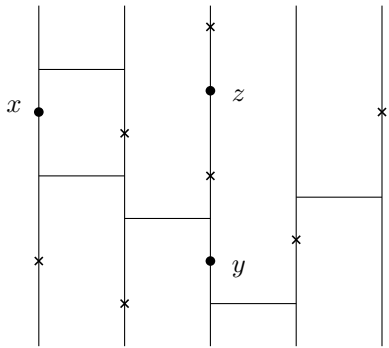}
    \caption{Only vertical lines are drawn, which are a part of the primal semi-discrete lattice. Horizontal segments are bridges and crosses are death points. Here, the points $x$ and $y$ are connected to each other but are disconnected from the point $z$.}
    \label{fig:connectivity}
\end{figure}

Having this notion of connectivity, we can talk about the FK-representation and the loop representation of the quantum Ising model.
The FK-representation is obtained by only looking at the clusters in the primal lattice.
We are interested in how the points in the primal lattice are connected, and the property of primal clusters.
The loop representation is obtained by looking at the clusters in both the primal and the dual lattices.
We draw all the interfaces separating them and we get a collection of loops or interfaces, depending on the boundary conditions.

The measure of the quantum Ising model at criticality is given as follows,
$$
\dd \bbP_{\lambda, \mu}^{\rm QI} (D, B) \propto 2^{k(D, B)} \dd \bbP_{\lambda, \mu}(D, B)
\propto \sqrt{2}^{l(D, B)} \dd \bbP_{\rho, \rho}(D, B)
$$
where $\bbP_{\lambda, \mu}$ is the law of Poisson point processes described above and $\rho = \sqrt{\lambda \mu}$.
Here, $D$ and $B$ are (random) sets of death points and bridges respectively.
Given a configuration of Poisson points $(D, B)$, the quantity $k(D, B)$ denotes the number of clusters in the FK-representation (on the primal lattice) and $l(D, B)$ denotes the number of loops in the loop representation (on the mid-edge lattice).

\bigskip

The interface in the FK-representation of the classical Ising model has been shown to be conformally invariant \cite{Smirnov-towards, Smirnov-conformal, DS-conf-inv, CS-universality} and the limiting curve can be identified with $\SLE_{16/3}$ \cite{Ising-SLE-convergence}.
Therefore, the same should also hold for the interface in the FK-representation / loop representation of the quantum Ising model.

\bigskip

In order to state the theorem, let us recall a few more important notions.
We consider a Dobrushin domain $(\Omega, a, b)$, \emph{i.e.} an open, bounded and simply connected set with two marked points on the boundary $a$ and $b$.
For every positive $\delta$, we can semi-discretize it by a subgraph $(\primaldomain, a_\delta, b_\delta)$ of the semi-discrete medial lattice $\frac\delta2 \bbZ \times \bbR$.
We shall look at the so-called \emph{Dobrushin boundary condition}, consisting of wired boundary condition on the arc $\arcab$ and free boundary condition on the arc $\arcba$.
In this case, the loop representation gives rise to a collection of loops and one interface connecting $a_\delta$ to $b_\delta$, separating the (primal) cluster connected to the wired arc and the (dual) cluster connected to the free arc.

In this paper, we prove the conformal invariance of the quantum Ising model by showing that the limit of interfaces when $\delta$ goes to 0 is conformally invariant.
This is the first quantum model proved to have such a property.
Here is the informal statement of our main theorem, the more precised version will be given in Theorem \ref{thm:main}.

\begin{thm*}
Let $(\Omega, a, b)$ be a Dobrushin domain.
Let $(\primaldomain, a_\delta, b_\delta)$ be its semi-discretized counterparts.
Define the FK-representation of the quantum Ising model on $(\primaldomain, a_\delta, b_\delta)$ and denote by $\interface$ the interface separating the (primal) wired boundary and the (dual) free boundary.
When $\delta$ goes to zero, the interface $\interface$ converges to the chordal Schramm-Löwner Evolution of parameter 16/3 in $(\Omega, a, b)$.
\end{thm*}

The proof is made possible by the similarity between the FK-representations of the quantum and the classical Ising models.
The FK-representation and the loop representation of the quantum model can be interpreted as the same representations of the classical model living on a more and more flattened rectangular lattice.
Thus, the proof almost comes from the same arguments as in the so-called isoradial case, except that some notions need to be adapted to the semi-discrete case.

\bigskip

Intuitively, using the universality of the classical Ising model \cite{CS-universality} on isoradial graphs would require an inversion of limits:
\begin{itemize}
\item On one hand, the universality result says that the classical FK-Ising model on a $\epsilon$-flattened (where $\epsilon$ is the flattened angle) isoradial rectangular lattice of meshsize $\delta$ has an interface which is conformally invariant in the limit $\delta \rightarrow 0$, provided that the angle $\epsilon$ is kept unchanged.
In this first approach, the lattice ``converges'' to the whole plane uniformly in all directions.
\item On the other hand, if we put the classical FK-Ising on a $\epsilon$-flattened $\delta$-isoradial rectangular lattice with flatter and flatter rectangles by making $\epsilon$ go to $0$, we would get continuous lines in the vertical direction, which is exactly the FK-representation of the quantum Ising model.
Therefore, to get the conformal invariance of the interface in the quantum FK-Ising, we would need to make $\delta$ go to 0, the distance between two neighboring vertical lines.
In this second approach, the lattice ``converges'' to the whole plane first in the vertical direction, then in the horizontal one.
\end{itemize}

\bigskip

The heuristic described above suggests strongly that the FK-representation of the quantum Ising model should also be conformally invariant in the limit, and that the interface in the limit should be the same as in the classical case.
However, one should make this argument mathematically rigorous, and that is why we work directly in the semi-discrete case.

In this paper, some classical notions need to be generalized and new tools be constructed.
We will define the Green's function on the semi-discrete lattice, give the notion of s-holomorphicity, show that the fermionic observable is s-holomorphic and give a proof of RSW property by the second moment method.
Everything is defined directly in the semi-discrete.
Sometimes, it will be easy to see the parallel with the discrete isoradial case, but sometimes it will be a bit more tricky.

\bigskip

Other results concerning the conformal invariance from the classical Ising model are also expected to have their counterparts in the quantum case.
For instance, the following results should be extendable.
Using fermionic spinor, Hongler and Smirnov proved the conformal invariance of the energy density in the planar Ising model \cite{HS-energy-Ising}.
Chelkak, Hongler and Izyurov showed that the magnetization and multi-spin correlations are conformally invariant in the scaling limit \cite{CHI-CI-spin}.
Recently, Hongler and Benoist gave a proof that the collection of critical Ising loops converges to CLE(3), Conformal Loop Ensemble \cite{BH-Ising-CLE}.

\bigskip

In this paper, we work directly on the semi-discrete graph to establish the conformal invariance of the interface in the FK-representation of the quantum Ising model and identify it.
This paper is divided as follows.
\begin{itemize}
\item In Section \ref{sec:prem}, we define the quantum Ising model properly and describe its geometrical representation in details.
\item In Section \ref{sec:SDCA}, the semi-discrete complex analysis is introduced.
In particular, we define a notion of \emph{s-holomorphicity}, differential operators and integration which are analogous to the continuum ones.
We also study counterparts of some classical objects and problems, \emph{e.g.} Brownian motion and Dirichlet boundary problem.
Moreover, we construct the Green's function via auxilary functions inspired by \cite{Kenyon-laplacian}, which is the main novel input of this part of the paper.
\item In Section \ref{sec:observables}, we define the (para-)fermionic observable $\obsmedial$ in the FK-represen\-tation and show that it is s-holomorphic.
This is the first step towards the conformal invariance of the interface.
\item In Section \ref{sec:CVG_thm}, we show the convergence theorem and identify the interface in the limit with an SLE curve.
We need to use the tools established in Section \ref{sec:SDCA}.
In addition, we prove the \emph{RSW property} by the second moment method inspired from \cite{RSW-harmonic}.
\end{itemize}

\begin{bibliormk}
The content of Section \ref{sec:observables} appeared recently in \cite{Bjornberg-obs} where s-holomorphicity is defined for the semi-discrete lattice.
The notion therein is formulated differently but is equivalent to the one we use in this paper.
A global review of the quantum Ising model is also given along with different representations.
It is also shown that in both the spin-representation and FK-representation, the fermionic observable satisfies the property of s-holomorphicity.
A brief idea of the proof towards conformal invariance of the interface is mentioned but the formal proof is not given there.
\end{bibliormk}

\section*{Acknowledgements} 
This research is supported by the NCCR SwissMAP, the ERC AG COMPASP, and the Swiss NSF.
The author is thankful to his advisors Hugo Duminil-Copin and Stanislav Smirnov for their help and constant support during his PhD.

%% file: prem2.tex
\section{Preliminaries} \label{sec:prem}

\subsection{Semi-discrete lattice}

The \emph{semi-discrete lattice} is defined by the Cartesian product $\bbZ \times \bbR$.
It can be seen as a collection of \emph{vertical lines} $\bbR$ indexed by $\bbZ$ with \emph{horizontal edges} connecting neighboring vertical lines with the same $y$-coordinate.
In our graphical representation, horizontal edges are not drawn for simplicity.

\begin{figure}[htb] \centering
  \begin{minipage}{.44\textwidth}
    \includegraphics[scale=0.8]{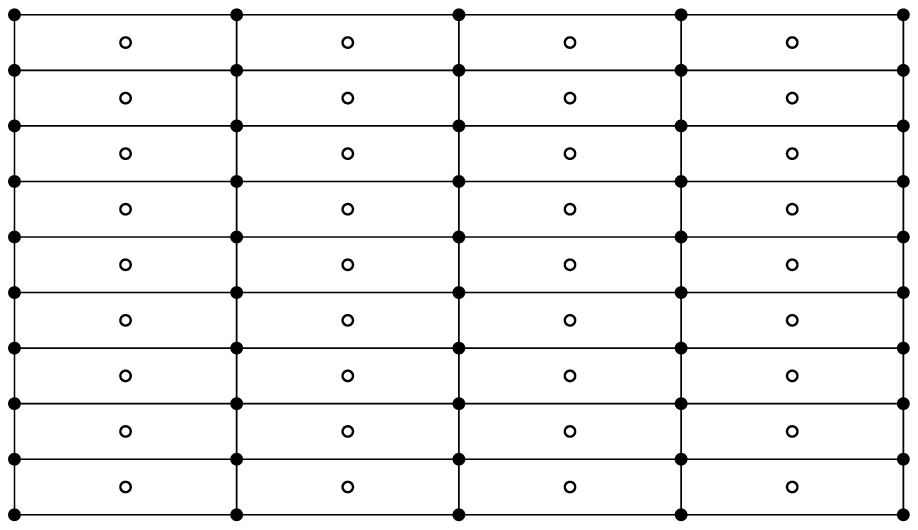}
    \caption{Black vertices and edges represent (a part of) a flattened rectangular lattice and white vertices its dual.}
    \label{fig:rec_lattice}
  \end{minipage}
  \hfill
  \begin{minipage}{.44\textwidth}
    \includegraphics[scale=0.8]{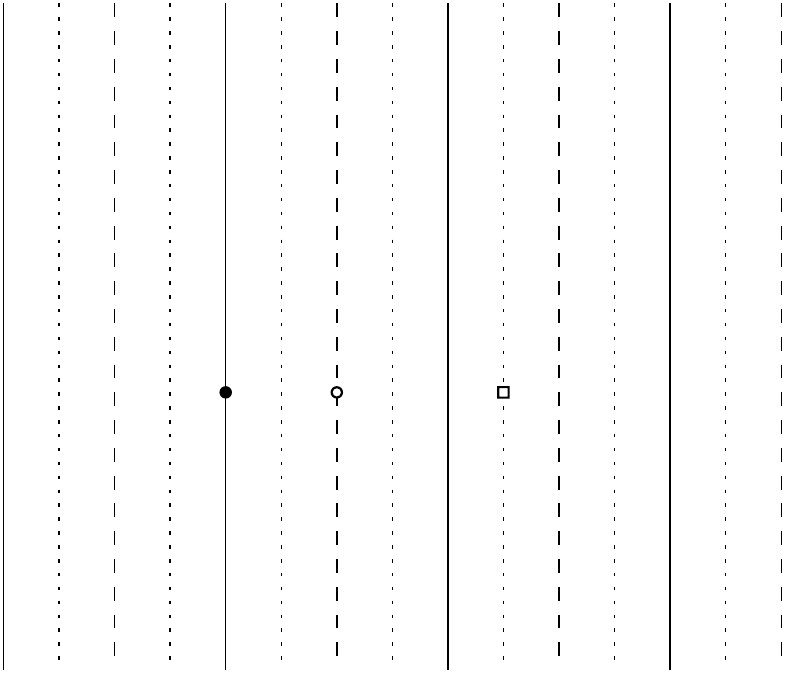}
    \caption{Solid lines represent $\primallattice$, dashed lines represent $\duallattice$ and dotted lines represent $\midedgelattice$.}
    \label{fig:lattices}
  \end{minipage}
\end{figure}

If we are given a discrete planar (isoradial) graph, it is easy to define its dual, medial and mid-edge graphs.
This will be the same for $\bbZ \times \bbR$, using the fact that it can be seen as the ``limit'' of a more and more flattened rectangular lattice $\bbZ \times \epsilon \bbZ$ as illustrated in Figure \ref{fig:rec_lattice}.
The faces are crushed together so that vertically, the notion of being neighbors gets degenerated.
We will only call neighbors two sites whoses $x$-coordinates differ by 1.
More details will be given below in a general setting.

Take $\delta > 0$.
We will define here the following notions related to the \emph{semi-discrete lattice with mesh size $\delta$}: primal, dual, medial and mid-edge lattices.
Formal definitions are given below and to visualize them, we refer to Figure \ref{fig:lattices}.

Let $\primallattice$ be the semi-discrete primal lattice $\delta \bbZ \times \bbR$.
We denote by $\duallattice$, the lattice $\delta (\bbZ + \frac{1}{2}) \times \bbR$ with the same notion of connectivity as $\delta \bbZ \times \bbR$, the \emph{dual} of $\primallattice$.
We can notice that the dual lattice is isomorphic to the primal one, by translation of $\frac{1}{2}$ in the $x$-coordinate.
As in the discrete setup, the dual lattice is given by the center of faces in the primal lattice, and edges by connecting two primal faces sharing a common edge.
Since the faces of $\primallattice$ are all crushed together vertically, the same happens to vertically-ordered dual vertices, giving us continuous lines isomorphic to $\bbR$.

The equivalent of the medial lattice from the discrete setup by taking the mid-points of neighboring primal-primal and dual-dual vertices gives us $\mediallattice = \primallattice \cup \duallattice$.
It is again isomorphic to the primal or dual lattice by scaling of factor $\frac{1}{2}$.

Finally, we take the middle of the two previous lattices to obtain the \emph{mid-edge lattice} $\delta( \frac{1}{2} \bbZ + \frac{1}{4}) \times \bbR$, denoted by $\midedgelattice$.
It is isomorphic to the medial lattice.

In the following graphical presentations, we will draw a filled black dot to represent a vertex on the primal lattice, a filled white dot a vertex on the dual lattice, and a filled white square when it is a mid-edge, or more precisely, the middle of the mid-edge.
See Figure \ref{fig:lattices}.

We need to define some more notions related to the semi-discrete lattice $\primallattice$, including segments, paths and domains.

A \emph{primal vertical segment} is denoted by $[\delta k + \icomp a, \delta k + \icomp b] := \{ \delta k\} \times [a, b]$, where $k \in \bbZ$ and $a < b$ are real numbers.
A \emph{primal horizontal segment} is denoted by $[\delta k + \icomp a, \delta l + \icomp a] := ([\delta k, \delta l] \times \{ a \}) \cap \primallattice = \{ \delta j + \icomp a, k \leq j \leq l \}$ where $k < l$ are integers and $a$ is a real number.
When a primal horizontal segment is of length $\delta$, we call it a \emph{primal segment}.

A sequence of points $(z_i)_{0 \leq i \leq n}$ on $\primallattice$ forms a \emph{path} if the consecutive points share the same $y$-coordinate (forming horizontal segments) or the same $x$-coordinate (forming vertical segments).

A \emph{primal domain} is a finite region delimited by primal horizontal and dual segments.
More precisely, it is given by a self-avoiding path consisting of $2n+1$ points $z_0, z_1, \dots, z_{2n}$ on $\primallattice$ such that
\begin{enumerate}
\item $[z_{2i}, z_{2i+1}]$ are horizontal segments for $i \in \llbracket0, n-1 \rrbracket$;
\item $[z_{2i+1}, z_{2i+2}]$ are vertical segments for $i \in \llbracket0, n-1 \rrbracket$;
\item these points form a closed path, \emph{i.e.} $z_0 = z_{2n}$.
\end{enumerate}

The set consisting of segments $\partial = \{ [z_{2i} z_{2i+1}], [z_{2i+1}, z_{2i+2}], i \in \llbracket 0, n-1 \rrbracket \}$ separates the plane into two connected \emph{open} components, a bounded one which is simply connected and an unbounded one.
The first one is called the \emph{domain} and is usually denoted by $\primaldomain$.
And $\partial$, or $\partial \Omega$, is called the \emph{boundary} of $\Omega$.
Except otherwise mentioned, the points $z_i$ are ordered counterclockwisely.

These same definitions apply to the dual lattice $\duallattice$ to get a \emph{dual domain}, usually denoted by $\dualdomain$, or to the medial lattice to get a \emph{medial domain}, $\medialdomain$.

The \emph{interior} of a primal domain $\primaldomain$, denoted by $\Int \primaldomain$, is the largest dual domain contained in $\primaldomain$.
It can also be seen as the set of dual vertices in $\primaldomain$ having both (primal) neighbors inside $\primaldomain$.
Similarly, the \emph{interior} of a dual domain $\dualdomain$ or a medial domain $\medialdomain$, denoted by $\Int \dualdomain$ or $\Int \medialdomain$, can also be defined in a similar way by replacing the word ``primal'' by ``dual'' or ``medial''.

Now we define a \emph{semi-discrete Dobrushin domain}, which is a medial domain with so-called \emph{Dobrushin boundary condition}.
Given $(a_w a_b)$ and $(b_b b_w)$ two horizontal edges, consider a primal path from $a_b$ to $b_b$ and a dual path from $b_w$ to $a_w$, such that the concatenation of both (first primal then dual) forms a counterclockwise boundary.
We denote $\partial_{ab}$ and $\partial^\star_{ba}$ the primal and dual parts.
See Figure \ref{fig:dob}.

\begin{figure}[htb] \centering
  \includegraphics[scale=0.75, page=1]{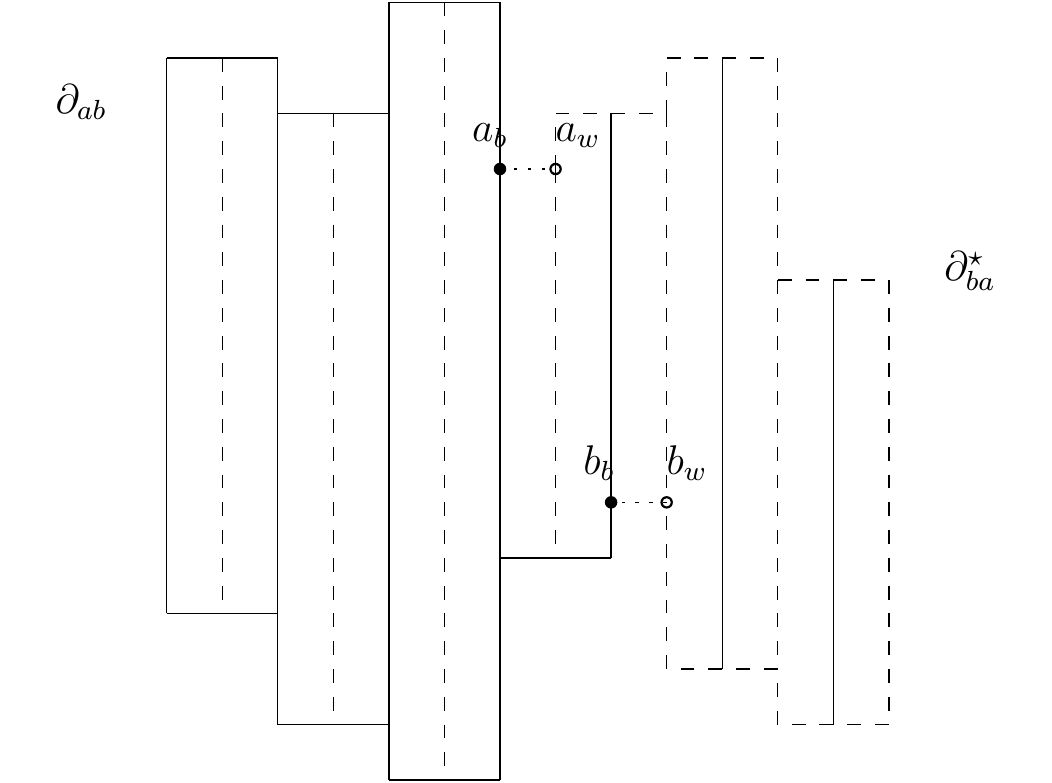}
  \caption{A Dobrushin domain.}
  \label{fig:dob}
\end{figure}

\subsection{Semi-discretization of a continuous domain}

A set in $\bbR^2$, or $\bbC^2$, is called a \emph{domain} if it is open, bounded and simply connected.
A \emph{Dobrushin domain} is a domain with two different marked points $a, b$ on the boundary.
It is often denoted by a triplet $(\Omega, a, b)$ where $\Omega$ is a domain and $a, b \in \partial \Omega$.

Here, we explain how to \emph{semi-discretize} such a domain to get its semi-discrete counterpart, on which the FK-representation of the quantum Ising model will be defined.
Consider $(\Omega, a, b)$ a Dobrushin domain in $\bbC$ or $\bbR^2$ and $\delta > 0$.
Let us denote by $[a_\delta^w a_\delta^b]$ and $[b_\delta^b b_\delta^w]$ two mid-edges with $a_\delta^b, b_\delta^b \in \primallattice$, $a_\delta^w, b_\delta^w \in \duallattice$ and mid-points $a_\delta^\flat$ and $b_\delta^\flat$ given by minimizing the distances between $a$ and $a_\delta^\flat$ and between $b$ and $b_\delta^\flat$ over all possible such mid-edge segments contained in $\Omega$.
Once we get these two distinguished edges $[a_\delta^w a_\delta^b]$ and $[b_\delta^b b_\delta^w]$, we complete the semi-discrete domain by making approximation with primal horizontal and vertical segments on the arc $\arcab$ then with dual horizontal and vertical segments on the arc $\arcba$.
This domain is denoted by $\Dobrushindomain$.
It lies in $\Int \Omega$.
See Figure \ref{fig:ex_dob} for an example.

\begin{figure} \centering
    \includegraphics[scale=0.8, page=1]{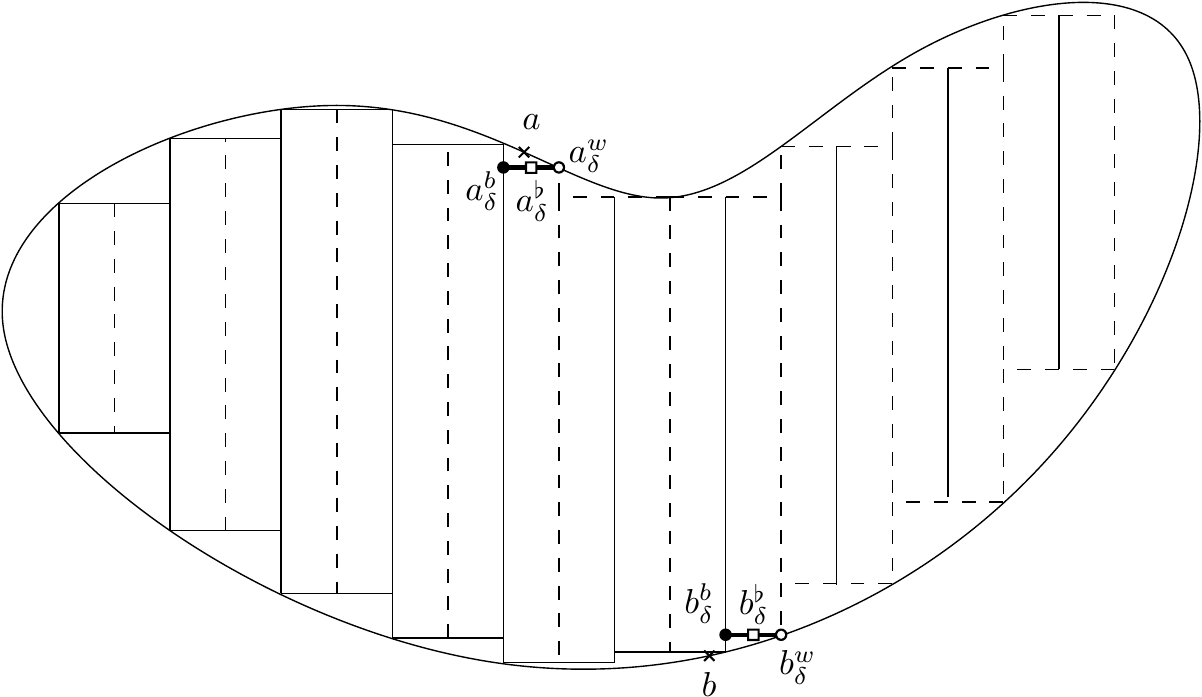}
    \caption{An example of approximation of a continuous Dobrushin domain by a semi-discrete one with mesh size $\delta$.}
    \label{fig:ex_dob}
\end{figure}

We write $\partial \medialdomain$ for the boundary of this Dobrushin domain.
This consists of four components:
\begin{align*}
[a_\delta^w a_\delta^b] & \quad \mbox{ an horizontal edge } \\
\arcab := \arcabfull & \quad \mbox{ the arc going from $a_\delta^b$ to $b_\delta^b$ } \\
[b_\delta^b a_\delta^w] & \quad \mbox{ an horizontal edge } \\
\arcba := \arcbafull & \quad \mbox{ the arc going from $b_\delta^w$ to $a_\delta^w$ }.
\end{align*}

They are ordered counterclockwise in Figure \ref{fig:ex_dob}.

\subsection{Quantum Ising model and loop representation} \label{sec:QI_loop}

On the semi-discrete medial Dobrushin domain $\Dobrushindomain$, one can define the \emph{continuum Bernoulli percolation with parameters $(\lambda, \mu)$} for $\lambda, \mu > 0$.

Consider two families of open intervals by taking the intersection between the interior of the domain $\Omega^\diamond$ and vertical primal or dual segments,
$$
\begin{array}{rcl} \primalint & := & \Int \medialdomain \cap \primallattice, \\
\dualint & := & \Int \medialdomain \cap \duallattice. \\
\end{array}
$$
We take two independent (one-dimensional) Poisson point processes of parameters $\lambda$ and $\mu$ on $\primalint$ and $\dualint$ respectively.
We denote by $(D, B)$ such a configuration, where $D$ contains the points in $\primalint$ and $B$ the points in $\dualint$.
The points in $D$ are called \emph{death points}.
They cut vertical lines into disjoint (primal) segments.
The points in $B$ are called \emph{bridges}.
They create horizontal connections between two neighboring vertical segments.
See Figure \ref{fig:qi_config} for an example.

\begin{figure}[htb]
  \begin{minipage}{\textwidth}
    \centering
    \includegraphics[scale=0.8, page=2]{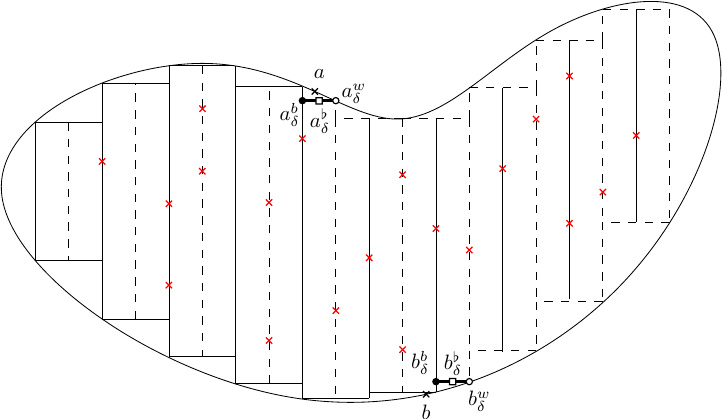}
    \caption{Example of a random configuration. Red crosses are points given by Poisson point processes.}
    \label{fig:qi_config}
  \end{minipage} \vfill
  \begin{minipage}{\textwidth}
    \centering
    \includegraphics[scale=0.8, page=3]{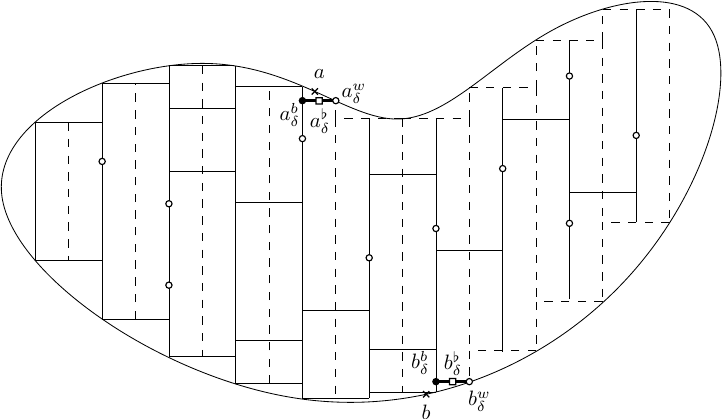}
    \caption{Representation with death points (red crosses on primal vertical lines) and bridges (red crosses on dual vertical lines) of the above configuration.}
    \label{fig:qi_presentation}
  \end{minipage}
\end{figure}

Given a configuration of continuum Bernoulli percolation $(D, B)$, we can define the notions of \emph{primal (or dual) connectivity} with respect to the points in these two sets.

Two points in the primal domain are said to have a \emph{primal connection} if there is a primal path going from one to another by taking primal vertical segments and horizontal bridges without crossing any death points.
The notion of having a \emph{dual connection} is similar by taking the dual graph, inversing primal and dual segments and death points and bridges.
More precisely, two points in the dual domain are said to have a \emph{dual connection} if there is a path going from one to another by taking dual lines and death points without crossing any bridges.

In the following, except otherwise mentioned, the connectivity always refers to the primal domain.
We call \emph{(primal) connected component of $v \in \primallattice$} a maximal collection (in terms of set) in $\Int \medialdomain$ containing connected (primal) segments and bridges.
It is sometimes called \emph{cluster} as well.

As in the discrete setup, this model has a \emph{loop representation} too.
A simple closed path living on the mid-edge lattice $\midedgelattice$ is called a \emph{loop}.
To get this representation, we just need to go through a simple operation: replace the pieces of our domain according to the rules explained in Figure \ref{fig:loops_rules}.

\begin{figure}[htb] \centering
    \includegraphics[scale=1]{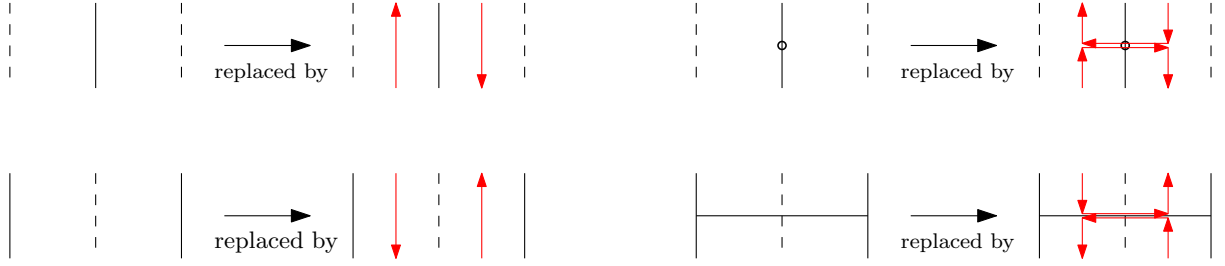}
    \caption{Transformation to get a loop representation from a continuum
Bernoulli percolation.}
    \label{fig:loops_rules}
\end{figure}

We notice that if we consider a Dobrushin domain, then we get a collection of loops surrounding either a primal or a dual connected component together with an interface connecting $a^\flat$ to $b^\flat$.
This is illustrated in Figure \ref{fig:loops_ex}.

\begin{figure}[htb] \centering
  \includegraphics[scale=0.8, page=4]{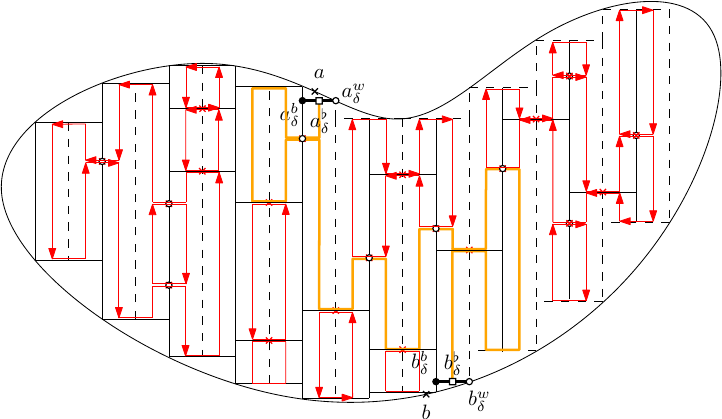}
  \caption{The loop representation corresponding to the configuration of Figure \ref{fig:qi_config}. Loops surrounding primal or dual clusters are in red
where as the interface connecting $a_\delta^\flat$ to $b_\delta^\flat$ is
in orange.}
  \label{fig:loops_ex}
\end{figure}

Let us denote by $\bbP_{\lambda, \mu}$ the law of the continuum Bernoulli percolation in the Dobrushin domain $\Dobrushindomain$.
We may also simply write $\bbP$ if the omitted parameters are clear.

Given $q \geq 1$, we will define the measure of the \emph{continuum FK-representation of parameter q}, denoted by $\dd \bbP_{q, \lambda, \mu}$ by weighing configurations according to the number of
connected components they have.
More precisely, we have the following proportionality between probability densities:
$$
\dd \bbP_{q, \lambda, \mu}(D, B) \propto q^{k(D, B)} \dd \bbP_{\lambda, \mu}(D, B)
$$
where $k(D, B)$ denotes the number of primal connected components, or clusters, in a given configuration $(D, B)$.
To make sense of this, we have to check that the partition function given below is finite:
\begin{equation*}
Z_{\lambda, \mu, q} = \int q^{k(D, B)} \dd \bbP_{\lambda, \mu}(D, B).
\end{equation*}

In fact, a bridge only decreases the number of clusters whereas a death point increases it by at most one.
So the number of clusters is bounded by the number of death points from above.
As a consequence,
$$
Z_{\lambda, \mu, q}
\leq \int q^{|D|} \dd \bbP_{\lambda, \mu}
= \sum_{n=0}^\infty q^n \frac{(\lambda C)^n}{n!} e^{-\lambda C}
= e^{(q-1) \lambda C}
$$
where $C$ is the one-dimensional Lebesgue measure of $\primaldomain$.
This quantity is finite (depending on $\delta$) since the domain $\primaldomain$ is bounded.

Thus, the measure $\bbP_{q, \lambda, \mu}$ is well-defined and it makes sense to work with the FK-representation.
In this article, we will only be interested in the case where $q=2$, a model that we now call the \emph{FK-representation of the quantum Ising model}.
Its measure $\bbP_{2, \lambda, \mu}$ can also be written as $\dd \bbP_{\lambda, \mu}^{\rm QI}$.
Its critical parameters are given by the relation $\mu / \lambda = 2$ \cite{Pfeuty-QI, BG-QI-sharp}.

We may also write the measure in terms of the loop representation introduced above.
To do so, we consider both primal and dual clusters, each of whom gives rise to a loop.
By a simple computation at criticality, we get the following relation:
\begin{equation} \label{eqn:measure_loop_rep}
\dd \bbP_{\lambda, \mu}^{\rm QI} (D, B) \propto \sqrt{2}^{l(D, B)} \dd \bbP_{\rho, \rho}(D, B)
\end{equation}
where $l(D, B)$ denotes the number of loops in a given configuration $(D, B)$ and $\rho = \sqrt{\lambda \mu}$.

\begin{killcontents}
\subsection{Relation to isoradial graphs}

Consider the usual model on a more and more flattened rectangular lattice and adjust the parameters according to ``flatness''.
More precisely, fix two parameters $\lambda$ and $\mu$ and consider the lattice $\bbZ \times \epsilon \bbZ$ for $\epsilon \leq \min(\frac{1}{\lambda}, \frac{1}{\mu})$ with vertical connection probability $1-\epsilon\lambda$ and horizontal
connection probability $\epsilon\mu$, all being independent.
By making $\epsilon$ go to 0, the probability measure converges weakly to the continuum Bernoulli percolation.
To get the random-cluster counterpart, one simply weighs clusters in the discrete setting before taking the weak limit.

\subsection{Quantum chain}

In the literature, the continuum model is sometimes represented by space-time evolutions.
Consider a spin configuration on the base graph $\bbZ$ and let the system evolve according to some quantum Hamiltonian.
This quantum Hamiltonian involves Pauli operators which generalize interactions between particules ($\sigma_x \sigma_y$ terms) and the presence of an external field ($h \sigma$ terms).
If we look at its time-evolution, the interactions between particles create bridges between neighboring sites and the external field creates death points on vertical lines, giving the geometric representation above.
Interested readers may check \cite{Ioffe-QI, GOS-entanglement, BG-QI-sharp} for more details.

A lot of techniques to studying discrete Ising model have their equivalents on quantum Ising model.
For example, the authors in \cite{BG-QI-sharp} show the sharpness of the phase transition of the quantum Ising model by using random-parity, which is an analogous notion of random-current used in \cite{ABF-Ising-sharp} to show that the phase transition of Ising model is sharp.
infrared bound \cite{Bjornberg-infrared, Bjornberg-infrared}
\end{killcontents}

\subsection{Main result}

\begin{killcontents}
It has been shown in \cite{Ising-SLE-convergence} that the interface of the Ising model on isoradial graphs converge in law to the so-called \emph{Schramm-Löwner Evolution} of parameter $3$ (spin-Ising) and $16/3$ (FK-Ising).

Smirnov first introduced the (para-)fermionic observable to show the conformal invariance of the classical Ising model on square lattice, see \cite{Smirnov-towards, Smirnov-conformal}.
Later on, Chelkak and Smirnov extended the theory of discrete complex analysis on isoradial graphs in \cite{CS-discreteca} which is the main tool to showing the universality of the Ising model on such graphs and the conformal invariance of the interface.
Details are given in \cite{CS-universality}.
To identify the limit, the property G2, a consequence of RSW property, is needed \cite{Kemp-SLE} to say that it is an SLE then to determine its parameter, we construct a martingale from the interface.
By the conformal invariance of the scaling limit of the martingale, we get the conditions which $W_t$ the driving function should satisfy, thus giving us the parameter of the SLE process.

Here, we proceed in a similar way for the quantum Ising model.
We establish the semi-discrete complex analysis in Section \ref{sec:SDCA} and show relations on the observables in Section \ref{sec:observables}.
Then, we use the convergence theorem in Section \ref{sec:CVG_thm} to show that the weak limit of interfaces is conformal invariant and identify the limiting object to be $\SLE_{16/3}$.
We establish the RSW property in Section \ref{sec:RSW_prop} (which is also a consequence of the forthcoming paper \cite{FK-universality}), giving us the condition G2 as described in \cite{Kemp-SLE}.
Then to get the parameter of the SLE process, it is exactly the same computation as in \cite{Ising-SLE-convergence}.
\end{killcontents}

We are ready to give a formal statement of the main theorem.

\begin{thm} \label{thm:main}
Let $(\Omega, a, b)$ a Dobrushin domain.
For $\delta > 0$, semi-discretize the domain to get a semi-discrete Dobrushin domain $\Dobrushindomain$ and consider the FK-representation of the critical quantum Ising model on it.
Consider the interface $\interface$ going from $a_\delta$ to $b_\delta$ in $\Dobrushindomain$.
The law of $\interface$ converges weakly to the chordal Schramm-Löwner Evolution $\SLE_{16/3}$ running from $a$ to $b$ in $\Omega$.
\end{thm}

%% file: sdca2.tex
\section{Semi-discrete complex analysis} \label{sec:SDCA}

In this section, we establish the main tool of our study: the semi-discrete complex analysis.
Most of the time, the notions in the classical case or isoradial case \cite{BMS, CS-discreteca, Bobenko-DCA-quad} can be generalized easily, such as derivatives, holomorphicity and harmonicity (Section \ref{sec:sdca_der}), Brownian motion (Section \ref{sec:BM}), Dirichlet boundary problem (Section \ref{sec:dirichlet}) and other related objects.
However, semi-discrete holomorphic functions require more care.

The integration on a primal or dual lattice of a semi-discrete function can also be defined similarly (Section \ref{sec:sdca_int}).
A minor difficulty could arise when it comes to integrating the product of two semi-discrete functions.
We will explain in Section \ref{sec:sdca_int_pair} how to define this so that to have expected properties as in the continuous setting.

We shall also construct the Green's function in the semi-discrete case (Section \ref{sec:Greens}).
This can be done by modifying the approach from \cite{Kenyon-laplacian}, which makes use of discrete exponential functions.

The notion of \emph{s-holomorphicity} will also be introduced in Section \ref{sec:s-hol}, which is quite important in the rest of the paper to define the observable (Section \ref{sec:observables}) and show its convergence (Section \ref{sec:CVG_thm}).

\subsection{Basic definitions} \label{sec:sdca_basics}

\subsubsection{Derivatives} \label{sec:sdca_der}

Let $\primaldomain$ be a primal semi-discrete domain.
A function $f$ defined on $\primaldomain$ is said to be \emph{continuous} if $y \mapsto f(x, y)$ is continuous for all $x \in \{ x, (x, y) \in \primaldomain \}$.
The same definition applies to functions defined on a dual domain $\dualdomain$, a medial domain $\medialdomain$ or a mid-edge domain $\midedgedomain$.
We then define in the same way a differentiable function on these domains, or even $\mathcal{C}^k$ functions, by demanding the property on all the vertical intervals.

Given a vertex $p \in \mediallattice$, we denote by $p^\pm$ the right and left neighboring vertices in $\mediallattice$, \emph{i.e.} $p^\pm := p \pm \frac{\delta}{2}$ in an algebraic way.
Similarly, we may write $e^\pm := e \pm \frac{\delta}{2}$ for neighboring mid-edge vertices of $e \in \midedgelattice$.

If $p \in \mediallattice$, we denote by $e_p^\pm$ the right and the left neighboring mid-edges, or $e_p^\pm := p \pm \frac{\delta}{4}$.
In the same way, given a mid-edge $e \in \midedgelattice$, we denote by $p_e^\pm$ the right and the left neighboring medial vertices, or $p_e^\pm := e \pm \frac{\delta}{4}$.

\begin{figure}[htb]
  \centering
    \includegraphics[scale=1]{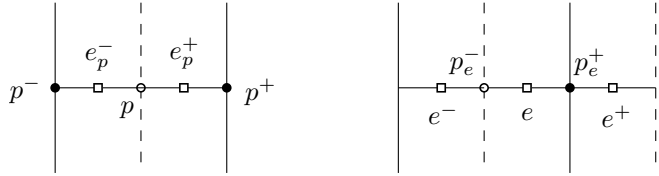}
    \caption{Illustration of the definitions above. \emph{Left}: neighboring mid-edges and vertices of a vertex. \emph{Right}: neighboring vertices and mid-edges of a mid-edge.}
    \label{fig:medial_midedge}
\end{figure}

For $e \in \midedgelattice$, we may also write $u_e$ (\emph{resp.} $w_e$) the neighboring primal (\emph{resp.} dual) vertex.
In other words, 
\begin{align*}
\{ u_e \} & = \{ p_e^+, p_e^- \} \cap \primallattice, \\
\{ w_e \} & = \{ p_e^+, p_e^- \} \cap \duallattice.
\end{align*}
See Figure \ref{fig:neighboring} for an illustration.

\begin{figure}[htb]
  \centering
    \includegraphics[scale=1]{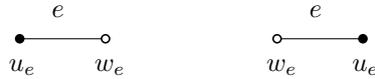}
    \caption{Illustration of the neighboring primal and dual vertices.}
    \label{fig:neighboring}
\end{figure}

The derivative $\derz$ or the anti-derivative $\derzbar$ of a differentiable function on $\mediallattice$ or $\midedgelattice$ can be defined by taking the ``semi-discrete counterpart''.
Thus, the notion of holomorphicity and harmonicity will be defined in the same way.

\begin{defn}
Let $f : \mediallattice \rightarrow \bbC$ be a complex function defined on the medial lattice.
Let $p \in \mediallattice$.
The \emph{$x$-derivative} at $p$ is given by
$$
\Deltax f(p) := \frac{f(p^+) - f(p^-)}{\delta}.
$$
If $p \in \mediallattice$, we define the \emph{second $x$-derivative} at $p$ by
$$
\Deltaxx f(p) := \Deltax \circ \Deltax f(p) = \frac{f(p^{++}) + f(p^{--}) - 2f(p)}{\delta^2}
$$
where $p^{++} = (p^+)^+$ and $p^{--} = (p^-)^-$.
\end{defn}

\begin{defn}
Let $f : \mediallattice \rightarrow \bbC$ be a differentiable complex function defined on a medial lattice $\mediallattice$.
Its \emph{derivative} and \emph{anti-derivative} at $p \in \bbL^\diamond$ are given by
\begin{align}
\derz f (p) & = \frac{1}{2} \left[ \Deltax f(p) + \frac{\dy f(p)}{\icomp} \right]
 = \frac{1}{2} \left[ \frac{f(p^+) - f(p^-)}{\delta} + \frac{\dy f(p)}{\icomp} \right], \label{eqn:der}\\
\derzbar f (p) &  = \frac{1}{2} \left[ \Deltax f(p) - \frac{\dy f(p)}{\icomp} \right]
= \frac{1}{2} \left[ \frac{f(p^+) - f(p^-)}{\delta} - \frac{\dy f(p)}{\icomp} \right].
\label{eqn:antider}
\end{align}

A semi-discrete function $f$ is said to be \emph{holomorphic at} $p \in \mediallattice$ if $\derzbar f(p) = 0$, and is said to be \emph{holomorphic in $\medialdomain$}, where $\medialdomain$ is a medial domain, if $\derzbar f(p) = 0$ for all $p \in \Int \medialdomain$.
\end{defn}

The same definitions \eqref{eqn:der} and \eqref{eqn:antider} could also be used for functions defined on mid-edge domains.
Let $f : \midedgelattice \rightarrow \bbC$ be a differentiable complex function defined on the mid-edge lattice, then
\begin{align}
\derz f (e) & = \frac{1}{2} \left[ \frac{f(e^+) - f(e^-)}{\delta}  + \frac{\dy f(e)}{\icomp} \right], \label{eqn:der_midedge}\\
\derzbar f (e) & = \frac{1}{2} \left[ \frac{f(e^+) - f(e^-)}{\delta} - \frac{\dy f(e)}{\icomp} \right].
\label{eqn:antider_midedge}
\end{align}

Consider a twice differentiable function $f$ defined on a primal (or dual) lattice.
We can define its \emph{Laplacian} by
\begin{equation} \label{eqn:Laplacian}
\laplacian f := 4 \derz \derzbar f = 4 \derzbar \derz f = \Deltaxx f + \ddy f.
\end{equation}
A twice differentiable function $f$ defined on the primal domain $\primaldomain$ is said to be 
\begin{itemize}
\item \emph{harmonic} if $\laplacian f(p) = 0$, for all $p \in \Int^2 \primaldomain$;
\item \emph{subharmonic} if $\laplacian f(p) \geq 0$, for all $p \in \Int^2 \primaldomain$;
\item \emph{superharmonic} if $\laplacian f(p) \leq 0$, for all $p \in \Int^2 \primaldomain$.
\end{itemize}
Here, we take twice $\Int$ because in the definition of the Laplacian, the second $x$-derivative is involved.
We extend these definitions to a twice differentiable function defined on a dual domain in the same way.

\subsubsection{Integration on primal and dual lattices} \label{sec:sdca_int}

We now define the notion of integral for a semi-discrete function $f$ living on $\mediallattice$.
Let $\calP = [k \delta + \icomp a, k \delta + \icomp b]$ be a vertical primal (\emph{resp.} dual) segment, meaning that $k \in \bbZ$ (\emph{resp.} $k \in \bbZ + \frac{1}{2}$) and $a < b$.
If the segment $\calP$ is oriented upwards, we write 
\begin{equation} \label{eqn:int_vert}
\int_\calP f(z) \dd z := \icomp \int_a^b f(\delta k + \icomp y) \dd y = \icomp \int_a^b f_k(y) \dd y
\end{equation}
to be the integral along the vertical segment $\calP$, where we define $f_k(\cdot) = f(\delta k + \icomp \cdot)$.
Both primal and dual vertical segments are called \emph{medial vertical segments}.

Let $\calP = \{ \delta k + \icomp t, m \leq k \leq n, k \in \bbZ \}$ a horizontal primal segment for $m, n \in \bbZ$.
We define
\begin{equation} \label{eqn:int_hor_primal}
\int_\calP f(z) \dd z := \delta \sum_{k=m}^{n-1} f \left( \delta \left( k+ \frac12 \right) + \icomp t \right)
= \delta \sum_{k=m}^{n-1} f_{k+\frac{1}{2}} (t)
\end{equation}
the integral along the horizontal primal segment $\calP$ oriented to the right.
If we have a horizontal dual segment $\calP = \{ \delta(k + \frac12) + \icomp t, m \leq k \leq n, k \in \bbZ \}$, we define in a similar way
\begin{equation} \label{eqn:int_hor_dual}
\int_\calP f(z) \dd z := \delta \sum_{k=m+1}^{n} f(\delta k + \icomp t)
= \delta \sum_{k=m+1}^{n} f_{k} (t)
\end{equation}
the integral along the horizontal dual segment.
Both primal and dual horizontal segments are called \emph{medial horizontal segment}.

In both vertical and horizontal cases, we define the integral of a reversed path to be the opposite of the integral of the original path.

The integrals above should be seen as ``semi-discrete complex'' integrals.
More precisely, we take into account the direction in which the segment goes, giving factors $\pm 1$, $\pm \icomp$.
We may also define integration against $|\dd z|$, removing thus these factors.
Moreover, our integrals shall be defined additively.

Given a semi-discrete primal domain $\primaldomain$ and a semi-discrete function $f$ on it, we may write its integral on the domain by
$$
\int_{\primaldomain} f(y) \dd y := \delta \sum_{\Int \primaldomain} \int_{\alpha_k}^{\beta_k} f_k(y) \dd y
$$
where the sum is taken over all the primal vertical axes in $\Int \primaldomain$.
Here, the integration against $\dd y$ should be seen as classical real integration.

An \emph{elementary primal (\emph{resp.} dual) domain} is given by
$$
B_k(\alpha, \beta) = \{ x + \icomp y, \delta k \leq x \leq \delta(k+1), \alpha \leq y \leq \beta \}
$$
for $\alpha < \beta$ integers and $k \in \bbZ$ (\emph{resp.} $k \in \bbZ + \frac{1}{2}$).
Its boundary consists of four primal (\emph{resp.} dual) segments, two vertical and two horizontal.
By convention, we will always orient the boundary counterclockwise.
An \emph{elementary medial domain} is either a primal or a dual elementary domain.
In the rest of the article, we denote it by $B^\lozenge_k(\alpha, \beta)$.

From integrals along horizontal and vertical segments given by Equations \eqref{eqn:int_vert}, \eqref{eqn:int_hor_primal} and \eqref{eqn:int_hor_dual}, we can define integrals along the boundary of any primal or dual domain.
Let us consider an elementary primal or dual domain $B_k(\alpha, \beta)$ as an example.
Denote by $\calC$ its boundary oriented counterclockwise.
If $f$ is a semi-discrete function which is piece-wise continuous, then its integral along the contour $\calC$ is given by
\begin{equation} \label{eqn:int_domain}
\oint_\calC f(z) \dd z :=
\delta \left[ f_{k+\frac{1}{2}}(\alpha) - f_{k+\frac{1}{2}}(\beta) \right]
+ \icomp \int_\alpha^\beta [ f_{k+1}(y) - f_k(y) ] \dd y,
\end{equation}
consisting of the four integrals coming from the four sides of the elementary domain.
Moreover, if $f$ is piece-wise differentiable, this can be rewritten as
%
where in $\derzbar f$, the term $\dy f$ is given in the sense of distributions.

Given a primal semi-discrete domain $\primaldomain$, we define the integral of a semi-discrete function $f$ along its counterclockwise-oriented contour by decomposing its boundary into vertical and horizontal segments and adding them up.

The integral along the contour of a dual domain is defined in a similar way and the corresponding properties can be obtained as well.
We notice that here we do not define the integral along the contour of a medial domain.
We will see later an alternative definition for this.

\begin{prop}[Green's formula] \label{prop:greens_formula}
Consider a primal domain $\primaldomain$.
Denote by $\calC$ its boundary oriented counterclockwise.
Given a semi-discrete function $f$ which is piece-wise differentiable, its integral along the contour $\calC$ satisfies the following relation
\begin{align*}
\oint_\calC f(z) \dd z
=  2 \icomp \int_{\primaldomain} \derzbar f_{k+\frac12}(y) \dd y.
\end{align*}
\end{prop}

\begin{proof}
We decompose the primal or dual domain into elementary ones and sum up the Equation \eqref{eqn:int_domain} corresponding to each of them.
In the right hand side, the integrals on boundary parts of elementary domains that appear twice in the sum cancel out, and the only terms that remain sum up to the integral along the contour $\calC$.
\end{proof}

Define a vector operator $\nabla = (\Deltax, \dy)$ called \emph{nabla}.
The following statement is similar to Green's formula.

\begin{prop}[Divergence theorem] \label{prop:div_thm}
Let $\primaldomain$ be a primal domain with contour $\calC$ which is oriented counterclockwise.
Let $\overrightarrow{F} = (F_x, F_y)$ be a semi-discrete function which is continuous and takes values in $\bbR^2$.
We have the following equality,
\begin{equation} \label{eqn:divergence_thm}
\int_{\primaldomain} \nabla \cdot \overrightarrow{F} (y) \dd y
= \oint_\calC \overrightarrow{F}(z) \cdot \overrightarrow{n}(z) |\dd z|,
\end{equation}
where $\overrightarrow{n}(z)$ is the vector obtained by a rotation of $-\frac\pi2$ from the tangent vector to $\calC$ at $z$ with norm 1.
\end{prop}

\begin{proof}
As usual, it suffices to show this for an elementary domain, and then sum up over a decomposition of $\primaldomain$ into elementary domains.
Let $B_k(\alpha, \beta)$ be an elementary domain.
Write $F_{k, x}(\cdot) = F_x (\delta k + \icomp \cdot)$ and $F_{k, y} = F_y (\delta k + \icomp \cdot)$.
The left-hand side of Equation \eqref{eqn:divergence_thm} can be rewritten as
\begin{align*}
& \delta \int_\alpha^\beta (\Deltax F_{k, x} + \dy F_{k, y}) (y) \dd y \\
& \qquad = \int_\alpha^\beta (F_{k+\frac12, x} - F_{k-\frac12, x}) (y) \dd y
+ \delta [ F_{k, y} (\beta) - F_{k, y} (\alpha) ]
\end{align*}
which is exactly the right-hand side of Equation \eqref{eqn:divergence_thm}.
\end{proof}

We notice again that this proposition is still valid even if $\overrightarrow{F}$ is only continuous by pieces and differentiable by pieces, as long as we interprete derivatives in the sense of distributions.

\subsubsection{Integration of a pair of functions} \label{sec:sdca_int_pair}

Here, we define the integration of a pair of functions and establish the equivalent of Green's theorem in the semi-discrete case.

Let us start again with integration on elementary segments.
Consider two functions defined on the semi-discrete lattice $f$ and $g$, a vertical primal (\emph{resp.} dual) segment $\calP = [k \delta + \icomp a, k \delta + \icomp b]$ with $k \in \bbZ$ (\emph{resp.} $k \in \bbZ + \frac{1}{2}$) and $a < b$.
Recall that $f_k (\cdot) = f(k \delta + \icomp \cdot)$.
If the segment $\calP$ is oriented upwards, we write 
\begin{equation} \label{eqn:int_pari_vert_up}
\int_\calP [f; g] \dd z := \int_a^b [ f_{k-\frac12}(y) g_{k+\frac12}(y) - f_{k+\frac12}(y) g_{k-\frac12}(y) ] \dd y
\end{equation}
to be the integral along $\calP$.
If the segment is oriented downwards, we take the opposite of the above quantity.

For $m, n \in \bbZ$, let $\calP = \{ k \delta + \icomp t, m \leq k \leq n, k \in \bbZ \}$ be a horizontal primal segment.
We define
\begin{equation} \label{eqn:int_pari_hor_primal}
\int_\calP [f; g] \dd z := \delta^2 \sum_{k=m}^{n-1} (g_{k+\frac12} \dy f_{k+\frac12} - f_{k+\frac12} \dy g_{k+\frac12})(t)
\end{equation}
the integral along the horizontal primal segment $\calP$, oriented towards the right.

In the same way as integration of a semi-discrete function along the (counter\-clockwise-oriented) contour of a semi-discrete domain, we define the counterpart of a pair of functions by decomposing the contour into vertical and horizontal segments and adding them up.

\begin{prop}[Green's theorem] \label{prop:greens_thm}
Consider a primal (or dual) domain $\primaldomain$.
Denote by $\calC$ its counterclockwise-oriented boundary.
Given two semi-discrete functions $f$ and $g$ which are continuous by pieces and differentiable by pieces in $\primaldomain'$ such that $\primaldomain \subset \Int \primaldomain'$, we have
$$
\oint_\calC [f; g] \dd z 
= \delta \int_{\primaldomain} [f_k \laplacian g_k - g_k \laplacian f_k ] (y) \dd y.
$$
\end{prop}

\begin{proof}
As usual, we start by showing this for an elementary domain since we can superpose these domains to obtain more general domains and the integration terms simplify.
Consider $B_{k-\frac12}(\alpha, \beta)$ a dual elementary domain and denote $\calC$ its contour with counterclockwise orientation.
By definition, we have
\begin{align*}
\oint_{\calC} [f; g] \dd z & = \int_\alpha^\beta [f_k g_{k+1} - f_{k+1} g_k](y) \dd y
- \int_\alpha^\beta [f_{k-1} g_k - f_k g_{k-1}](y) \dd y \\
& \qquad + \delta^2 \left[ (g_k \dy f_k - f_k \dy g_k)(t) \right]_\beta^\alpha \\
& = \int_\alpha^\beta [ f_k( \delta^2 \Deltaxx g_k + 2 g_k) - g_k (\delta^2 \Deltaxx f_k + 2 f_k) ] (y)\dd y \\
& \qquad + \delta^2 \int_\alpha^\beta \dy (f_k \dy g_k - g_k \dy f_k)(t) \dd t   \\
& = \delta^2 \int_\alpha^\beta [ f_k \Deltaxx g_k - g_k \Deltaxx f_k ] (y)\dd y 
+ \delta^2 \int_\alpha^\beta (f_k \ddy g_k - g_k \ddy f_k)(t) \dd t   \\
& = \delta^2 \int_\alpha^\beta (f_k \laplacian g_k - g_k \laplacian f_k )(y) \dd y.
\end{align*}
\end{proof}

\begin{killcontents}
\subsubsection{Integration on medial lattice} \comment{à réécrire}

When it comes to integration along the boundary of a medial domain, we always work with a pair of functions.
Given any medial domain $\medialdomain$, first define its outer boundary $\calC^\lozenge$, consisting of two parts: the union of boundaries of $\medialdomain$ translated by $\frac14 \delta$ and $-\frac14 \delta$.
Both of them are oriented counterclockwise as always 

Let $f$ and $g$ be two semi-discrete functions which are continuous and $B_k^\lozenge(\alpha, \beta)$ be a medial domain with contour $\calC$.
We define the integral along the contour to be
\begin{align*}
\oint_\calP [f; g] \dd z & = \delta \left[ f_k g_k \right]^{\icomp \alpha}_{\icomp \beta} 
 + \icomp \int_\alpha^\beta (f_{k} g_{k+\frac12} + f_{k+\frac12} g_k - f_{k-\frac12} g_k - f_k g_{k-\frac12})(\icomp y) \dd y.
\end{align*}

We are now ready to state Green's theorem for a pair of semi-discrete functions.

\begin{thm}[Green's theorem]
Let $\medialdomain$ be a medial domain and denote by $\calC$ its boundary.
If $f$ and $g$ are semi-discrete functions on $\medialdomain$, then
$$
\oint_\calC [f; g] \dd z =
2 \icomp \int_{\medialdomain} [f \derzbar g + g \derzbar f] (\icomp y) \dd y
$$
\end{thm}

\begin{proof}
We show the formula for an elementary domain.
For a more general medial domain, we just need to juxtapose elementary domains and sum up equalities.

Let $B_k^\lozenge(\alpha, \beta)$ be an elementary domain.
\end{proof}

We notice that by taking $g \equiv 1$, we recover Green's formula from Proposition \ref{prop:greens_formula}.
\end{killcontents}

\begin{killcontents}

\subsection{Cauchy's Kernel}

We define $\Cauchy(z, \zeta)$ a singularity function with the following property.
\begin{enumerate}
\item The function $\Cauchy$ is translational invariant. (Thus, $\Cauchy(z, \zeta) = \Cauchy(0, \zeta-z) = \Cauchy(\zeta-z)$.)
\item At fixed $\zeta$, the function $z \mapsto \Cauchy(z, \zeta)$ is continuous except at $\zeta$. Around $z = \zeta$, it satisfies
$$
\lim_{\epsilon \rightarrow 0} [\Cauchy(\zeta + \icomp \epsilon, \zeta) - \Cauchy(\zeta - \icomp \epsilon, \zeta)] = 1.
$$
\item At fixed $\zeta$, the function $z \mapsto K(z, \zeta)$ is semi-discrete holomorphic except at $\zeta$.
\item If $z = x + \icomp y$ with $x \in \frac12\bbZ, y \in \bbR$, we have the following asymptotic behaviors:
\begin{align*}
\Cauchy(z, \zeta) = \calO (|x|^{-1}) & \mbox{ as } x \rightarrow \infty \\
\Cauchy(z, \zeta) = \calO (|y|^{-1}) & \mbox{ as } y \rightarrow \infty.
\end{align*}
\end{enumerate}

The explicit formula of $K$ is given in \cite{Kurowski-sda}.

\begin{prop}[Cauchy kernel]
If $z = t + \icomp \frac{m}{2}$, the value $\Cauchy(z, 0)$ is given by
\begin{align*}
\Cauchy(z, 0) := & - \frac{1}{2 \pi} \int_0^\pi e^{-2|t| \sin u} \sin(mu) du \\
& \quad +  \sgn(t) \cdot \frac{i}{2 \pi} \int_0^\pi e^{-2|t| \sin u} \cos(mu) du.
\end{align*}
\end{prop}

\begin{prop}[Cauchy's formula]

\end{prop}

\end{killcontents}

\subsection{Brownian motion, harmonic measure and Laplacian} \label{sec:BM}

Let $\delta > 0$.
The standard Brownian motion on the semi-discrete lattice $\primallattice = \delta \bbZ \times \bbR$ can be seen as a continuous-time random walk in the horizontal direction; and a standard Brownian motion on $\bbR$ in the vertical direction.
We give a more precise description below.

\begin{defn}
Let $(T_i)_{i \in \bbN}$ be a family of \emph{i.i.d.} exponential random variables of rate 1 and $(D_i)_{i \in \bbN}$ be a family of \emph{i.i.d.} uniform random variables taking value in $\{ +1, -1 \}$.
We define
$$
S_t = \sum_{i=1}^{N(t)} D_i
$$
where
$$
N(t) = \sup \{ n \in \bbN, T_1 + \dots + T_n \leq t \}.
$$
The continuous-time process $(S_{t})_{t \in \bbR}$ is the \emph{standard continuous-time simple random walk} on $\bbZ$.
\end{defn}

\begin{rmk}
We can easily compute the expectation and variance of $S_t$ which are respectively 0 and $t$.
It also has a good scaling property and one can show that $(\delta S_{t / \delta^2})$ converges to $(B_t)$ in law when $\delta$ goes to 0, where $(B_t)$ is a standard one-dimensional Brownian motion.
Here, the process $(\delta S_{t / \delta^2})$ can be seen as the continuous-time random walk of parameter $\frac{1}{\delta^2}$ with symmetric jumps $\pm \delta$.
\end{rmk}

We can now define the \emph{semi-discrete standard Brownian motion} on $\primallattice$.

\begin{defn} \label{def:sds_BM}
A \emph{semi-discrete standard Brownian motion} on $\primallattice$ is given by
$$
(\BM{t} = (X_t, Y_t) = (\delta S_{t / \delta^2}, B_t))_{t \geq 0}
$$
where $(S_t)$ is a standard one-dimensional continuous-time simple random walk and $(B_t)$ is a standard one-dimensional Brownian motion, both of them being independent of each other.
The starting point $\BM{0}$ is arbitrary, which is given by the starting points of $(S_t)$ and $(B_t)$.
\end{defn}

As in the discrete and continuous cases, we can define the notion of \emph{harmonic measure} via the standard Brownian motion.

\begin{defn}
Given a primal domain $\primaldomain$ and $(\BM{t})$ a Brownian motion on $\primaldomain$ starting at some point $(x, y) \in \primaldomain$, we define
\begin{equation} \label{eqn:stopping_time}
\recstoppingtime = \inf \{ t \geq 0, \BM{t} \notin \Int \primaldomain \}
\end{equation}
The \emph{harmonic measure} of $\primaldomain$ with respect to $(x, y)$, denoted by $\dd \HM((x, y), \cdot)$, is the law of $\BM{\recstoppingtime}$.
\end{defn}

Here, we are interested in the harmonic measure on centered elementary rectangular domains $\recdomain = \{ -\delta, 0, \delta\} \times [-\epsilon, \epsilon]$.
On such domains, the harmonic measure with respect to $0$, denoted by $\hm$, is the sum of two Dirac masses at $\pm i \epsilon$ and two density measures which are symmetric in both discrete and continuous directions on $\{ \pm \delta \} \times [-\epsilon, \epsilon]$.

We will write $\exitingproba$ for the probability that the Brownian motion $\BM{t}$ leaves $\recdomain$ (the first time) from its left or right sides.
This can be expressed by using the harmonic measure on $\recdomain$ as follows,
\begin{equation}
\exitingproba = \int_{-\epsilon}^\epsilon \hm(- \delta, y) \dd y + \int_{-\epsilon}^\epsilon \hm(\delta, y) \dd y.
\end{equation}
Thus, we can write the Dirac masses at $\pm i \epsilon$ in this way:
$$
\hm (\pm i \epsilon) = \hm(0, \pm \epsilon) = \frac{1 - \exitingproba}{2}
\cdot \delta(\cdot).
$$

\begin{defn}
Given a primal domain $\primaldomain$, a function $f : \primaldomain \to \bbR$ is said to satisfy the \emph{mean-value property} on rectangles if for all $(x, y) \in \primaldomain$ and $\epsilon > 0$ such that $(x, y) + \recdomain \subset \primaldomain$, we have
\begin{equation} \label{eqn:mean_value}
f(x, y) = \bbE_{(x, y)} \left[ f \left( \BM{ \recstoppingtime } \right) \right].
\end{equation}
Here, $\BM{t}$ is the standard Brownian motion starting at $(x, y)$ and $\recstoppingtime$ the stopping time defined in Equation \eqref{eqn:stopping_time}.
\end{defn}

\begin{rmk}
In terms of harmonic measure, Equation \eqref{eqn:mean_value} can be reformulated as (without loss of generality, we take $(x, y) = (0, 0)$)
\begin{align*}
f(0, 0) & = \int_{\boundaryrecdomain} f(z) \rho_\epsilon(z) |\dd z| \\
& = \int_{-\epsilon}^\epsilon f (- \delta, y) \hm (- \delta, y) \dd y 
+ \int_{-\epsilon}^\epsilon f (\delta, y) \hm (\delta, y) \dd y \\
& \qquad + \frac{1-\exitingproba}{2} \cdot ( f(-i \epsilon) + f(i \epsilon) ).
\end{align*}
\end{rmk}

\begin{prop}
The probability that the Brownian motion $\BM{t}$ leaves $\recdomain$ (the first time) from the left or right boundary is
$$
\exitingproba
= \frac{\cosh( \sqrt{2} \epsilon / \delta) - 1}{\cosh( \sqrt{2} \epsilon / \delta)}
= \left( \frac\epsilon\delta \right)^2 + \calO_\delta \left(  \epsilon^4 \right)
$$
where the asymptotics is given for $\epsilon \rightarrow 0$.
\end{prop}

\begin{proof}
Fix $(\BM{t} = (\delta S_{t/\delta}, B_t))_{t \geq 0}$ as in Definition \ref{def:sds_BM}.
This probability is exactly $\bbP[\delta^2 T_1 < \tau]$ where $T_1$ is an exponential law of parameter 1, which is independent of the stopping time $\tau = \tau_\epsilon \wedge \tau_{-\epsilon}$ for the standard 1D Brownian motion.
Here
$$
\tau_x := 
\left \{
\begin{array}{ll}
\inf \{ t, B_t \geq x \}, & \quad \mbox{if } x > 0, \\
\inf \{ t, B_t \leq x \}, & \quad \mbox{if } x < 0.
\end{array}
\right.
$$
By Fubini, we have
\begin{align*} 
1 - \exitingproba & = \bbP[\delta^2 T_1 > \tau]
= \bbE [ \bbP[ T_1 > \tau / \delta^2 \mid \tau ] ]
= \bbE[\exp(- \tau / \delta^2)],
\end{align*}
which is the Laplace transform of $\tau$.

To calculate this, we notice that the continuous-time process 
$$
M_t = \exp \left( \sqrt{2} B_t/\delta - t/\delta^2 \right)
$$
is a martingale with respect to the canonical filtration.
Moreover, for the stopping time $\tau$, the process $(M_{t \wedge \tau})_t$ is a martingale bounded by $e^{\sqrt{2} \epsilon / \delta}$.
The stopping time being finite almost surely, we can apply Doob's optional stopping theorem, giving us:
\begin{align} \label{eqn:martingale}
\begin{split}
1 & = \bbE \bracks{ M_0 } = \bbE \bracks{ M_\tau } \\
& = \frac{1}{2} \bbE \bracks{ M_\tau | \tau = \tau_\epsilon } + \frac{1}{2} \bbE \bracks{ M_\tau | \tau = \tau_{-\epsilon} } \\
& = \frac{1}{2} \exp \pars{ \sqrt{2} \epsilon / \delta } \bbE \bracks{ \exp(-\tau / \delta^2) | \tau = \tau_\epsilon } \\
& \qquad + \frac{1}{2} \exp \pars{- \sqrt{2} \epsilon / \delta } \bbE \bracks{ \exp(-\tau / \delta^2) | \tau = \tau_{-\epsilon} }.
\end{split}
\end{align}

Since $(B_t)_t$ and $(-B_t)_t$ are equal in law, we have
$$
\bbE \bracks{ \exp(-\tau / \delta^2) | \tau = \tau_\epsilon } = \bbE \bracks{ \exp(-\tau / \delta^2) | \tau = \tau_{-\epsilon} }.
$$
Moreover,
$$
\bbE \bracks{ \exp(-\tau / \delta^2) }
= \frac{1}{2} \bbE \bracks{ \exp(-\tau / \delta^2) | \tau = \tau_\epsilon }
+ \frac{1}{2} \bbE \bracks{ \exp(-\tau / \delta^2) | \tau = \tau_{-\epsilon} },
$$
giving
$$
\bbE \bracks{ \exp(-\tau / \delta^2) }
= \bbE \bracks{ \exp(-\tau / \delta^2) | \tau = \tau_\epsilon }
= \bbE \bracks{ \exp(-\tau / \delta^2) | \tau = \tau_{-\epsilon} }.
$$
Thus, Equation \eqref{eqn:martingale} becomes
$$
1 = \cosh \pars{ \sqrt{2} \epsilon / \delta } \cdot \bbE \bracks{ \exp(-\tau / \delta^2) },
$$
which implies
$$
\bbE \bracks{ \exp(-\tau / \delta^2) } = \frac{1}{\cosh \pars{ \sqrt{2} \epsilon / \delta}}
$$
and
$$
\exitingproba = 1 - \bbE \bracks{\exp(-\tau / \delta^2)}
= \frac{ \cosh \pars{\sqrt{2} \epsilon / \delta} - 1}{ \cosh \pars{ \sqrt{2} \epsilon / \delta } }.
$$
\end{proof}

\begin{prop} \label{prop:equiv_laplacian}
Let $\primaldomain$ be a primal domain and $h : \primaldomain \to \bbR$ be a $\calC^2$ function defined on it.
Then the following two statements are equivalent :
\begin{enumerate}
\item $h$ satisfies the mean-value property on elementary rectangles,
\item $\laplacian h \equiv 0$ on $\primaldomain$.
\end{enumerate}
\end{prop}


\begin{proof}
Here, we will show that the point 1 implies the point 2.
The converse will be discussed later in Section \ref{sec:dirichlet}.

Consider a function $f$ as in the statement.
We will apply the mean-value property at a point of $\primaldomain$ and consider smaller and smaller elementary rectangles to prove the desired property.
Let $\epsilon > 0$ and consider an elementary rectangle $R_\epsilon$.
Let us first approximate the contribution of $\bbE_0 \bracks{ f \left( \BM{T} \right) }$ on the left boundary by $h(-\delta, 0)$:
\begin{align*}
& \int_{-\epsilon}^\epsilon h(-\delta, y) \hm(-\delta, y) dy - \frac{\exitingproba}{2} \cdot h(-\delta, 0) \\
& \quad = \int_{-\epsilon}^\epsilon [ h(- \delta, y) - h(- \delta, 0) ] \rho_\epsilon(- \delta, y) \dd y \\
& \quad = \int_{-\epsilon}^\epsilon [y \dy h(- \delta, 0) + E_\epsilon(- \delta, y) ] \rho_\epsilon(- \delta, y) \dd y.
\end{align*}
The harmonic measure $\hm$ is symmetric in $y$, thus the integral of $y \dy h$ gives zero.
The error term can be expressed as follows
$$
E_\epsilon(- \delta, y) = \int_0^y \ddy h (- \delta, t) (y - t) \dd t
$$
giving the upper bound
$$
|E_\epsilon(- \delta, y)| \leq C \cdot \frac{y^2}{2}, \quad \forall y \in [- \epsilon, \epsilon],
$$
where $C = \sup \{ \ddy h (- \delta, y), y \in [-\epsilon, \epsilon] \}$.
In consequence, we have
\begin{align*}
& \left| \int_{-\epsilon}^\epsilon h(- \delta, y) \hm(-\delta, y) \dd y
- \frac{\exitingproba}{2} \cdot h(- \delta, 0) \right| \\
& \quad \leq C \int_{-\epsilon}^\epsilon \frac{\epsilon^2}{2} \hm (- \delta, y) \dd y
= \frac{C \epsilon^2}{2} \cdot \frac{\exitingproba}{2},
\end{align*}
allowing us to write
\begin{equation} \label{eqn:left}
\int_{-\epsilon}^\epsilon h(- \delta, y) \rho_\epsilon(- \delta, y) \dd y
= \frac{\exitingproba}{2} \cdot \left[ h(- \delta, 0) + \bigO{\epsilon^2} \right].
\end{equation}
Similarly, we also have
\begin{equation} \label{eqn:right}
\int_{-\epsilon}^\epsilon h(\delta, y) \rho_\epsilon(\delta, y) \dd y
= \frac{\exitingproba}{2} \cdot \left[ h( \delta, 0) + \bigO{\epsilon^2} \right].
\end{equation}

Combining Equations \eqref{eqn:left} and \eqref{eqn:right} and inserting in \eqref{eqn:mean_value}, we get
\begin{align*}
0 & = \frac{\exitingproba}{2} \cdot \bracks{ h(- \delta, 0) + h( \delta, 0) - 2 h(0, 0) + \bigO{\epsilon^2} } \\
& \qquad + \frac{1 - \exitingproba}{2} \cdot \bracks{ h(0, \epsilon) + h(0, -\epsilon) - 2 h(0, 0) } \\
& = \frac{\exitingproba}{2} \cdot \left[ \delta^2 \Deltaxx h(0, 0) + \bigO{\epsilon^2} \right]
+ \frac{1 - \exitingproba}{2} \cdot \left[ \epsilon^2 h_{yy}(0, 0) + \bigO{\epsilon^2} \right].
\end{align*}
We divide everything by $\epsilon^2$ to get
\begin{align*}
0 & = \frac{\exitingproba}{2 \epsilon^2} \cdot \left[ \delta^2 \Deltaxx h(0, 0) + \bigO{\epsilon^2} \right]
+ \frac{1 - \exitingproba}{2} \cdot \left[ h_{yy}(0, 0) + \calO(1) \right].
\end{align*}
When $\epsilon$ goes to $0$, we obtain
$$
\frac{1}{2} \laplacian h (0, 0) = \frac{1}{2} \Deltaxx h (0, 0) + \frac{1}{2} \partial_{yy} h (0, 0) = 0.
$$
\end{proof}

The semi-discrete Laplacian can also be interpreted with the notion of generator.
The \emph{generator} of a continuous-time Markov process $(X_t, Y_t)$ is the linear application $P$ such that
$$
P f (x, y) = \lim_{t \rightarrow 0} \frac{\bbE_{(x, y)}[f(X_t, Y_t)] - f(x, y)}{t}
$$
for $\calC^2$ functions $f : \bbR^2 \rightarrow \bbR$.

\begin{prop}
The generator of $B^{(\delta)}$ is $ \frac{1}{2} \Deltaxx + \frac{1}{2}\partial_{yy}$.
\end{prop}

\begin{proof}
We omit the proof here.
\end{proof}

In $\bbR^2$, the generator of the standard 2D Brownian motion is one half of the planar Laplacian, we may also expect the same property between the semi-discrete Brownian motion and Laplacian.
It is actually satisfied by the above proposition and Equation \eqref{eqn:Laplacian}.

\subsection{Dirichlet boundary problem} \label{sec:dirichlet}

Dirichlet boundary problems are of great importance in discrete and continuous harmonic analysis, which is closely related to the complex analysis.
In this section, we establish the Maximum principle and study such problems.

\begin{prop}[Maximum principle] \label{prop:max_principle}
Consider a primal semi-discrete domain $\primaldomain$.
Let $u$ be a subharmonic function defined on $\primaldomain$, \emph{i.e.} $\laplacian u (z) \geq 0$ for all $z \in \Int \primaldomain$.
Then we have
$$
\sup_{z \in \primaldomain} u(z) = \sup_{z \in \borderprimaldomain} u(z),
$$
meaning that the maximum of $u$ is reached on the boundary.
\end{prop}

\begin{proof} First of all, let us assume that $u$ is strictly subharmonic, meaning that $\laplacian u > 0$ on $\Int \primaldomain$.
Take $z \in \Int \primaldomain$ a point at which $u$ reaches its maximum.
Since it is a maximum on vertical axes, we have $\partial_{yy} u(z) \leq 0$.
And we also have $\Deltaxx u(z) = u(z+\delta) + u(z-\delta) - 2u(z) \leq 0$.
Thus, the semi-discrete Laplacian $\laplacian u(z) = \Deltaxx u(z) + \partial_{yy} u(z) \leq 0$.
This is a contradiction.

In a more general case with $\laplacian u \geq 0$ on $\Int \primaldomain$, let us consider the family of functions $(u_\epsilon)_{\epsilon > 0}$ defined by
Let $\epsilon > 0$ and consider
$$
u_\epsilon (z) = u(z) + \epsilon y^2
$$
where $y$ is the second coordinate of $z$.
We have $\laplacian u_\epsilon = \laplacian u + 2 \epsilon$, meaning that $u_\epsilon$ is subharmonic.
From the first part of the proof, we deduce that
$$
\sup_{z \in \primaldomain} u_\epsilon (z) = \sup_{z \in
\borderprimaldomain} u_\epsilon(z).
$$
Since both terms are finite and decreasing while $\epsilon$ decreases to 0, taking the limit implies the desired result.
\end{proof}

Given a primal semi-discrete domain $\primaldomain$ and a function $g : \borderprimaldomain \to \bbR$, the associated \emph{Dirichlet problem} consists of determining a function $h : \primaldomain \to \bbR$ which
\begin{enumerate}
\item coincides with $g$ on the boundary, \emph{i.e.} $g = h_{\mid \borderprimaldomain}$,
\item satisfies $\laplacian h \equiv 0$.
\end{enumerate}
In such case, we say that $h$ is a \emph{solution} to Dirichlet problem.

\begin{prop}[Existence of solution] \label{prop:existence_dirichlet}
A solution to the Dirichlet problem is given by
\begin{equation} \label{eqn:sol_dirichlet}
h(z) = \bbE \bracks{ g(\BM{T} ) }, \quad \forall z \in \primaldomain,
\end{equation}
where
$$
T = \inf \{ t \geq 0, \BM{t} \notin \primaldomain \}.
$$
\end{prop}

\begin{proof}
We notice that Equation \eqref{eqn:sol_dirichlet} is well defined because the trajectory of $\BM{T}$ is almost surely continuous (in the semi-discrete sense), thus $\BM{T} \in \borderprimaldomain$.

If $h$ is given by Equation \eqref{eqn:sol_dirichlet}, then it satisfies the mean-value property on elementary rectangles as well.
Indeed, take $z \in \primaldomain$ and $\epsilon > 0$ small enough such that $z + \recdomain \subset \primaldomain$.
Consider the stopping time
$$
T' = \inf \{ t \geq 0, \BM{t} \notin z + \recdomain \}
$$
and write
\begin{align*}
h(z) & = \bbE_z \bracks{ g \pars{ \BM{T} } }
= \bbE_z \bracks{ \bbE_{B_{T'}} \bracks{ g \pars{ \BM{T} } \middle| T' } }
= \bbE_z \bracks{ h \pars{ \BM{T'} } }
\end{align*}
which is exactly the mean-value property.
Moreover, one can also show that $h$ is $\calC^2$ using classical arguments (convolution for example), Proposition \ref{prop:equiv_laplacian} gives $\laplacian h \equiv 0$ on $\primaldomain$.
\end{proof}

\begin{prop}[Uniqueness] \label{prop:uniqueness_Dirichlet}
The solution to the Dirichlet problem is unique.
\end{prop}

\begin{proof}
By linearity, it is enough to show uniqueness when the boundary condition is 0.
Consider a semi-discrete domain $\primaldomain$ and $h : \overline{\primaldomain} \to \bbR$ which is zero on the boundary $\borderprimaldomain$ and harmonic in $\primaldomain$.
Applying the maximum principle to $h$ and $-h$, the function $h$ should reach its maximum and minimum on the boundary.
Therefore, it is zero everywhere.
\end{proof}

\begin{proof}[Proof of Proposition \ref{prop:equiv_laplacian}]
Here, we finish the proof of the proposition by using the uniqueness of the solution to the Dirichlet problem.
Consider a semi-discrete domain $\primaldomain$ and a function $f : \primaldomain \to \bbR$ satisfying $\laplacian f \equiv 0$.
We want to show that it satisfies the mean-value property on rectangles.

Take $z \in \primaldomain$ and $\epsilon > 0$ such that $z + \recdomain \subset \primaldomain$.
Consider $g : \partial (z + \recdomain) \to \bbR$ which coincides with $f$.
There exists a unique function $h : z + \recdomain \to \bbR$ such that $h_{z + \partial \recdomain} \equiv g$ and $\laplacian h \equiv 0$ over $z + \recdomain$.
Since $f$ satisfies exactly the same conditions, we have $f \equiv h$ on $z + \recdomain$.

By the construction of the solution to the Dirichlet problem (Proposition \ref{prop:existence_dirichlet}), $f$ satisfies the mean-value property on rectangles.
\end{proof}

\subsection{Green's function} \label{sec:Greens}

A Green's function is a function which is harmonic everywhere except at one point, where it has a singularity given by the Dirac mass.
It is closely related to random walks in the discrete setting and to Brownian motions in the continuous setting.
We will explain its construction in the semi-discrete setting, show that it is unique up to an additive constant and derive some of its properties and asymptotics.

\subsubsection{Construction and properties}

A Green's function is a function $\Greens(z, \zeta)$ defined on the semi-discrete (primal) lattice satisfying the following three properties.
\begin{enumerate}
\item The function $\Greens$ is translational invariant, \emph{i.e.} there exists a function $\Greens$ such that $\Greens(\zeta-z) = \Greens(z, \zeta)$.
\item The function $\zeta \mapsto \Greens(\zeta)$ is $\calC^\infty$ and semi-discrete harmonic except at $\zeta = 0$, where it is only continuous.
\item When $\epsilon > 0$ is small, the quantities $\Greens(\icomp \epsilon)$ and $\Greens(-\icomp \epsilon)$ coincide at zero and second order,
\begin{align*}
\lim_{\epsilon \rightarrow 0^+} \Greens( \icomp \epsilon)
& = \lim_{\epsilon \rightarrow 0^+} \Greens( - \icomp \epsilon), \\
\lim_{\epsilon \rightarrow 0^+} \ddy \Greens( \icomp \epsilon)
& = \lim_{\epsilon \rightarrow 0^+} \ddy \Greens( - \icomp \epsilon),
\end{align*}
whereas at the first order, we have
$$
\dy \Greens(\icomp 0^+) - \dy \Greens(\icomp 0^-)
= \lim_{\epsilon \rightarrow 0} [\dy \Greens(\icomp \epsilon) - \dy \Greens(-\icomp \epsilon) ] = \frac{1}{\delta}.
$$
This is the \emph{normalization} of the Green's function.
\end{enumerate}

If $\Greens$ is a Green's function, we can apply Green's formula (Proposition \ref{prop:greens_formula}) or the Divergence Theorem (Proposition \ref{prop:div_thm}) to get the usual property that, for a dual domain $\dualdomain$ such that $0 \in \Int \dualdomain$,
$$
\int_{\dualdomain} \laplacian \Greens(y) \dd y = 1
$$
and
$$
\int_{\dualdomain} f(y) \laplacian \Greens(y) \dd y = f(0)
$$
where $f$ is a semi-discrete function on $\dualdomain$.

We will show that there exists a unique function (up to an additive constant) having these properties.
To prove the existence and uniqueness of $\Greens$ and compute it, we generalize the method of discrete exponentials from \cite{Kenyon-laplacian}.
First of all, let us define $\Greens$ for $\delta = 1$.
Consider a family of meromorphic functions on $\bbC$ indexed by vertices in $\mediallattice$ in the following way:
\begin{itemize}
\item at the origin: $f_0(z) = \frac{1}{z}$;
\item if $t \in \bbR$, then $f_{\icomp t} (z) = f_0(z) \cdot \exp \left[ 2 \icomp t \left( \frac{1}{z+1} + \frac{1}{z-1} \right) \right]$;
\item if $p \in \mediallattice$, then $f_{p^+} (z) = f_p(z) \cdot \frac{z+1}{z-1}$.
\end{itemize}
In other words, if $\zeta = m + it$ with $m \in \frac12 \bbZ$ and $t \in \bbR$, we can write
$$
f_\zeta (z) = \frac{1}{z}
\cdot \exp \left[ 2 \icomp t \left( \frac{1}{z+1} + \frac{1}{z-1} \right) \right]
\cdot \left( \frac{z+1}{z-1} \right)^{2m}.
$$

\begin{prop}[Green's function] \label{prop:free_Greens}
The following function is a Green's function on $\bbL_1$
\begin{equation} \label{eqn:Green}
G(\zeta) := \frac{1}{8 \pi^2 \icomp}\int_C f_\zeta(z) \ln (z) \dd z
\end{equation}
where $C$ in a path in $\bbC$ depending on $\zeta$, surrounding $\{ e^{\icomp \theta}, 0 \leq \theta \leq \pi \}$ and leaving the origin outside the contour.
For the complex logarithm, we define it in $(\theta-\pi, \theta+\pi)$ where $\theta$ is an argument of $\zeta$.
\end{prop}

\begin{rmk}
We can estimate the Green's function of Proposition \ref{prop:free_Greens} with the help of the residue theorem.
The possible poles of the function $f_t(z) \ln (z)$ are $1$ and $-1$ and the choice of the branch creates a possible discontinuity only when $\myr(t) = 0$ since elsewhere, we have $\ln(-1) - \ln(1) = \pm \icomp \pi$ with a $+$ sign in the upper half-plane and a $-$ sign in the lower half-plane.
\end{rmk}

\begin{proof}
We start by checking that $G$ is well-defined.
If we change the lift of the logarithm, (equivalent to adding $2k\pi$ to $\log$), we need to show that this does not change the value of $G$, that is to say
\begin{equation*}
\int_C f_\zeta(z) \dd z = 0
\end{equation*}
for all $\zeta \in \primallattice$.
This is shown in the Appendix, see Proposition \ref{prop:int_zero}.

In each of the half-planes, the function $G$ is $\calC^\infty$ because we integrate a smooth function along a path and the branch of the logarithm does not cross $1$ or $-1$ where we take residues to estimate the integral.

On the $\bbZ$ axis (except from the origin), we can develop the exponential and see that the residues at $1$ and at $-1$ coincide at all orders.
It is explained in Proposition \ref{prop:Greens_Ck}.
This tells us that $G$ is $\mathcal{C}^\infty$ on $\bbL_1$ except at 0.

Let us now check that $G$ is harmonic everywhere apart from the origin.
Actually, it is sufficient to check that $f_\zeta$ is harmonic (with respect to $\zeta$) except at the origin.
Writing $\zeta = x + iy$ with $x \in \bbZ$ and $y \in \bbR$, we find
\begin{align*}
\discretelaplacian f_t & = \Deltaxx f_t + \ddy f_t \\
& = \left[ (2 \icomp)^2 \left( \frac{1}{z+1} + \frac{1}{z-1} \right)^2 
+ \left( \frac{z+1}{z-1} \right) ^2
+ \left( \frac{z-1}{z+1} \right) ^2
- 2 \right] f_t \\
& =0.
\end{align*}
Here, it is allowed to add all the terms together since we consider always the same branch of logarithm.

To conclude the proof, we need to check the third property.
This follows from a direct computation:
\begin{align*}
\lim_{\epsilon \rightarrow 0^+} \dy G(\icomp \epsilon) & =
\frac{1}{4 \pi^2} \int_C \left( \frac{1}{z+1} + \frac{1}{z-1} \right) \frac{\ln(z)}{z} \dd z \\
& = \frac{\icomp}{2 \pi} \left[ \ln(1) - \ln(-1) \right ] = \frac{1}{2}
\end{align*}
where we use the residue theorem in the second equality.
Similarly, we have
\begin{align*}
\lim_{\epsilon \rightarrow 0^-} \dy G(\icomp \epsilon) = -\frac{1}{2}.
\end{align*}
Then, to see that the zero and second orders of $G$ around zero coincide in the upper and lower half-planes, we use Lemma \ref{lem:general_zero}.

As such, we get all the properties we were looking for.
\end{proof}

\begin{prop}[Asymptotics of Green's function] \label{prop:asymp_Green's}
Let $\zeta \in \bbL_1$.
When $|\zeta|$ goes to infinity, we have the following asymptotic behavior,
\begin{equation} \label{eqn:Green's_asymptotics}
G(\zeta) = \frac{1}{2 \pi} \ln (4|\zeta|) + \frac{\gamma_{\rm Euler}}{2 \pi} + \bigO{ \frac{1}{|\zeta|^2} }.
\end{equation}
\end{prop}

\begin{proof}
The proof is similar to the one given in \cite{Kenyon-laplacian} for the discrete Green's function on isoradial graphs.
We need to be careful when dealing with the exponential term in $f_\zeta$.
Let us write $d = |\zeta|$.
To get the improved error term $\bigO{1/d^2}$, we use the method from \cite{Bucking-circlepack}.

Consider $\zeta = m + \icomp t \in \primallattice$ and write $d = |\zeta|$, $\arg \zeta = \theta_0$.
Take $r = \bigO{1/d^4}$ and $R = \bigO{d^4}$.
We will consider the path $C$ going as follows:
\begin{enumerate}
\item counterclockwise around the ball of radius $R$ around the origin from the angles $\theta_0 -\pi$ to $\theta_0 + \pi$,
\item along the direction $e^{\icomp \theta_0}$ from $-R$ to $-r$,
\item around the ball of radius $r$ around the origin from angle $\theta_0 + \pi$ back to $\theta_0 -\pi$,
\item along the direction $e^{\icomp \theta_0}$ from $-r$ to $-R$, back to the starting point.
\end{enumerate}
This path is illustrated in Figure \ref{fig:int_path}.

\begin{figure}[htb]
  \centering
    \includegraphics[scale=0.8]{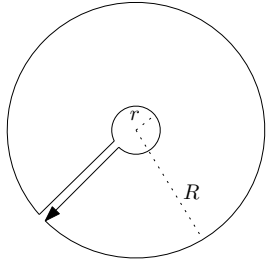}
    \caption{The path along which we integrate in the proof of Proposition \ref{prop:asymp_Green's}.}
    \label{fig:int_path}
\end{figure}

We estimate these integrals separately, combining the two integrals along the direction $e^{\icomp \theta_0}$.
First of all, to study the integral around the ball of radius $r$, we start by developping $f_\zeta$ for $z = r e^{\icomp \theta}$ when $r$ is small:
\begin{align*}
\frac{1}{1+ z} - \frac{1}{1- z} = \bigO{r} \quad \mbox{ and } \quad
\left( \frac{1+z}{1-z} \right) = \exp(\bigO{r})
\end{align*}
Thus,
$$
f_\zeta(z) = \frac{1}{z} \exp( \bigO{dr} ) = \frac{1}{z} (1 + \bigO{dr}).
$$
The integral around the ball of radius $r$ is
\begin{align*}
\frac{1}{8 \pi^2 \icomp} \int_{\theta_0 +\pi}^{\theta_0 -\pi} (1 + \bigO{dr}) (\ln r + \icomp \theta) \icomp \dd \theta
= - \frac{\ln r}{4 \pi} (1 + \bigO{dr})
= - \frac{\ln r}{4 \pi} + \bigO{\frac{1}{d^2}}
.
\end{align*}
Similarly, the integral around the ball of radius $R$ is
$$
\frac{\ln R}{4 \pi} + \bigO{\frac{1}{d^2}}.
$$
On the direction $e^{\icomp \theta_0}$ from $-R$ to $-r$, we add up the two integrals.
Since the logarithm differs by $2\pi \icomp$ on the both sides, by combining we get
$$
\frac{e^{\icomp \theta_0}}{4 \pi} \int_{-R}^{-r} f_\zeta(s e^{\icomp \theta_0}) \dd s.
$$
We should split this integral into 3 parts, $I_1$ for the integral from $-R$ to $\sqrt{d}$, $I_2$ from $-\sqrt{d}$ to $-1/\sqrt{d}$ and $I_3$ from $-1/\sqrt{d}$ to $-r$.
For $|z|$ small, we develop $f_\zeta$ to two orders further:
$$
f_\zeta(z) = \frac{e^{4 \overline{\zeta} z + \bigO{dz^3}}}{z}.
$$
Thus, $I_1$ can be rewritten (let $\alpha = 4 \overline{\zeta} e^{\icomp \theta_0} = 4 |\zeta| = 4d$)
\begin{align*}
I_1 & = \frac{1}{4 \pi} \int_{-1/\sqrt{d}}^{-r} \frac{e^{\alpha s + \bigO{ds^3}}}{s} \dd s
 = \frac{1}{4 \pi} \left(
\int_{-\alpha / \sqrt{d}}^{-\alpha r} \frac{e^s}{s} \dd s
+
\int_{-1/\sqrt{d}}^{-r} \bigO{ds^3} \frac{e^{\alpha s}}{s} \dd s
\right) \\
& = \frac{1}{4 \pi} \left(
\int_{-\alpha / \sqrt{d}}^{-1} \frac{e^s}{s} \dd s
+ \int_{-1}^{-\alpha r} \frac{e^s-1}{s} \dd s
+ \int_{-1}^{-\alpha r} \frac{\dd s}{s}
 \right) +\bigO{\frac{1}{d^2}} \\
& = \frac{1}{4 \pi} \left(
\ln( \alpha r) + \gamma_{\rm Euler}
\right) +\bigO{\frac{1}{d^2}},
\end{align*}
where in the first equality, we develop the exponential and the integral with $\calO$ term gives $\bigO{1/d^2}$; and in the last line, $\ln(\alpha r)$ comes from the third term and Euler's constant comes from the first two integrals by taking $\alpha r \rightarrow 0$ and $\alpha/\sqrt{d} \rightarrow \infty$.

In a similar way (or by making the change of variable $s \rightarrow 1/s$), we get
$$
I_3 = \frac{1}{4 \pi} (-\ln(R/\alpha) + \gamma_{\rm Euler}) + \bigO{\frac{1}{d^2}}.
$$

When it comes to the intermediate term, we can show that it is negligeable.
Let $z = s e^{\icomp \theta_0}$ for $s \in [-\sqrt{d}, -1/\sqrt{d}]$.
For the exponential term in $f_\zeta$, we have
\begin{align*}
\exp \left[ 2 \icomp t \left( \frac{1}{z+1} + \frac{1}{z-1} \right) \right] & 
= \exp \left[ - 2 t \myi \left( \frac{1}{z+1} + \frac{1}{z-1} \right) \right] \\
& = \exp \left[ 2 t s \sin \theta_0 \left( \frac{1}{|z+1|^2} + \frac{1}{|z-1|^2} \right) \right] \\
& \leq \exp \left( - \bigO{ \frac{t^2}{d^{3/2}} } \right).
\end{align*}
Then, for the other one, assume $m \geq 0$ (so $\cos \theta_0 \geq 0$),
\begin{align*}
\left| \frac{z+1}{z-1} \right|^{2m} & = \left( \frac{s^2+1 + 2s \cos \theta_0}{s^2+1 - 2s \cos \theta_0} \right)^m 
\leq \left( 1+ \frac{4s \cos \theta_0}{(s-1)^2}  \right)^m \\
& \leq \exp \left( \frac{4ms \cos \theta_0}{(s-1)^2} \right)
\leq \exp \left( -\bigO{\frac{m^2}{d^{3/2}}} \right)
\end{align*}
and we have the same bound for $m \leq 0$.
After all, the intermediate term can be bounded by
$$
I_2 \leq \sqrt{d} e^{-\bigO{\sqrt{d}}} \leq \bigO{\frac{1}{d^2}}.
$$

Finally, we sum up all the terms and take the limits $r \rightarrow 0$ and $R \rightarrow \infty$ to have
$$
G(\zeta) = \frac{1}{2 \pi} \left[ \ln( 4 |\zeta| ) + \gamma_{\rm Euler} \right] + \bigO{\frac{1}{d^2}}. \qedhere
$$
\end{proof}

\begin{prop} \label{prop:free_delta_Greens}
On $\primallattice$, there exists a unique Green's function with the following normalization at $0$
$$
G_\delta(0) = \frac{1}{2 \pi} ( \ln \delta - \ln 4 - \gamma_{\rm Euler}).
$$
Moreover, its asymptotic behavior when $\frac{|\zeta|}{\delta}$ goes to $\infty$ is
$$
\Greens(\zeta) = \frac{1}{2 \pi} \ln |\zeta| + \bigO{\frac{\delta^2}{|\zeta|^2}}.
$$
We call $\Greens$ the \emph{free Green's function on $\primallattice$}.
\end{prop}

\begin{proof}
To construct a Green's function with such normalization, we can consider
\begin{equation*}
\Greens(\zeta) = G \left( \frac{\zeta}{\delta} \right) + \frac{1}{2 \pi} \left(  \ln \delta - \ln 4 - \gamma_{\rm Euler} \right).
\end{equation*}
To show the uniqueness, assume that $\Greens$ and $\tilde{\Greens}$ are two such functions.
Let $G = \Greens - \tilde{\Greens}$.
The first order singularities at zero cancel out due to the same normalization, so the function $G$ is $\calC^1$ around 0.
Since the second order terms of $\Greens(\icomp \epsilon)$ and $\Greens(- \icomp \epsilon)$ coincide (same for $\tilde{\Greens}$), $G$ is $\calC^2$ around 0.
Finally, $G$ is harmonic in $\primallattice$ and is bounded (due to the asymptotic behaviour), it should be zero everywhere by Harnack principle (see below, Proposition \ref{prop:Harnack}).
\end{proof}


Given a primal semi-discrete domain $\primaldomain$, we can define the \emph{Green's function on $\primaldomain$} by
$$
\GreensDomain = \Greens - H_{\primaldomain}
$$
where $\Greens$ is the free Green's function on $\primallattice$ and $H_{\primaldomain}$ the unique solution to the Dirichlet problem on the primal semi-discrete domain $\primaldomain$ whose boundary condition is given by ${\Greens}_{\mid \borderprimaldomain}$
Here, we notice that $\GreensDomain$ is non-positive.

\subsubsection{Link with Brownian motion}

In the discrete setting, the Green's function measures how much time a standard Brownian motion spends on each site in average; in the continuous setting, it also gives an analogous of this.
We establish an equivalent of such a property in the semi-discrete case.

\begin{prop}
\label{prop:Green's_BM}
Let $\BM{}$ be a Brownian motion started at $x \in \Qprimal$ as defined in Section \ref{sec:BM}, stopped at $\tau$, the exiting time of the domain $\Qprimal$.
Then, we have the asymptotic behaviour which is independent of $\delta$
$$
\bbE[ \tau ] \asymp \int_{\Qprimal} |\GreensDomain (x, y)| \dd y
$$
where the left-hand side is the average time spent by the semi-discrete Brownian motion $\BM{}$ in $\Qprimal$; and the right-hand side is the integral of the Green's function on the same domain.
\end{prop}

On a discrete graph, the expectation of the number of visits of a random walk (stopped after an exponential time if it is recurrent) is given by the opposite of its associated Green's function.
Similarly, the average time spent by a Brownian motion in a continuous space ($\bbR^d$ for example) is also given by the opposite of its associated Green's function (again stopped after an exponential time if it is recurrent).
Here in the semi-discrete setting, we should interpret the Green's function in the continuous direction as the average time spent by the Brownian motion; whereas in the discrete direction, the expectation of the number of « visits », which explains the factor $\delta$ in $\int_{\Qprimal}$.

\begin{proof}
The semi-discrete Brownian motion $\BM{}$ converges in law to its continuous 2D counterpart, so does the semi-discrete Green's function, which converges uniformly on every compact not containing 0 to the 2D Green's function.
As such, the integral of semi-discrete Green's function converges to the integral of 2D Green's function on every compact not containing 0.
Moreover, $\ln y$ is integrable on $[0, \epsilon]$, so the integral of the semi-discrete Green's function on a small rectangular domain around 0 can be well controlled, and this quantity goes to 0 when $\epsilon \rightarrow 0$.
\end{proof}


\begin{prop}[Link with harmonic measure] \label{prop:link_green_harm_measure}
Consider a semi-discrete primal domain $\primaldomain$ with $u_0 \in \Int \primaldomain$.
Let $a \in \partial \primaldomain := \calC_h \cup \calC_v$ be a point on the boundary, where $\calC_h$ and $\calC_v$ denote respectively the horizontal and vertical parts.
Write $\omega_{\primaldomain}(u_0, \{ a \})$ for the harmonic measure with respect to $u_0$.
We notice that it should be seen as a density when $a \in \calC_v$ and Dirac masses when $a \in \calC_h$.
Then,
\begin{itemize}
\item if $a \in \calC_v$, we have $\omega(u_0, \{ a \} ) = - \frac{1}{\delta} G_{\primaldomain}(u_0, a_{int})$,
\item if $a \in \calC_h$, we have $\omega(u_0, \{ a \} ) = \pm \dy G_{\primaldomain}(u_0, a)$,
\end{itemize}
where $a_{int}$ is the unique vertex in $\{ a \pm \delta \} \cap \Int \primaldomain$, $\dy$ is the vertical derivative with respect to the second coordinate and we take the + sign if the boundary is oriented to the left at $a$ and the - sign otherwise.
\end{prop}

\begin{proof}
It is immediate from Green's Theorem (Proposition \ref{prop:greens_thm}) by taking $f = \omega_{\primaldomain} (\cdot, \{ a \})$ and $g = G_{\primaldomain}(u_0, \cdot)$ and the fact that $\int_{\primaldomain} f \laplacian g = 1$.
\end{proof}

\begin{lem} \label{lem:bound_harm_measure}
We keep the same notation as above and take $\primaldomain = \ball(u_0, R)$.
There exist two positive constants $c_1$ and $c_2$, independent of $\delta$, such that
\begin{itemize}
\item if $a \in \calC_v$, we have $c_1 \leq \omega(u_0, \{ a \} ) \leq c_2$ ;
\item if $a \in \calC_h$, we have $c_1 \delta \leq \omega(u_0, \{ a \} ) \leq c_2 \delta$.
\end{itemize}
\end{lem}

\begin{proof}
We link the harmonic measure to Green's function via Proposition \ref{prop:link_green_harm_measure}, which can be estimated more easily by its asymptotic behavior given in Proposition \ref{prop:free_delta_Greens}.
We can write
\begin{align*}
\GreensDomain (u_0, u) - \frac{1}{2\pi} \ln \frac{|u-u_0|}{R}
= \left[ \Greens(u_0, u) - \frac{1}{2\pi} \ln |u-u_0| \right]
- \left[ \GreensDomain^*(u, u_0) - \frac{1}{2\pi} \ln R \right].
\end{align*}
The first term is $\bigO{ \frac{\delta^2}{|u-u_0|^2}}$. The second term is harmonic in $\ball(u_0, R)$, thus by maximum principle, we get
\begin{align*}
\left| \GreensDomain^*(u, u_0) - \frac{1}{2\pi} \ln R \right|
& \leq \sup_{v \in \partial \ball(u_0, R)} \left| \Greens(v, u_0) - \frac{1}{2\pi} \ln R \right| \\
& \leq \frac{\delta}{\pi R} + \bigO{\frac{\delta^2}{R^2}},
\end{align*}
where we use the fact that $R - 2\delta \leq |v - u_0| \leq R$ since $v \in \partial \ball(u_0, R)$.
In summary, for all $u \in \ball(u_u, R)$,
$$
\left| \GreensDomain (u_0, u) - \frac{1}{2\pi} \ln \frac{|u-u_0|}{R} \right|
= \frac{\delta}{\pi R} + \bigO{ \frac{\delta^2}{|u-u_0|^2} + \frac{\delta^2}{R^2}}.
$$
By taking $u = a_{int}$ with $a \in \calC_v$, we get the first part of the proposition.
By taking $u \in \ball(u_0, R)$ closer and closer to $a \in \calC_h$, we get the second part.
\end{proof}

\subsection{Harnack Principle and convergence theorems}

This part deals mostly with harmonic analysis.
We give the semi-discrete version of Harnack Lemma, Riesz representation and a convergence theorem of harmonic functions.

\begin{prop}[Semi-discrete Harnack Lemma] \label{prop:Harnack}
Let $u_0 \in \primaldomain$ and $0 < r < R$ such that $\ball(u_0, R) \subset \primaldomain$.
Consider a non-negative semi-discrete harmonic function $H : \ball(u_0, R) \rightarrow \bbR$.
Let $M = \max H$ on $\primaldomain$.
If $u, u^+ \in \ball(u_0, r)$, then 
\begin{align*}
|H(u^+) - H(u)| & \leq \Cst \cdot \frac{\delta M}{R-r}, \\
\dy H(u) & \leq \Cst \cdot \frac{M}{R-r}.
\end{align*}
\begin{killcontents}
If $u_1, u_2 \in \ball(u_0, r) \subset \Int \ball(u_0, R)$, then
$$
\exp \left[ - \Cst \cdot \frac{r}{R-r} \right] \leq \frac{H(u_2)}{H(u_1)}
\leq \exp \left[ \Cst \cdot \frac{r}{R-r}\right]
$$
\end{killcontents}
\end{prop}

\begin{proof}
The proof is classical.
We use coupled semi-discrete Brownian motions issued from two neighboring sites on semi-discrete lattice by reflection.
For the derivative in $y$, we use the same method.
The difference between $H(u + \icomp \epsilon)$ and $H(u - \icomp \epsilon)$ can be bounded by $\Cst \cdot \frac{\epsilon M}{R-r}$ and by dividing everything by $\epsilon$ and taking the limit, we obtain the inequality.
\end{proof}

\begin{prop}[Riesz representation] \label{prop:Riesz}
Let $f$ be a function on $\primaldomain$ vanishing on the boundary $\partial \primaldomain$.
Then, for all $y \in \primaldomain$,
$$
f(y) = \int_{\primaldomain} \laplacian f(x) G_{\primaldomain}(x, y) \dd x.
$$
Here, if $f$ is not differentiable, we can define the integral in the sense of distributions.
\end{prop}

\begin{proof}
The function $f - \int_{\primaldomain} \laplacian f(x) G_{\primaldomain}(x, \cdot) \dd x$ is harmonic and zero on the boundary, thus zero everywhere.
\end{proof}

\begin{thm}[Convergence theorem for harmonic functions] \label{thm:cvg_harm}
Let $(h_\delta)_{\delta > 0}$ be a family of semi-discrete harmonic functions on $\primaldomain$.
It forms a precompact family for the uniform topology on compact subsets of $\Omega$ if one of the following conditions is satisfied.
\begin{enumerate}
\item The family $(h_\delta)$ is uniformly bounded on any compact subset of $\Omega$.
\item For any compact $K \subset \Omega$, there exists $M = M(K) > 0$ such that for all $\delta > 0$, we have
$$
\int_{K_\delta} |h_\delta(x)|^2 \dd x \leq M.
$$
\end{enumerate}
\end{thm}

\begin{proof}
The first point comes from Arzelà-Ascola since $(h_\delta)$ is uniformly Lipschitz (Proposition \ref{prop:Harnack}).

We will show that the second point implies the first one to conclude.
Start by choosing a compact subset $K \subset \Omega$.
Denote by $d = d(D, \partial \Omega)$ the distance between $K$ and the boundary of $\Omega$.
Let $K'$ be the $d/2$-neighborhood of $K$.

Let $0 < \delta < d/2$ and $x \in \Int K_\delta$.
Choose $Q$ to be a rectangular domain in $K'$ which is centered at $x$.
Write $Q_\delta = (x + [-r\delta, r\delta] \times [-s, s]) \cap \primallattice$, $r \in \bbN$, for its semi-discrete counterpart.
It is possible to have $r\delta > d/4$ and $s > d/4$ due to the assumption on the distance, and we assume so in the following.

If we write $H_k = \{ k \delta\} \times [-s, s]$, the hypothesis implies
$$
\sum_{k=\frac{r}{2}}^{r} \delta \left( \int_{H_{-k}} + \int_{H_{k}} \right) |h_\delta(y)|^2 \dd y \leq M(K') =: M
$$
for a certain constant $M$ which is uniform in $\delta$.
Take $p \in \llbracket r/2, r \rrbracket$ such that the summand is minimum, we get
$$
\left( \int_{H_{-p}} + \int_{H_{p}} \right) |h_\delta(y)|^2 \dd y \leq \frac1\delta \frac{M}{r/2} \leq c_1,
$$
where $c_1$ is a uniform constant in $\delta$.

For $t \in [0, s]$, denote $H_p^t = \{ p \delta \} \times [-t, t]$.
We can write, by linearity, $h_\delta$ as linear combination of harmonic measures,
\begin{equation} \label{eqn:avg_harm}
h_\delta(x) = \left( \int_{H_{-p}^t} + \int_{H_{p}^t} \right) h_\delta(y) \omega_t(x, y) \dd y
+ \sum_{\substack{k=-p+1 \\ y = k\delta \pm \icomp t}}^{p-1} h_\delta(y) \omega_t(x, y),
\end{equation}
where $\omega_t$ is the harmonic measure in $[-p \delta, p \delta] \times [-t, t]$.

We integrate Equation \eqref{eqn:avg_harm} from $t=s/2$ to $t=s$ and get
\begin{align*}
h_\delta (x) = &
\frac{2}{s} \int_{s/2}^s \left( \int_{H_{-p}^t} + \int_{H_{p}^t} \right) h_\delta(y) \omega_t(x, y) \dd y \dd t \\
& \quad + \frac{2}{s} \int_{s/2}^2 \sum_{\substack{k=-p+1 \\ y = k\delta \pm \icomp t}}^{p-1} h_\delta(y) \omega_t(x, y) \dd t.
\end{align*}
We want to show that $h_\delta(x)$ is uniformly bounded in $x$ and in $\delta$.
We will take its square and apply Cauchy-Schwarz.

Below, denote respectively the first and second integrals $A$ and $B$.
First of all, notice that $h_\delta(x)^2 \leq 2 (A^2 + B^2)$, so we just need to show that $A$ and $B$ are bounded.
We have,
\begin{align*}
A^2 & \leq \left( \frac{2}{s} \right)^2 \int_{s/2}^s \left( \int_{H_{-p}^t} + \int_{H_{p}^t} \right) h_\delta(y)^2 \dd y \dd t
\int_{s/2}^s \left( \int_{H_{-p}^t} + \int_{H_{p}^t} \right) \omega_t(x, y)^2 \dd y \dd t \\
& \leq \left( \int_{H_{-p}^s} + \int_{H_{p}^s} \right) h_\delta(y)^2 \dd y
 \left( \int_{H_{-p}^t} + \int_{H_{p}^t} \right) \omega_s(x, y)^2 \dd y
\end{align*}
where $(2/s)^2$ is distributed once in the first term and once in the second, normalizing the integrals. (The length of the segment along which we integrate is $s/2$.)
Then, for $x$ and $y$ fixed, $\omega_t(x, y) \leq \omega_s(x, y)$ for $t \leq s$.
Here, the first term is bounded by $c_1$ by hypothesis, and the second by another constant $c_2$ from Lemma \ref{lem:bound_harm_measure}.

For the second term, Cauchy-Schwarz gives
\begin{align*}
B^2 & \leq \left( \frac{2}{s} \right)^2 \int_{s/2}^s \sum_{\substack{k=-p+1 \\ y = k\delta + \icomp t}}^{p-1} \delta h_\delta(y)^2 \dd t
\int_{s/2}^s \sum_{\substack{k=-p+1 \\ y = k\delta + \icomp t}}^{p-1} \frac{\omega_t(x, y)^2}{\delta} \dd t.
\end{align*}
On the right-hand side, the first term is bounded by $M$ by assumption and the second term bounded by a uniform constant $c_3$ because $\omega_t(x, y)$ can be bounded by $c_4 \delta$ uniformly (in a similar manner as before) in $\delta$ and in $y-x$, and there are $\bigO{1/\delta}$ terms in the sum.
\end{proof}

\begin{prop}[Estimation on the derivative of the Green's function] \label{prop:estimation_Green's}
Let $Q \subset \primaldomain$ such that $9Q \subset \primaldomain$.
There exists $C > 0$ such that for all $\delta > 0$ and $y \in 9\Qprimal$, we have
$$
\int_{Q_\delta} |\Deltax G_{9\Qprimal} (x, y)| \dd x \leq C \int_{\Qprimal} |G_{9 \Qprimal} (x, y)| \dd x.
$$
\end{prop}

\begin{proof}
For $y \in 9\Qprimal \backslash 3\Qprimal$, we estimate the Green's function in terms of the Brownian motion, or more precisely, the harmonic measure.
We recall that $G_{9\Qprimal} (\cdot, y)$ is non positive, so that we can write for $x \in 2 \Qprimal$,
$$
|G_{9\Qprimal} (x, y)| =
\int_{\calC_v} |G_{9\Qprimal}(z, y)| \omega_{2\Qprimal}(z, y) |\dd z|
+ \frac1\delta \int_{\calC_h} |G_{9\Qprimal}(z, y)| \omega_{2\Qprimal}(z, y) |\dd z|
$$
where we denote the vertical and horizontal parts of the boundary $\partial(2\Qprimal)$ by $\calC_v$ and $\calC_h$.
We can assume that 
$$
H := \int_{\calC_h} |G_{9\Qprimal} (z, y)| |\dd z| \geq V:= \int_{\calC_v} |G_{9\Qprimal} (z, y)| |\dd z|.
$$
The estimations of $\omega_{9\Qprimal}$ in Lemma \ref{lem:bound_harm_measure} gives us the lower and upper bounds easily
\begin{align*}
|G_{9\Qprimal} (x, y)| \leq c_2 \oint_{\partial (2\Qprimal)} |G_{9\Qprimal}(z, y)| |\dd z| 
= c_2 (H+V) \leq 2 c_2 H\
\end{align*}
and
\begin{align*}
|G_{9\Qprimal} (x, y)| & \geq c_1 \oint_{\partial (2\Qprimal)} |G_{9\Qprimal}(z, y)| |\dd z| \\
& \geq c_1 \int_{\calC_h} |G_{9\Qprimal}(z, y)| |\dd z| = c_1 H.
\end{align*}
Thus, for $x, x' \in 2\Qprimal$,
$$
\frac{1}{c_3} |G_{9\Qprimal} (x, y)| \leq |G_{9\Qprimal} (x', y)| \leq c_3 |G_{9\Qprimal}(x, y)|
$$
with $c_3 = 2 c_2 / c_1$.
Knowing that $G_{9\Qprimal}(\cdot, y)$ is harmonic in $\Qprimal$, we apply Proposition \ref{prop:Harnack} and the above inequality to get
$$
|\Deltax G_{9\Qprimal} (x, y)| \leq c_4 \max_{x' \in \Qprimal} |G_{9 \Qprimal}(x', y)| \leq c_3 c_4 |G_{9\Qprimal} (x, y)|.
$$

For $y \in 3\Qprimal$, from Proposition \ref{prop:Green's_BM}, the average time spent by the Brownian motion, stopped when touching $\partial 9\Qprimal$, in $\Qprimal$ is proportional to
$$
\int_{\Qprimal} |G_{9\Qprimal}(x, y)| \dd x.
$$
This quantity can be bounded from below by a constant $c_5$ because the semi-discrete Brownian motion converges to its continuous counterpart in $\bbR^2$.

Now it remains to show that the left-hand side can be bounded from above by a constant.

We write
$$
G_{9\Qprimal} (x, y) = [G_{9\Qprimal} (x, y) - G_\delta (x, y)] + G_\delta(x, y),
$$
where $G_\delta$ is the Green's function on $\primallattice$ defined in Proposition \ref{prop:free_delta_Greens}.
The first part $G_{9\Qprimal} (\cdot, y) - G_\delta(\cdot, y)$ is harmonic in $9\Qprimal$ (the singularities cancel out); moreover, on the boundary $\partial 9\Qprimal$, $G_{9\Qprimal}$ is zero and $G_\delta$ is bounded by a constant depending only on the domain $Q$ by using the asymptotic behavior of the free Green's function in Proposition \ref{prop:free_delta_Greens}.
The Harnack principle (Proposition \ref{prop:Harnack}) gives
$$
| \Deltax [ G_{9\Qprimal} (x, y) - G_\delta (x, y) ] | \leq c_6.
$$
Thus,
$$
\int_{\Qprimal} | \Deltax [ G_{9\Qprimal} (x, y) - G_\delta (x, y) ]  | \dd x \leq \delta \cdot \frac{c_7}{\delta} \cdot c_6 = c_6 c_7.
$$
Concerning $|\Deltax G_\delta(x, y)|$, we can look at its asymptotic behaviour and show that, for $\zeta = m + \icomp t \in \Qprimal$ with $|\zeta| \gg \delta$,
\begin{align*}
\Deltax G_\delta(\zeta, y) & = \frac{1}{\delta} \left[ G_\delta(\zeta + \frac\delta2, y) - G_\delta(\zeta - \frac\delta2, y) \right] \\
& = \frac{1}{2 \pi \delta} \left[ \ln\left( \frac{(m+\frac\delta2)^2 + t^2}{m^2+t^2} \right) - \ln\left( \frac{(m-\frac\delta2)^2 + t^2}{m^2+t^2} \right) \right] \\
& = \frac{1}{2 \pi \delta} \left[ \frac{2\delta m}{m^2+t^2} + \bigO{\delta^2} \right] = \frac{1}{\pi} \frac{m}{m^2+t^2} + \bigO{\delta}.
\end{align*}
Thus, by integrating along vertical axes in $\primaldomain$, we  get some quantity with the same asymptotic behaviour (independent of $\delta$ but on $\Qprimal$),
$$
\int_{\Qprimal} |\Deltax G_\delta(x, y)| \dd x \leq c_8.
$$
The proof follows readily.
\end{proof}

\subsection{Convergence to the continuum Dirichlet problem}

In this subsection, we study the convergence of semi-discrete harmonic functions when the mesh size of the lattice goes to 0.

\begin{lem} \label{lem:cvg_harm}
Let $\Omega$ be a domain and $(\primaldomain)$ its semi-discretized approximations converging to $\Omega$ in the Carathéodory sense.
For each $\delta > 0$, consider a semi-discrete harmonic function $h_\delta$ on $\primaldomain$.
Assume that $h_\delta$ converges uniformly on any compact subset of $\Omega$ to a function $h$, then $h$ is also harmonic.
\end{lem}

\begin{proof}
From Proposition \ref{prop:Harnack} and Theorem \ref{thm:cvg_harm}, we know that the family $(\Deltax h_\delta)$ is precompact thus we can extract from it a converging subsequence.
Since $\partial_x h$ is the only possible sub-sequential limit, $(\Deltax h_\delta)$ converges.
Similarly, one can also prove that $\laplacian h_\delta = 0$ converges to $\Laplacian h$, which is also zero.
\end{proof}

\begin{prop} \label{prop:cvg_Dirichlet}
Let $\Omega$ be a domain with two marked points on the boundary $a, b \in \partial \Omega$.
Consider $f$ a bounded continuous function on $\partial \Omega \backslash \{ a, b\}$ and $h$ the solution associated to the Dirichlet boundary value problem with boundary condition $f$.
For each $\delta > 0$, let $\primaldomain$ be the semi-discretized counterpart of the domain $\Omega$, $a_\delta$ and $b_\delta$ approximating $a$ and $b$.
Let $f_\delta : \partial \primaldomain \rightarrow \bbR$ be a sequence of uniformly bounded functions converging uniformly away from $a$ and $b$ to $f$ and $h_\delta$ be the solution of the semi-discrete Dirichlet boundary problem with $f_\delta$ as boundary condition.
Then,
$$
h_\delta \rightarrow h
$$
uniformly on compact subsets of $\Omega$.
\end{prop}

\begin{proof}
We first notice that the semi-discretized domains converge in the Carathéo\-dory sense to $(\Omega, a, b)$.
Since $(f_\delta)$ is uniformly bounded, it is the same for the family $(h_\delta)$.
Theorem \ref{thm:cvg_harm} says that $(h_\delta)$ is a precompact family.
Let $\tilde{h}$ be a subsequential limit, which should also be harmonic inside $\Omega$ by Lemma \ref{lem:cvg_harm}.
To show that $h = \tilde{h}$, we need to prove that $\tilde{h}$ can be extended to the boundary by $f$ in a continuous way.
This can be done by using the weak Beurling's estimate, obtained by analyzing the Brownian motion on semi-discrete lattice.
This result is classical and can be adapted easily to our semi-discrete setting.
\end{proof}

\subsection{S-holomorphicity} \label{sec:s-hol}

The notion of \emph{s-holomorphicity} will turn out to be important when it comes to the convergence Theorem \ref{thm:cvg_shol}.
Actually, the semi-discrete holomorphicity provides us only with half of Cauchy-Riemann equations and the rest of the information can be ``recovered'' by s-holomorphicity.
More precisely, it allows us to define a primitive of $\myi f^2$ where $f$ is s-holomorphic in Section \ref{sec:obs_primitive} and the convergence of this primitive will tell us more about the convergence of $f$.
We will discuss this in more details in Sections \ref{sec:observables} and \ref{sec:CVG_thm}.

\begin{defn}
Let $f : \midedgedomain \to \bbC$ be a function defined on the mid-edge semi-discrete lattice.
It is said to be \emph{s-holomorphic} if it satisfies the two following properties.
\begin{enumerate}
\item Parallelism: for $e \in \midedgedomain$, we have $f(e) \parallel \tau(e)$ where $\tau(e) = [\icomp (w_e - u_e)]^{-1/2}$, $u_e$ and $w_e$ denote respectively the primal and the dual extremities of the mid-edge $e$.
In other words,
\begin{itemize}
\item $f(e) \in \nu \bbR$ if $p_e^+$ is a dual vertex, and
\item $f(e) \in \icomp \nu \bbR$ if $p_e^+$ is a primal vertex,
\end{itemize}
where $\nu = \exp(- \icomp \pi /4)$.
\item Holomorphicity: for all vertex $e$ on the mid-edge lattice $\midedgedomain$, we have $\derzbar f(e) = 0$.
\end{enumerate}
\end{defn}

\begin{defn}
Let $g : \medialdomain \to \bbC$ be a function defined on the medial semi-discrete domain.
It is said to be \emph{s-holomorphic} if it satisfies the two following properties.
\begin{enumerate}
\item Projection: for every $e = [p_e^- p_e^+] \in \midedgedomain$, we have
\begin{equation}
\label{eqn:proj_prop}
\Proj [g(p_e^-), \tau(e)] = \Proj [g(p_e^+), \tau(e)]
\end{equation}
where $\Proj(X, \tau)$ denotes the projection of $X$ in the direction of $\tau$ :
$$
\Proj[X, \tau] = \frac{1}{2} \left[ X + \frac{\tau}{\overline{\tau}} \cdot
\overline{X} \right].
$$
\item Holomorphicity: for all vertex $e$ on medial lattice $\medialdomain$, we have $\derzbar f(e) = 0$.
\end{enumerate}
\end{defn}

We have a correspondance between s-holomorphic functions on $\medialdomain$ and on $\midedgedomain$.

\begin{prop} \label{prop:s_hol_equiv}
Given a s-holomorphic function $f : \midedgedomain \rightarrow \bbC$, one can define $g : \medialdomain \to
\bbC$ by:
$$
g(p) = f(e_p^-) + f(e_p^+), \quad p \in \medialdomain.
$$
Then, the new function $g$ is still s-holomorphic.

Conversely, given a s-holomorphic function $g : \medialdomain \rightarrow \bbC$, one can define $f : \midedgedomain \to \bbC$ by:
$$
f(e) = \Proj [g(p_e^-), \tau(e)] = \Proj [g(p_e^+), \tau(e)], \quad e \in \midedgedomain.
$$
Then, the new function $f$ is still s-holomorphic.
\end{prop}

\begin{proof}
Assume that $f : \midedgedomain \rightarrow \bbC$ is s-holomorphic.
Let us show that $g$ as defined above is s-holomorphic on $\medialdomain$.
The projection property is satisfied from the parallelism of $f$ and so is the holomorphicity.

Assume that $g : \medialdomain \rightarrow \bbC$ is s-holomorphic.
Let us show that $f$ as defined above is s-holomorphic on $\midedgedomain$.
The parallelism is clearly satisfied by the definition.
We just need to check the holomorphity of $f$.
Let $e \in \midedgedomain$.
We can assume that $p_e^+$ is primal and $p_e^-$ dual such that $\tau(e) \parallel e^{\icomp \pi/4}$.
We want to calculate $\dy f(e)$.
\begin{align*}
\dy f(e) & = \dy \Proj[g(p), \tau(e)] \\
& = \frac12 \dy [g(p) + \icomp \overline{g(p)}] \\
& = \frac12 \left[ \frac{\icomp}{\delta} (g(p^+) - g(p^-)) + \icomp \left( -\frac\icomp\delta \right) (\overline{g(p^+)} - \overline{g(p^-)}) \right] \\
& = \frac\icomp\delta [f(e^+) - f(e^-)]
\end{align*}
where we use $\tau(e) / \overline{\tau(e)} = \icomp$ and $\tau(e^+) / \overline{\tau(e^+)} = \tau(e^-) / \overline{\tau(e^-)} = -\icomp$.
\end{proof}

%% file: obs2.tex
\section{Observable on semi-discrete lattice} \label{sec:observables}

\subsection{Definition and illustration}

Let us take a Dobrushin domain $(\Omega, a, b)$ in $\bbR^2$.
Consider $\delta > 0$ and the semi-discretized domain $\Dobrushindomain$ with mesh size $\delta$, on which we put the loop representation of the critical quantum Ising model (Equation \eqref{eqn:measure_loop_rep}) with parameter $\rho$, which is the density of Poisson point processes on both primal and dual vertical lines.
Here, we choose $\rho$ to be proportional to $1 / \delta$ so that the model is not degenerated when we take the limit $\delta \rightarrow 0$.
Actually, we take $\rho = \frac{1}{\sqrt2 \delta}$, the constant $\frac{1}{\sqrt2}$ being chosen to make the model isotropic, in the sense that we get the correct multiplicative constant in the relation of (s-)holomorphicity.

The loop representation of the quantum Ising model gives an interface going from $a_\delta$ to $b_\delta$.
If $e \in \midedgedomain$ is a mid-edge vertex of the Dobrushin domain $\Dobrushindomain$, we can define our observable at this point by
\begin{equation} \label{eqn:observable}
\obs (e) := F_{\Dobrushindomain} (e)
   = \frac{\nu}{\sqrt{\delta}} \cdot \bbE \left[ \exp \left( \frac\icomp2
      W(e, b_\delta) \right) \mathbbm{1}_{e \in \interface} \right]
\end{equation}
where $\interface$ denotes the (random) interface going from $a_\delta$ to $b_\delta$ and $W(e, b_\delta)$ its winding from $e$ to $b_\delta$ and $\nu = \exp(- \icomp \pi / 4)$.

\begin{rmk}
For the readers who might have read \cite{CS-universality}, since here the graph is oriented differently, the multiplicative factor $\nu$ is chosen so that we can keep the same notations for properties that follow later.
\end{rmk}

\begin{rmk} \label{rmk:obs_parallel}
Since the domain we consider here is simply connected, the winding $W(e, e_b^\delta)$ for a mid-edge vertex $e$ on the boundary does not depend on the random configuration.
We have two cases:
\begin{itemize}
\item If $p_e^- \in \primaldomain$ and $p_e^+ \in \dualdomain$, then the winding $W(e, b_\delta)$ is a multiple of $2 \pi$ and $\obs(e)$ is parallel to $\nu$.
\item If $p_e^- \in \dualdomain$ and $p_e^+ \in \primaldomain$, then the winding $W(e, b_\delta)$ is a multiple of $2 \pi$ plus $\pi$ and $\obs(e)$ is parallel to $\icomp \nu$.
\end{itemize}
This says that $\obs$ satisfies the property of parallelism.
\end{rmk}

\begin{figure}[htb] \centering
    \includegraphics[scale=0.8]{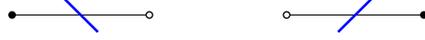}
    \caption{Local relative position of primal / dual vertices with the direction of $\obs$ in blue.}
    \label{fig:obs_parallel}
\end{figure}

We can notice that the winding at $e_b$ is $W(e_b, e_b) = 0$, thus $\obs (e_b) = \frac{\nu}{\sqrt{\delta}}$, which is called the \emph{normalizing constant}.

We then define the observable on $\medialdomain$ for all $p \in \medialdomain$ by
\begin{equation} \label{eqn:medial_observable}
\obsmedial(p) = \obs(p^+) + \obs(p^-).
\end{equation}
If $p \in \partial \medialdomain$, one of $p^+$ and $p^-$ is not defined.
We then take the undefined term to be $0$.
As such, we define $\obsmedial$ everywhere on $\medialdomain$.
We notice that $\obsmedial$ satisfies the projection property \eqref{eqn:proj_prop}.

Let $p$ be a primal or dual point on the arcs $\arcab$ and $\arcba$.
We denote by $\tau(p)$ the tangent vector to $\borderprimaldomain$ oriented from $b_\delta^b$ to $a_\delta^b$ if $p$ is on $\arcab$, and oriented from $b_\delta^w$ to $a_\delta^w$ if $p$ is on $\arcba$.

\begin{prop} \label{prop:obs_border}
For $p \in \arcab \cup \arcba$, we have $\obsmedial(p) \parallel \tau(p)^{-1/2}$.
\end{prop}

\begin{proof}
We can assume that $p \in \arcab$ since the proof is similar for $p \in \arcba$.
In this case, we get two types of tangent vector: $\tau(p)$ is horizontal when $p$ is a dual vertex and vertical when $p$ is a primal vertex.
\begin{enumerate}
\item Assume that the tangent vector $\tau(p)$ is vertical.
We may assume that $\tau(p)$ is oriented from right to left, then the paths counted in $\obs(e_p^+)$ are exactly those counted in $\obs(e_p^-)$, because the interface going through $e_p^+$ is forced to turn left and go through $e_p^-$.
Thus, the observable $\obsmedial(p)$ at $p$ can be written as
$$
\obsmedial(p) = (1-\icomp) \obs(e_p^+).
$$
The quantity $\obs(e_p^+)$ being parallel to $\icomp \nu$, we have $\obsmedial(p) \in \bbR$.
Also, we know that $\tau(p)^{-1/2}$ is parallel to $1$.
The case where $\tau(p)$ is oriented from left to right can be treated in the same way.
\item Assume that the tangent vector $\tau(p)$ is horizontal
which can be oriented either upwards or downwards.
If $\tau(p)$ is oriented downwards, $\obsmedial(p)$ takes the same value as $\obs(e_p^-)$, which belongs to $\icomp \nu \bbR$ due to Remark \ref{rmk:obs_parallel}.
Moreover, $\tau(p)^{-1/2}$ is parallel to $\icomp \nu$.
It is similar if $\tau(p)$ is oriented upwards.
\end{enumerate}
\end{proof}

\begin{figure}[htb] \centering
    \includegraphics[scale=0.8, page=5]{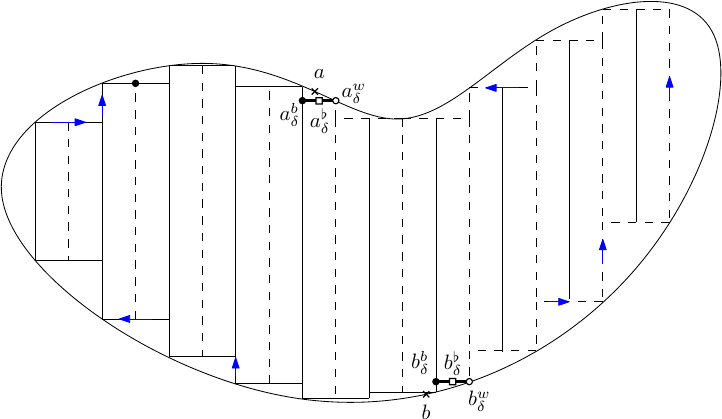}
    \caption{The Dobrushin domain shown in Figure \ref{fig:ex_dob} with tangent vectors $\tau$ drawn in blue with arrows on the boundary.}
    \label{fig:ex_tau}
\end{figure}


\subsection{Relations and holomorphicity} \label{sec:obs_relations}

To study the observables $\obs$ and $\obsmedial$, we need to establish local bijections between configurations which will give us local relations for $\obs$ and $\obsmedial$.
Our goal is to get a relation between the observable $\obs$ and its derivative $\dy \obs$.

To do so, we will fix a \emph{local window} with height equal to $\epsilon$ and width covering three columns, as shown in Figure \ref{fig:edges_not}.
We notice that, in the loop representation, if we reverse primal and dual axes, the loops and interfaces will also reverse their paths.
By studying the difference between the contribution of the term in the expectation in Equation \eqref{eqn:observable}, and by making $\epsilon$ go to 0, we will get the derivative.
Since the number of points given by point Poisson processes is proportional to the length of the interval, thus to $\epsilon$ here, in the limit, only the first order term in $\epsilon$ counts.
In consequence, we only need to constraint ourselves to configurations with at most one Poisson point in the local window.

Some abbreviations will be introduced to lighten our notations.
We denote the north-/south- west/middle/east mid-edge vertices by taking their initials: $nw$, $nm$, $ne$, $sw$, $sm$ and $se$.
We denote the primal extremity shared by $nw$ and $nm$ by $b_n$ and the one shared by $sw$ and $sm$ by $b_s$.
Similarly for $w_n$ and $w_s$.
See again Figure \ref{fig:edges_not}.

\begin{figure}[htb] \centering
  \includegraphics[scale=1]{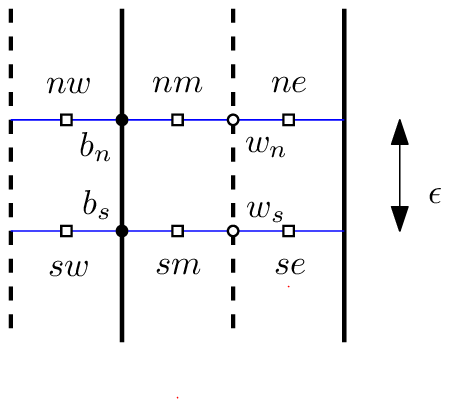}
  \caption{A local window with height $\epsilon$. The same notations are used in the following figures and tables of this section.}
  \label{fig:edges_not}
\end{figure}

To understand the bijections between configurations, the reader is invited to have a look at Figure \ref{fig:configs_bijection} while reading the following explanation.
The bijections are obtained by starting with an interface going through the middle column, which is not a loss of generality.
In our case, it goes down due to the choice of the local window.
We assume that there is not any Poisson points in this local window.

We will then analyze different possibilities.
Once the interface goes out of the local window, it may never come back to the neighboring mid-edge axes (\emph{i.e.} west and east), which is the case of (1a).
Otherwise, the interface may come back to one of the neighboring mid-edge axes.
In (2a), it comes back through the East column and in (3a), through the West column.

Now we can consider Poisson points in our configurations.
As we mentioned earlier, we are only interested in configurations with at most one such point.
In (1b), (2b) and (3b), we add one Poisson point between $b_s$ and $b_n$ whereas in (1c), (2c) and (3c), we add one between $w_s$ and $w_n$.
The configurations (1a), (1b) and (1c) are in bijection, same for (2a), (2b) and (2c) or (3a), (3b) and (3c).
Notice that these configurations do not have the same weight, but we know the ratio between their weights, which will allow us to get linear relations between the contribution of $\exp(\frac\icomp2 W(e, b_\delta)) \mathbbm{1}_{e \in \interface}$ to $f$ at $nw$, $sw$, $nm$, $sm$, $ne$ and $se$.

\begin{figure}[htb] \centering
  \includegraphics[scale=0.75]{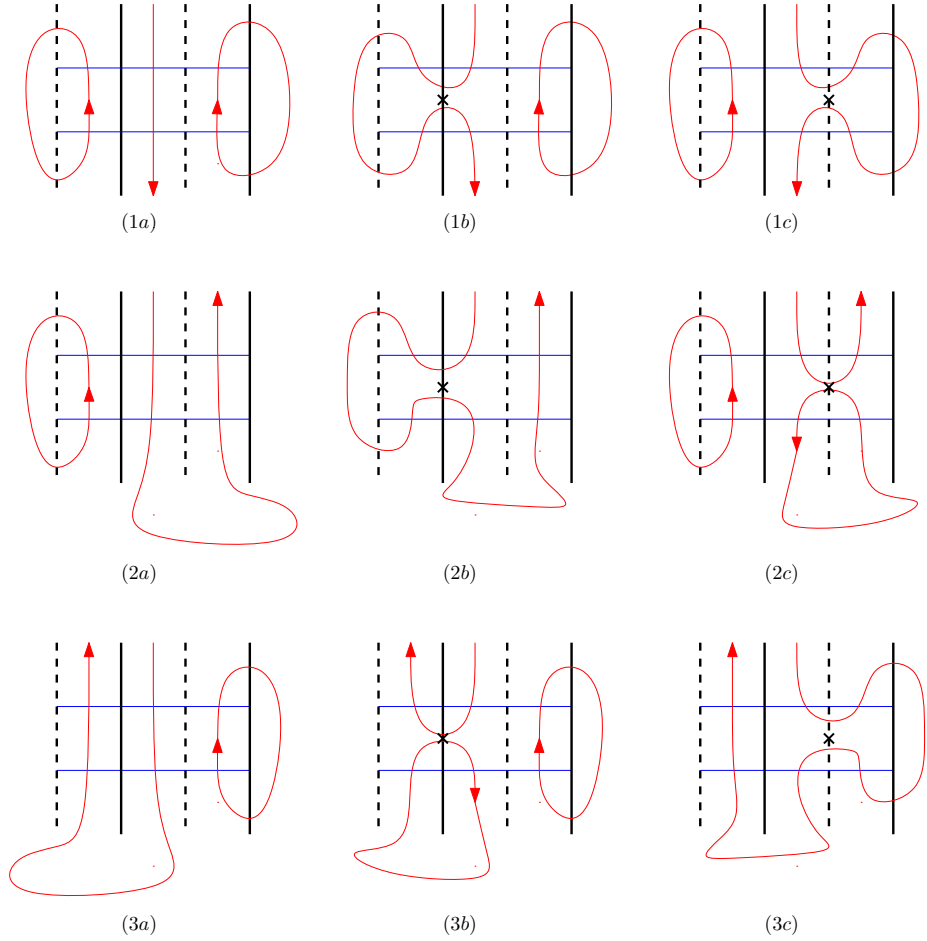}
  \caption{Bijection between configurations in a local window chosen above.}
  \label{fig:configs_bijection}
\end{figure}

We first establish Table \ref{table:first_quantities} containing different contributions.
The last columns contain the weight of each configuration up to a multiplicative constant depending on the original configuration in (1a), (2a) and (3a).
However, the fact that this multiplicative constant is unknown does not raise any difficulty since we only need linear relations between values of $\obs$ at different mid-edge vertices.

\begin{table}[htb] \centering
\begin{tabular}{c|ccccccc} \hline & $nw$ & $sw$ & $nm$ & $sm$ & $ne$ & $se$ & weights \\
\hline \hline
1a & 0 & 0 & 1 & 1 & 0 & 0 & $\sqrt{2}$ \\
1b & $e^{\icomp\frac{\pi}{2}}$ & $e^{-\icomp\frac{\pi}{2}}$ & 1 & 1 & 0 & 0 & $\epsilon \rho$ \\
1c & 0 & 0 & 1 & 1 & $e^{-i\frac{\pi}{2}}$ & $e^{i\frac{\pi}{2}}$ & $\epsilon \rho$ \\
\hline
2a & 0 & 0 & 1 & 1 & $e^{-i\frac{\pi}{2}}$ & $e^{-i\frac{\pi}{2}}$ & $\sqrt{2}$ \\
2b & $e^{i\frac{\pi}{2}}$ & $e^{-i\frac{\pi}{2}}$ & 1 & 1 & $e^{-i\frac{\pi}{2}}$ & $e^{-i\frac{\pi}{2}}$ & $\epsilon \rho$ \\
2c & 0 & 0 & 1 & 0 & $e^{-i\frac{\pi}{2}}$ & 0 & $2 \epsilon \rho$ \\
\hline
3a & $e^{i \frac\pi2}$ & $e^{i \frac\pi2}$ & 1 & 1 & 0 & 0 & $\sqrt{2}$ \\
3b & $e^{i \frac\pi2}$ & 0 & 1 & 0 & 0 & 0 & $2 \epsilon \rho$ \\
3c & $e^{i \frac\pi2}$ & $e^{i \frac\pi2}$ & 1 & 1 & $e^{-i \frac\pi2}$ & $e^{i \frac\pi2}$ & $\epsilon \rho$ \\
\hline
\end{tabular}
\caption{Contributions of the exponential term in each configuration at different positions.}
\label{table:first_quantities}
\end{table}

We take the difference of contributions between the north mid-edge and the south mid-edge in each of the three columns to get Table \ref{table:differences}.
After this, we get $\obs(w)$, $\obs(m)$ and $\obs(e)$ from Table \ref{table:first_quantities} by ignoring terms of
order higher than $\epsilon$; and $\dy \obs(w)$, $\dy \obs(m)$ and $\dy \obs(e)$ by dividing the quantities in Table \ref{table:differences} by $\epsilon$ and then making it go to 0.

\begin{table}[htb]
  \renewcommand\baselinestretch{2}\selectfont
  \begin{minipage}{.48\textwidth}
    \centering
    \begin{tabular}{c||c|c|c}
      \hline
      & $nw-sw$ & $nm-sm$ & $ne-se$ \\
      \hline
      \hline
      1 & $2i \epsilon \rho$ & 0 & $-2i \epsilon \rho$ \\
      \hline
      2 & $2i \epsilon \rho$ & $2 \epsilon \rho$ & $- 2i \epsilon \rho$ \\
      \hline
      3 & $2i \epsilon \rho$ & $2 \epsilon \rho$ & $- 2i \epsilon \rho$ \\
      \hline
    \end{tabular}
    \renewcommand\baselinestretch{1}\selectfont
    \caption{Computation of the difference between the contributions of North and South.}
    \label{table:differences}
  \end{minipage}
  \renewcommand\baselinestretch{1}\selectfont
  \hfill
  \begin{minipage}{.48\textwidth}
    \centering
    \begin{tabular}{c||c|c|c}
      \hline
      & $\obs(w)$ & $\obs(m)$ & $\obs(e)$ \\
      & $\dy \obs(w)$ & $\dy \obs(m)$ & $\dy \obs(e)$ \\
      \hline
      \hline
      \multirow{2}{*}{1} & 0 & $\sqrt{2}$ & 0 \\
      & $2i \rho$ & 0 & $-2i \rho$ \\
      \hline
      \multirow{2}{*}{2} & 0 & $\sqrt{2}$ & $- \icomp \sqrt{2}$ \\
      & $2i \rho$ & $2 \rho$ & $- 2i \rho$ \\
      \hline
      \multirow{2}{*}{3} & $\icomp \sqrt{2}$ & $\sqrt{2}$ & 0 \\
      & $2i \rho$ & $2 \rho$ & $- 2i \rho$ \\
      \hline
    \end{tabular}
    \caption{By considering order 0 and order 1 terms in $\epsilon$, we get $\obs$ and $\dy \obs$.}
    \label{table:final_quantities}
  \end{minipage}
\end{table}

The quantities in the first and the second lines of Table \ref{table:final_quantities} satisfy
\begin{equation}
(\obs(e) - \obs(w)) \cdot \icomp \sqrt{2} \rho = \dy \obs(m).
\end{equation}
Moreover, those in the first and the third line satisfy also the same equation.
By summing over the possible local configurations and by gathering them together, we obtain that for each $m \in \midedgedomain$
\begin{equation} \label{eqn:holo}
\derzbar \obs (m) = \frac12 \left[ \frac{\obs(e)-\obs(w)}{\delta} - \frac{\dy \obs(m)}{\icomp} \right] = 0.
\end{equation}

Gathering all the above computations and using Proposition \ref{prop:s_hol_equiv}, we obtain the following proposition.
\begin{prop} \label{prop:obs_prop}
The observables $\obs$ and $\obsmedial$ satisfy the following properties.
\begin{enumerate}
\item The observable $\obs$ is s-holomorphic on $\midedgedomain$.
\item The observable $\obsmedial$ is s-holomorphic on $\medialdomain$.
\end{enumerate}
\end{prop}

By Proposition \ref{prop:s_hol_equiv}, the observables $\obs$ and $\obsmedial$ encode the same amount of information.
We will then sometimes work with $\obs$, sometimes with $\obsmedial$, according to our convenience.

\subsection{Primitive of $\obsmedial^2$} \label{sec:obs_primitive}

We will be interested in Riemann-Hilbert boundary value problem in Section \ref{sec:BVP}.
To solve this problem in continuum, we make use of the fact that the function $h = \myi \int f^2$, where $f$ is a solution, is harmonic.
Therefore, in the semi-discrete setting, we try to make sense of a primitive of $\obsmedial^2$ then show that it is not far from being harmonic.
This will be illustrated in Section \ref{sec:CVG_thm}.

Given $v$ a site on the lattice, we denote by $e_v^+$ and $e_v^-$ the mid-edges having $v$ as extremity on the right side and the left side of $v$, as illustrated in Figure \ref{fig:neighbors}.
In a similar way, we denote by $e_v^{++}$ and $e_v^{--}$ the second on the right or left.

\begin{figure}[htb] \centering
  \includegraphics[scale=1.0]{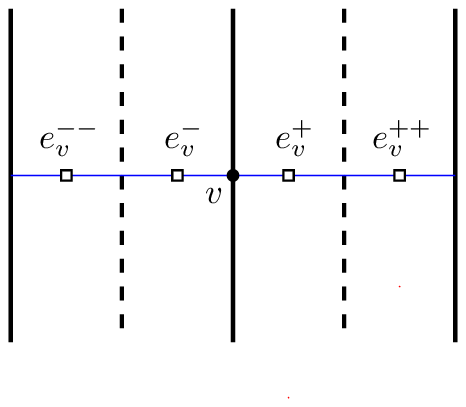}
  \caption{Notations for neighboring mid-edges.}
  \label{fig:neighbors}
\end{figure}

Let us define $\prim$, a ``primitive'' of $\obsmedial^2$ in the following way.
Since $\obsmedial$ and $\obs$ can be related, in the definitions below we only work with $\obs$ first.

\begin{enumerate}
\item If $b$ and $b'$ are primal vertices such that $\myr b = \myr b'$ and $[b b'] \subset \medialdomain$, define
\begin{equation} \label{eqn:H_primal}
\prim(b') - \prim(b) = 2 \cdot \myi \int_b^{b'} \obs(e_v^-) \overline{\obs}(e_v^+) \dd v
\end{equation}
\item If $w$ and $w'$ are dual vertices such that $\myr w = \myr w'$ and $[w w'] \subset \medialdomain$, define
\begin{equation} \label{eqn:H_dual}
\prim(w') - \prim(w) = - 2 \cdot \myi \int_w^{w'} \obs(e_v^-) \overline{\obs}(e_v^+) \dd v.
\end{equation}
\item If $b$ and $w$ are neighboring primal and dual vertices in $\medialdomain$, define
\begin{equation} \label{eqn:H_edge}
\prim(b) - \prim(w) = \delta |\obs(bw)|^2.
\end{equation}
\end{enumerate}

\begin{prop}
The primitive $\prim$ is well-defined up to an additive constant.
\end{prop}

\begin{proof} It is sufficient to check that the difference of $\prim$ along cycles is always 0, or equivalently, the difference along elementary rectangles is always 0.
Let $u_1$, $u_2$, $v_2$, $v_1$ be a rectangle as shown in Figure \ref{fig:rec}.
We denote by $e_{m_i}$ the mid-edge between $u_i$ and $v_i$ for $i = 1, 2$.

\begin{figure}[htb] \centering
  \includegraphics[scale=1.0]{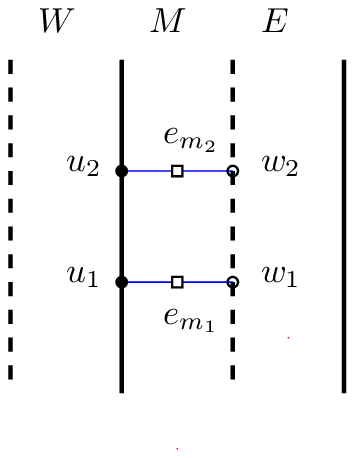}
  \caption{An elementary rectangle $u_1 u_2 w_2 w_1$.}
  \label{fig:rec}
\end{figure}

We need to show that the difference of $\prim$ along $u_1u_2$, $u_2w_2$, $w_2w_1$ and $w_1u_1$ gives 0.
We apply directly the definitions from Equations \eqref{eqn:H_primal}, \eqref{eqn:H_dual} and \eqref{eqn:H_edge}.

\begin{align*}
& \prim(u_2) - \prim(u_1) + \prim(w_2) - \prim(u_2)
 + \prim(w_1) - \prim(w_2) + \prim(u_1) - \prim(w_1) \\
& = 2 \myi \int_{u_1}^{u_2} \obs(e_v^-) \overline{\obs}(e_v^+) \dd v
 - \delta |\obs(u_2w_2)|^2 
 + 2 \myi \int_{w_1}^{w_2} \obs(e_v^-) \overline{\obs}(e_v^+) \dd v
 + \delta |\obs(u_1w_1)|^2 \\
& = 2 \myi \int_{m_1}^{m_2} [\obs(e_m^-) \overline{\obs}(e_m) + \obs(e_m) \overline{\obs}(e_m^+) ] \dd m
 - \delta (|\obs(u_2 w_2)|^2 - |\obs(u_1 w_1)|^2).
\end{align*}

The first term can be rewritten as:
\begin{align*}
2 \myi \int_{m_1}^{m_2} [\obs(e_m^-) \overline{\obs}(e_m) + \obs(e_m) \overline{\obs}(e_m^+) ] \dd m
= 2 \myi \int_{m_1}^{m_2} [\obs(e_m^-) - \obs(e_m^+) ] \overline{\obs}(e_m) \dd m,
\end{align*}
and the second term as:
\begin{align*}
- \delta (|\obs(u_2 w_2)|^2 - |\obs(u_1 w_1)|^2)
& = - \delta \myr \int_{m_1}^{m_2} 2 \dy \obs(e_m) \overline{\obs}(e_m) \dd m \\
& = - \myr \int_{m_1}^{m_2} 2 \icomp (\obs(e_m^+) - \obs(e_m^-)) \overline{\obs}(e_m) \dd m \\
& = 2 \myi \int_{m_1}^{m_2} [\obs(e_m^+) - \obs(e_m^-)] \overline{\obs}(e_m) \dd m \\
\end{align*}
where we use the holomorphic relation \eqref{eqn:holo} in the second line.
Thus, the quantity we were looking for is indeed 0.
\end{proof}

From the previous proposition, we can fix $\prim(b_\delta^w)$ to be zero, thus fixing the additive constant.
Now, we can really talk about \emph{the} primitie $\prim$.

Starting from the definitions above, we will reformulate to get relations for $\prim$ between different points on the same axis (Proposition \ref{prop:H_axis}) and neighboring points on medial lattice (Proposition \ref{prop:H_neighbors}).
This will give us a simpler expression in terms of $\obsmedial$: $\prim = \myi \int^\delta (\obs (z))^2 \dd^\delta z$.

\begin{prop} \label{prop:H_axis}
Let $p, p' \in \medialdomain$ such that $\myr p = \myr p'$ and $[p p'] \subset \medialdomain$.
Then we have
\begin{equation} \label{eqn:H_axis}
\prim(p') - \prim(p) = \myi \int_p^{p'} \icomp \cdot (\obsmedial(v))^2 \dd v.
\end{equation}
\end{prop}

\begin{proof}
We first assume that $p, p' \in \primaldomain$.
Given $v \in [p p']$, since $\obs(e_v^-) \in \icomp \nu \bbR$ and $\obs(e_v^+) \in \nu \bbR$, we have
\begin{align*}
\myi \left[ \icomp \cdot \obsmedial(v)^2 \right]
& = \myr \left[ \obs(e_v^-)^2 + \obs(e_v^+)^2 + 2 \obs(e_v^-) \obs(e_v^+) \right] \\
& = 2 \myr \left[ \obs(e_v^-) \obs(e_v^+) \right] \\
& = 2 \myi \left[ \obs(e_v^-) \overline{\obs} (e_v^+) \right].
\end{align*}
The same computation for $p, p' \in \dualdomain$ and $v \in [p p']$ leads to
\begin{align*}
\myi \left[ \icomp \cdot \obsmedial(v)^2 \right]
 = - 2 \myi \left[ \obs(e_v^-) \overline{\obs} (e_v^+) \right].
\end{align*}
Using Equations \eqref{eqn:H_primal} and \eqref{eqn:H_primal}, we get the result.
\end{proof}

\begin{prop} \label{prop:H_neighbors}
Let $p \in \medialdomain$ such that $p^-, p^+ \in \medialdomain$.
Then,
\begin{equation} \label{eqn:H_neighbors}
\prim(p^+) - \prim(p^-) = \myi [ \obsmedial(p)^2 (p^+ - p^-)].
\end{equation}
\end{prop}

\begin{proof}
We can assume that $p \in \dualdomain$ and $p^-, p^+ \in \primaldomain$.
The other case when $p \in \primaldomain$ can be treated in the same way.
 
From the parallelism property, we know that $\obs(e_p^-) \in \nu \bbR$ and $\obs(e_p^+) \in \icomp \nu \bbR$.
A simple computation gives
\begin{align*}
\myi \left[ \obsmedial(p)^2 \right] & = \myi \left[ \obs(e_p^-)^2 + \obs(e_p^+)^2 + 2 \obs(e_p^-) \obs(e_p^+) \right] \\
& = |\obs(e_p^+)|^2 - |\obs(e_p^-)|^2.
\end{align*}
Since $p^+ - p^- = \delta$, this completes the proof.
\end{proof}

\begin{cor}
The primitive $\prim$ is constant on both arcs $\arcab = \arcabfull$ and $\arcba = \arcbafull$.
Moreover,
\begin{equation}
{\prim}_{\mid \arcab} = 1 \quad \mbox{ and } \quad {\prim}_{\mid \arcba} = 0.
\end{equation}
\end{cor}

\begin{proof}
Proposition \ref{prop:obs_border} gives the direction of $\obsmedial$ and Equations \eqref{eqn:H_axis} and \eqref{eqn:H_neighbors} give the relation of $\prim$ on each part of the two arcs.
We conclude easily that $\obsmedial(p)$ is constant on both arcs.
The difference of these constants can be obtained by estimating $\prim$ at, for example, $b_\delta = [b_\delta^b b_\delta^w]$:
$$
\prim(b_\delta^b) = \prim(b_\delta^b) - \prim(b_\delta^w) = \delta |\obs (b_\delta)|^2 = 1.
$$
\end{proof}

\begin{prop} \label{prop:H_supersub_harm}
The primitive $\prim$ is subharmonic on primal axes and superharmonic on dual axes, $i.e.$
$$
\laplacian H (u) \geq 0 \quad \mbox{ and } \quad \laplacian H (w) \leq 0
$$
for all $u \in \primaldomain$ and $w \in \dualdomain$.
\end{prop}

\begin{proof}
We remind that in a semi-discrete lattice, the Laplacian is defined as follows:
$$
\discretelaplacian \prim (u) = \Deltaxx \prim(u) + \ddy \prim(u)
$$
for $u$ a vertex in primal or dual axis.

First, we assume that $u$ is a primal vertex.
By using the definition of $\prim$, the second derivative along $x$ can be reformulated,
$$
\Deltaxx \prim (u) = \frac1\delta \left[ |\obs (e_u^{++})|^2 - |\obs (e_u^{+})|^2
- |\obs (e_u^{-})|^2 + |\obs (e_u^{--})|^2 \right].
$$
Similarly, the second derivative along $y$ can be rewritten as:
\begin{align*}
\partial_{yy} \prim(u) & = 2 \myi [\partial_y (\obs(e_u^-) \overline{\obs}(e_u^+) ) ] \\
& = 2 \myi [ \dy \obs(e_u^-) \overline{\obs}(e_u^+) + \obs(e_u^-) \overline{\dy \obs}(e_u^+) ] \\
& = 2 \myi [ \dy \obs(e_u^-) \overline{\obs}(e_u^+) - \overline{\obs}(e_u^-) \dy \obs(e_u^+) ] \\
& = 2 \myi \left[ \frac\icomp\delta \left[ \obs(e_u^+) - \obs(e_u^{--}) \right] \cdot \overline{\obs}(e_u^+)
- \overline{\obs}(e_u^-) \cdot \frac\icomp\delta \left[ \obs(e_u^{++}) - \obs(e_u^-) \right] \right] \\
& = \frac2\delta \myr [ (\obs(e_u^+) - \obs(e_u^{--})) \overline{\obs}(e_u^+)
- \overline{\obs}(e_u^-) (\obs(e_u^{++}) - \obs(e_u^-)) ] \\
& = \frac2\delta \left[ |\obs(e_u^+)|^2 + |\obs(e_u^-)|^2
- \myr [ \obs(e_u^{--}) \overline{\obs}(e_u^+) + \overline{\obs}(e_u^-) \obs(e_u^{++}) ] \right].
\end{align*}

We also notice that
\begin{align*}
& |\obs(e_u^{++}) - \obs(e_u^-)|^2 + |\obs(e_u^+) - \obs(e_u^{--})|^2 \\
& \qquad \qquad = |\obs(e_u^{++})|^2 + |\obs(e_u^{+})|^2 + |\obs(e_u^{-})|^2 + |\obs(e_u^{--})|^2 \\
& \qquad \qquad \qquad - 2 \myr \left[ \overline{\obs}(e_u^-) \obs(e_u^{++}) + \obs(e_u^{--}) \overline{\obs}(e_u^+) \right] \\
& \qquad \qquad = \delta \Deltaxx \prim (u) + \delta \partial_{yy} \prim(u) \\
& \qquad \qquad = \delta \laplacian \prim(u).
\end{align*}
In consequence, the primitive $H$ is subharmonic on primal axes.

The proof for the superharmonicity on dual axes is similar.
We do the same calculation and obtain the above equation with a minus sign.
\end{proof}

%% file: cvg2.tex
\section{Uniform convergence theorem} \label{sec:CVG_thm}

\subsection{Boundary modification trick} \label{sec:boundary_modif}

A semi-discrete Dobrushin domain can be extended to a (semi-discrete) primal (\emph{resp.} dual) domain.
This technique, called \emph{boundary modification trick}, is presented below.
We can also extend our (semi-discrete) functions on these larger domains, making them easier to study.

The primal (\emph{resp.} dual) domain extended from a Dobrushin domain is given by keeping the primal boundary $\borderab$ and by adding an extra layer $\extborderba$ to the dual boundary $\borderba$.
More precisely, on the arc $\borderba$ we change the horizontal parts from dual to primal and add one more primal layer \emph{outside} (defined below) the original domain.
The same procedure applies similarly if we want to get an extended dual domain: we get $\extborderab$ from $\borderab$ and keep $\extborderba$.
We will denote by $\modifiedprimaldomain$ and $\modifieddualdomain$ these two modified domains.
See Figure \ref{fig:dob_primal} and Figure \ref{fig:dob_dual} for examples.

Each dual point $p \in \arcba$ possesses two primal neighbors $p^-$ and $p^+$.
One of them is in $\Int\medialdomain$ and the other is not (although it may lie on the boundary $\borderab$).
We include the one which is not in $\Int \medialdomain$, thus providing us with the new boundary.

\begin{figure}[htb] \centering
  \begin{minipage}{.44\textwidth}
    \includegraphics[scale=0.75, page=2]{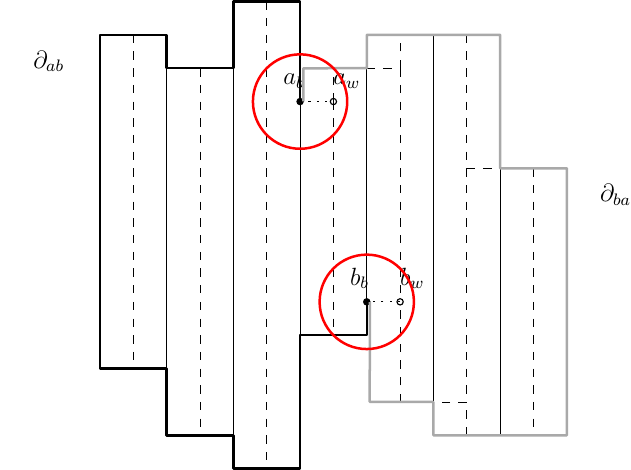}
    \caption{The primal domain extended from the Dobrushin domain given in Figure \ref{fig:dob}. The red part indicates the overlapping part of the extended boundary with the arc $\arcab$.}
    \label{fig:dob_primal}
  \end{minipage}
  \hfill
  \begin{minipage}{.44\textwidth}
    \includegraphics[scale=0.75, page=3]{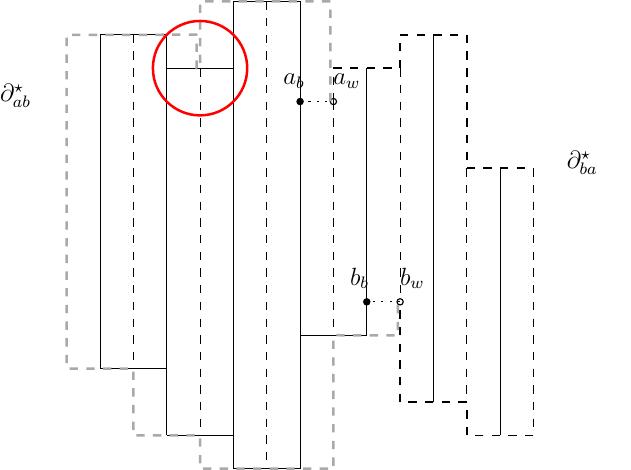}
    \caption{The dual domain extended from a Dobrushin domain given in Figure \ref{fig:dob}. The red part represents the extended boundary overlapping itself.}
    \label{fig:dob_dual}
  \end{minipage}
\end{figure}

We notice that some points may be added twice (red part in Figure \ref{fig:dob_dual}) and some points may overlap the other boundary (red part in Figure \ref{fig:dob_primal}).
Thus, the boundary of the extended domain is not described by a Jordan curve anymore.
However, this is not a problem: we just keep these double points and consider that they are situated on the two different sides of the same boundary and all the theorems concerning boundary value problem will still be valid.
We can also see this as a domain minus a slit.

The following lemma tells us how to extend the primitive $\prim$ to the extended domain after boundary modification trick.

\begin{lem} \label{lem:extend_H}
Let $w \in \duallattice \cap \arcba$ be a dual vertex on the arc $\arcba$.
Assume $u_{int}$ to be the neighboring primal vertex of $w$ which is in the domain $\medialdomain$ and $u_{ext}$ the primal one to be added via boundary modification trick.
Then, if we set $\prim (u_{ext}) = H(w)$, the function $\prim$ remains subharmonic at $u_{int}$.
We can also extend $H$ on $\modifieddualdomain$ in a similar way, keeping a similar result.
\end{lem}

\begin{proof}
By abusing the notation, we continue writing $\obs$, $\obsmedial$ and $\prim$ on the extended domain $\modifiedprimaldomain$.
We notice that if we let $\obs(u_{ext} w) = 0$ and $\obsmedial(w) = \obs(w u_{int})$, the properties in Proposition \ref{prop:obs_prop} are still satisfied.
This can be computed by establishing a similar table as Table \ref{table:final_quantities} on the boundary.
Then, by setting $\prim(u_{ext}) = H(w)$, we get a primitive $\prim$ which always satisfies Equations \eqref{eqn:H_dual} and \eqref{eqn:H_edge}.
In such a way, the Proposition \ref{prop:H_supersub_harm} still holds.
\end{proof}

\subsection{Riemann-Hilbert boundary value problem} \label{sec:BVP}

We have studied Dirichlet problem in Section \ref{sec:dirichlet}.
The Riemann-Hilbert boundary value problem we are going to introduce is similar to this but its resolution is a bit more complicated.
We will study it by making some links to the Dirichlet problem.

In a semi-discrete Dobrushin domain $\Dobrushindomain$, we say that a function $\obs$ defined on $\medialdomain$ is a solution to the \emph{(semi-discrete) boundary value problem} with respect to the Dobrushin domain $\Dobrushindomain$ if the following three conditions are satisfied:

\begin{enumerate}[(A)]
\item s-holomorphicity: $\obsmedial$ is s-holomorphic in $\medialdomain$;
\item boundary conditions: for $p \in \arcab \cup \arcba$, $\obsmedial(p)$
is parallel to $\tau(p)^{-1/2}$;
\item normalization: $\obs (e_b^\delta) = \Proj[ \obsmedial (b_\delta^w), \nu] = \frac{\nu}{\sqrt{\delta}}$.
\end{enumerate}

Existence of such a solution has been shown already.
In effect, the observable we introduced earlier satisfies these three conditions, as shown in Section \ref{sec:obs_relations}.
When it comes to uniqueness, we will use the primitive $H$ we constructed in Section \ref{sec:obs_primitive} along with the boundary modification trick.

\begin{prop}[Existence of solution]
The observable $\obsmedial$ given by \eqref{eqn:observable} is a solution to the above boundary value problem.
\end{prop}

\begin{proof}
It is direct from the properties of the observable as shown in Propositions \ref{prop:obs_prop} and \ref{prop:obs_border}.
\end{proof}

\begin{prop}[Uniqueness of solution]
For each semi-discrete Dobrushin domain $\Dobrushindomain$, the semi-discrete boundary value problem has a unique solution.
\end{prop}

\begin{proof}
Assume that there are two solutions $\calF_{\delta, 1}$ and $\calF_{\delta, 2}$ to the boundary value problem mentioned above.
Let $\calF_\delta := \calF_{\delta, 1} - \calF_{\delta, 2}$.
Notice that $\calF_\delta$ is still s-holomorphic being difference of two such functions.
Consider $H_\delta := \myi \int (\calF_\delta(z))^2 dz$ the primitive defined in Section \ref{sec:obs_primitive}.
The function $H_\delta$ is constant on the arcs $\arcab$ and $\arcba$ respectively.
Moreover, the identity $\calF_\delta (b_\delta) = 0$ says that these two constants should be the same.
Apply the boundary modification trick to extend $\medialdomain$ into the primal domain $\modifiedmedialdomain$.

Extend the primitive $\prim$ to the new boundary of $\modifiedmedialdomain$ as in Lemma \ref{lem:extend_H}.
The Lemma also says that $\prim$ stays subharmonic in $\primallattice \cap \modifiedmedialdomain$ and subharmonic in $\duallattice \cap \modifiedmedialdomain$.
Then, we have $0 \geq (\prim)_{\mid \modifiedprimaldomain} \geq (\prim)_{\mid \modifieddualdomain} \geq 0$ by uniqueness of Dirichlet problem (Proposition \ref{prop:uniqueness_Dirichlet}), $\prim$ is constant everywhere.

The fact that $\prim$ is constant everywhere on $\modifiedmedialdomain$ tells us that $\calF_\delta$ is zero everywhere on $\midedgedomain$.
Thus, these two solutions must be equal.
\end{proof}


\subsection{Convergence theorem}

\begin{thm}[Convergence theorem for \emph{s-holomorphic} functions] \label{thm:cvg_shol}
Let $Q \subset \Omega$ be a rectangular domain such that $9Q \subset \Omega$.
Let $(\obsmedial)_{\delta > 0}$ be a family of s-holomorphic functions on $\medialdomain$ and $\prim = \myi \int \obsmedial^2$.
If $(\prim)_{\delta>0}$ is uniformly bounded on $9Q$, then $(\obsmedial)$ is precompact on $Q$.
\end{thm}

\begin{rmk}
For each $z \in \Int \Omega$, we can find a neighborhood $Q$ of $z$ small enough such that $9Q \subset \Omega$ to have precompactness of $(\obsmedial)$ near $z$.
Then we can use a diagonal argument to extract a subsequence of $(\obsmedial)$ converging uniformly on all compacts of $\Omega$.
\end{rmk}

\begin{proof}
It is sufficient to show the second point in Theorem \ref{thm:cvg_harm}.
We write
$$
\delta \int_{\Qmedial} |\obsmedial(v)|^2 \dd v
= \delta \int_{\Qprimal} |\Deltax \prim(x)| \dd x
+ \delta \int_{\Qdual} |\Deltax \prim(x)| \dd x
$$
which is exactly the definition of $\prim$ in Proposition \ref{prop:H_neighbors}.
These two terms can be treated in a similar way.
We will thus just look at the first one and show that it is bounded by a constant uniformly in $\delta$.
On the primal semi-discretized domain $9\Qprimal$, write $\prim = S_\delta + R_\delta$ where $S_\delta$ is semi-discrete harmonic with boundary values $\prim_{|\partial 9\Qprimal}$ on $\partial 9 \Qprimal$.

\begin{align*}
\int_{\Qprimal} |\Deltax S_\delta (x)| \dd x &
\leq \int_{\Qprimal} c_1 \cdot \sup_{9\Qprimal} |S_\delta|
\leq \frac{c_2}{\delta} \cdot c_1 \cdot \sup_{9\Qprimal} |S_\delta| \\
& = \frac{c_1 c_2}{\delta} \sup_{\partial 9\Qprimal} |S_\delta|
= \frac{c_3}{\delta} \sup_{\partial 9 \Qprimal} |H_\delta| \leq \frac{c_4}{\delta}.
\end{align*}
Here, we use Proposition \ref{prop:Harnack} in the first inequality; the total length of axes in $Q_\delta$ is proportional to $\delta^{-1}$ in the second; the maximum principle (Proposition \ref{prop:max_principle}) in the third; $H_\delta$ and $S_\delta$ coincide on the boundary $\partial 9 \Qprimal$ in the fourth ; and finally, $H_\delta$ is bounded by hypotheses.
Moreover, the constants $c_i$ may depend on the domain $\Omega$ but are uniform in $\delta$.

We will now do something similar to $R_\delta$.
First, we write (Proposition \ref{prop:Riesz})
$$
R_\delta(x) = \int_{9 \Qprimal} \laplacian R_\delta(y) G_{9\Qprimal} (x, y) \dd y.
$$
Since $H$ is subharmonic, it is the same for $R_\delta$.
Thus, $\laplacian R_\delta \geq 0$ in $9\Qprimal$.
Then, we have
\begin{align*}
\int_{\Qprimal} |\Deltax R_\delta(x)| \dd x
& \leq \int_{\Qprimal} \int_{9\Qprimal} \laplacian R_\delta (y) |\Deltax G_{9\Qprimal}(x, y)| \dd y \dd x \\
& = \int_{9\Qprimal} \laplacian R_\delta (y) \int_{\Qprimal} |\Deltax G_{9\Qprimal}(x, y)| \dd x \dd y \\
& \leq \int_{9\Qprimal} c_5 \cdot \laplacian R_\delta(y) \int_{\Qprimal} G_{9\Qprimal}(x, y) \dd x \dd y \\
& = c_5 \int_{\Qprimal} \int_{9\Qprimal} \laplacian R_\delta (y) G_{9\Qprimal}(x, y) \dd y \dd x \\
& = c_5 \int_{\Qprimal} R_\delta(x) \dd x \\
& \leq c_5 \cdot \frac{c_6}{\delta} = \frac{c_5 c_6}{\delta}
\end{align*}
where we use the triangular inequality in the first line; Fubini in the second line (all the terms are non-negative); Proposition \ref{prop:estimation_Green's} in the third; Fubini again in the fourth, Riesz representation (Proposition \ref{prop:Riesz}) again in the fifth; and finally $R_\delta$ is bounded in the last one (because $H_\delta$ and $S_\delta$ are bounded).
\end{proof}

With all what we have done so far, we can determine the uniform limit of $\prim$ and $\obsmedial$ when $\delta$ goes to 0.
First of all, we need to describe the continuous version of the boundary-valued problem.
Given a continuous Dobrushin domain $(\Omega, a, b)$, we say that a function $f$ defined on $\Omega$ is a solution to the \emph{boundary-valued problem} if
\begin{enumerate}[(a)]
\item holomorphicity: $f$ is holomorphic in $\Omega$ with singularities at $a$ and $b$;
\item boundary conditions: $f(\zeta)$ is parallel to $\tau(\zeta)^{-1/2}$ for $\zeta \in \partial \Omega \backslash \{ a, b \}$, where $\tau(\zeta)$ denotes the tangent vector to $\Omega$ oriented from $a$ to $b$ (on both arcs);
\item normalization: the function $h := \myi \int (f(\zeta))^2 \dd \zeta$ is uniformly bounded in $\Omega$ and
$$
h_{\mid \arcab} = 0, \quad h_{\mid \arcba} = 1.
$$
\end{enumerate}

Note that (a) and (b) guarantee that $h$ is harmonic in $\Omega$ and constant on both boundary arcs $(ab)$ and $(ba)$.
Thus, if we write $\Phi$ the conformal mapping from $\Omega$ onto the infinite strip $\bbR \times (0, 1)$ sending $a$ and $b$ to $\mp \infty$, the function $h \circ \Phi^{-1}$ is still harmonic.
Moreover, the harmonic function on the strip $\bbR \times (0, 1)$ with boundary condition 1 on $\bbR \times \{ 1 \}$ and 0 on $\bbR \times \{ 0 \}$ is $\myi(z)$, we obtain that $h(z) = \myi \Phi(z)$
And from the definition of $h$ in (c), we get
$$
h(v) - h(u) = \myi (\Phi(v) - \Phi(u)) = \myi \int_u^v (f(\zeta))^2 \dd \zeta
$$
for $u, v \in \Omega$.
At $u$ fixed, since $\Phi(v) - \Phi(u)$ and $\int_u^v (f(\zeta))^2 \dd \zeta$ are both holomorphic in $v$ and have the same imaginary part, they differ only by a real constant.
By taking the derivative, we can deduce that
$$
\Phi'(v) = f(v)^2
$$
or equivalently,
$$
f = \sqrt{\Phi'}.
$$
Since $\Phi$ is a conformal map, its derivative is never 0 on $\Omega$, we can define the square root in a continuous manner (with respect to $\Omega$), and the solution $f$ is well-defined up to the sign.
Moreover, this tells us that $f(\zeta) = c(\zeta) \tau(\zeta)^{-1/2}$ for all $\zeta \in \partial \Omega \backslash \{ a, b \}$ where $c$ keeps the same sign all along the boundary.
Therefore, we can choose the branch of the logarithm such that $\sqrt{\Phi'}$ corresponds to $c$ positive and $-\sqrt{\Phi'}$ corresponds to $c$ negative.
Actually, if we look around $b$, this branch is given by $\sqrt{1} = 1$.


\begin{thm} \label{thm:convergence_FK}
The solutions $\obs$ of the semi-discrete boundary value problems are uniformly close in any compact subset of $\Omega$ to their continuous counterpart $f$ defined by (a), (b) and (c).
In other words, $\obs$ converges uniformly on all compact sets of $\Omega$ to $\sqrt{\Phi'}$ where $\Phi$ is any conformal map from $\Omega$ to $\bbR \times (0, 1)$ mapping $a$ and $b$ to $\mp \infty$ respectively.
\end{thm}

\begin{proof}
We start by showing the convergence of the discrete primitive $\prim := \myi \int_\delta (\obsmedial(\zeta))^2 \dd \zeta$, using the boundary modification trick introduced in Section \ref{sec:boundary_modif}.
We extend $\prim$ on $\modifiedprimaldomain$ and denote its restriction on the primal axes $\extprim$.
By Lemma \ref{lem:extend_H}, $\extprim$ is still subharmonic, thus it is smaller than the harmonic function $h_\delta^\bullet$ with boundary condition $0$ on $\arcab$ and $1$ on $\arcba$.
Proposition \ref{prop:cvg_Dirichlet} tells that $h_\delta^\bullet$ converges to the solution $H$ of the continuous Dirichlet boundary problem with boundary conditions $0$ on $\borderab$ and $1$ on $\borderba$.
We can deduce that
$$
\limsup_{\delta \rightarrow 0} \extprim \leq h
$$
on any compact subset of $\Omega$.
In a similar manner, denote $\extprimdual$ the function $\prim$ extended on $\modifieddualdomain$ which is restricted on dual axes.
As before, this time by superharmonicity, we deduce that
$$
\liminf_{\delta \rightarrow 0} \extprimdual \geq h
$$
on any compact subset of $\Omega$.
By definition (Equation \eqref{eqn:H_edge}), for a sequence of $w_\delta$ and $b_\delta$ neighbors in $\primaldomain$, both approximating $u \in \Omega$ (\emph{i.e.} $w_\delta \rightarrow u$, $b_\delta \rightarrow u$), we have
$$
h(u) \leq \liminf_{\delta \rightarrow 0} \extprimdual (w_\delta)
\leq \limsup_{\delta \rightarrow 0} \extprimdual (b_\delta)
\leq h(u).
$$
Since the convergence to $h$ on $\modifiedprimaldomain$ and $\modifieddualdomain$ is uniform on compact subsets, it is the same for the convergence of both $\extprim$ and $\extprimdual$.

Consider $Q \subset \Omega$ such that $9Q \subset \Omega$.
By the uniform convergence of $\prim$, the family $(\prim)$ is bounded uniformly in $\delta > 0$ on $9Q$.
Theorem \ref{thm:cvg_shol} implies that $\obsmedial$ is a precompact family of semi-discrete s-holomorphic functions on $Q$.

Consider $\delta_n$ a subsequence such that $\calF_{\delta_n}$ converges uniformly on all compact subsets of $\primaldomain$ to $\calF$.
For $u, v \in \primaldomain$ and converging subsequences $u_n \rightarrow u$ and $v_n \rightarrow v$, we have
\begin{align*}
h(v) - h(u) & = \lim_{n \rightarrow \infty} (H_{\delta_n}(v_n) - H_{\delta_n}(u_n)) \\
& = \lim_{n \rightarrow \infty} \myi \int_{v_n}^{u_n} \calF_{\delta_n}(z)^2 \dd z \\
& = \myi \int_v^u \calF(z)^2 \dd z.
\end{align*}
Same as the discussion just above, the limit $\calF$ is given by $\sqrt{\Phi'}$ where $\Phi$ is any conformal map from $\Omega$ to $\bbR \times (0, 1)$ mapping $a$ and $b$ to $\mp \infty$.
\end{proof}

\subsection{RSW property: random-current representation} \label{sec:RSW_prop}

In the previous section, we established the conformal invariance of the limit of our semi-discrete observables.
To show that the interface is given by an SLE curve in the limit and to determine its parameter, we need the so-called RSW property.
This provides the hypothesis needed in \cite{Kemp-SLE} which, along with Theorem \ref{thm:convergence_FK}, shows the main Theorem \ref{thm:main}.

The goal of this section is to show the following property.

\begin{prop}[RSW property] \label{prop:RSW_prop}
Let $\alpha > 0$.
Consider $R_{n, \alpha} = [-n, n] \times [-\alpha n, \alpha n]$ a rectangular domain and write $R_{n, \alpha}^\delta$ for its semi-discretized counterpart (primal domain).
Let $\xi$ a boundary condition on $R_{n, \alpha}^\delta$.
Then, there exists $c(\alpha) > 0$ independent of $n$ and $\delta$ such that
\begin{align*}
c & \leq \bbP^\xi ( \calC_h(R_{n, \alpha}^\delta) ) \leq 1-c, \\
c & \leq \bbP^\xi ( \calC_v(R_{n, \alpha}^\delta) ) \leq 1-c
\end{align*}
where $\calC_h$ and $\calC_v$ denote the events ``having a horizontal / vertical crossing''.
\end{prop}

The proof of Proposition \ref{prop:RSW_prop} is based on the use of the same fermionic observable introduced in Section \ref{sec:observables} and the second moment method to estimate the crossing probabilities.
This idea comes from \cite{RSW-harmonic} where the classical Ising case is treated and here we adapt the proof to the case of quantum Ising.

To show the RSW property in Proposition \ref{prop:RSW_prop}, we only need to show the lower bound for free boundary condition by duality \cite{RSW-harmonic}.
In this section, we will just show the property for the horizontal crossing, since the proof to estimate the probability of the vertical crossing is similar.

We recall that $\Dobrushindomain$ is a Dobrushin domain, meaning that the arc $\arcab$ is wired and the arc $\arcba$ is free.
In Section \ref{sec:boundary_modif}, we introduced the notion of modified primal and dual domains of a Dobrushin domain, which are denoted by $\modifiedprimaldomain$ and $\modifieddualdomain$ respectively.
Let us write $\hmprim$ and $\hmdual$ the harmonic functions on modificed domains $\modifiedprimaldomain$ and $\modifieddualdomain$ having boundary conditions 1 on the (extended) wired arc ($\borderab$ for $\modifiedprimaldomain$ and $\extborderab$ for $\modifieddualdomain$) and 0 on the (extended) free arc ($\borderba$ for $\modifiedprimaldomain$ and $\extborderba$ for $\modifieddualdomain$).

We start by noticing that the connection probability of a vertex next to the free arc $\arcba$ to the wired arc $\arcab$ can be written in a simple way by using the parafermionic observable.

\begin{prop} \label{prop:f_proba}
Let $u \in \primaldomain$ such that $\{ u^+, u^- \} \cap \arcba \neq \emptyset$ (equivalently, $u$ is next to the free arc).
Write $e$ for the mid-edge between $u$ and the free arc $\arcba$.
Then, we have
$$
\PDob (u \leftrightarrow \arcab)^2 = \delta |F(e)|^2.
$$
\end{prop}

\begin{proof}
We take the definition of $F$,
\begin{align*}
\delta |F(e)|^2 & = \left| \EDob \left[ \exp \left( \frac{\icomp}{2} W(e, b_\delta) \right) \mathbbm{1}_{e \in \interface} \right] \right|^2 \\
& = | \EDob [\mathbbm{1}_{e \in \interface}] |^2 = \PDob (e \in \interface)^2
\end{align*}
where the winding $W(e, b_\delta)$ is always a constant if $e$ is adjacent to the boundary.
We also notice that $e \in \interface$ is equivalent to $u$ connected to the wired arc $\arcab$.
\end{proof}

By using harmonic functions $\hmprim$ and $\hmdual$, we can get easily the following proposition.

\begin{prop} \label{prop:proba_harm}
Let $u \in \primaldomain$ next to the free arc.
Write $w \in \{ u^+, u^- \}$ which is not on the free arc.
We have
$$
\sqrt{\hmdual(w)} \leq \PDob (u \leftrightarrow \arcab)  \leq \sqrt{\hmprim(u)}
$$ 
\end{prop}

\begin{proof}
Write $w_\partial$ the neighbor of $u$ which is on the free arc.
We have
$$
H(u) = H(u) - H(w_\partial) = \delta |F(e)|^2 = \PDob (u \leftrightarrow \arcab)^2.
$$
Moreover, by Lemma \ref{lem:extend_H}, $H$ is subharmonic on $\modifiedprimaldomain$, we get $H(u) \leq \hmprim(u)$.
Similarly, writing $e = (u w)$, 
$$
H(w) = \delta |F(e)|^2 - \delta |F(e')|^2 \leq \delta |F(e)|^2 = \PDob (u \leftrightarrow \arcab)^2
$$
and we conclude by superharmonicity of $H$ on $\modifieddualdomain$.
\end{proof}

\begin{killcontents}
\begin{lem} \label{lem:estimate_harm}

\end{lem}

\begin{proof}
 local central limit and gambler's ruin type estimate
\end{proof}
\end{killcontents}

Now, we are ready to show the RSW property.
We keep the same notation as in the statement of Proposition \ref{prop:RSW_prop}.
We write $\leftborder$ and $\rightborder$ for the left and right (primal) borders of $R_{n, \alpha}^\delta$.
We define the random variable $N$ given by the 2D Lebesgue measure of the subset of $\leftborder \times \rightborder \subset \bbR^2$ consisting of points which are connected in $R_{n, \alpha}^\delta$.
More precisely, 
$$
N = \iint_{\substack{x \in \leftborder \\ y \in \rightborder}} \mathbbm{1}_{x \leftrightarrow y} \dd x \dd y.
$$
To show Proposition \ref{prop:RSW_prop}, we use the second moment method.
In other words, by using Cauchy-Schwarz, we need to show that the lower bound of
\begin{equation} \label{eqn:second_moment}
\bbP^0 (N > 0) = \bbE^0 [\mathbbm{1}_{N > 0}^2] \geq \frac{\bbE^0 [N]^2}{\bbE^0 [N^2]}
\end{equation}
is uniform in $n$ and $\delta$.
First, we get a lower bound for $\bbE^0 [N]$.

\begin{lem} \label{lem:estimate_N}
There exists a uniform constant $c$ independent of $n$ and $\delta$ such that
$$
\bbE^0 [N] \geq c n.
$$
\end{lem}

\begin{proof}
We decompose the right boundary into $m = \lfloor n/\delta \rfloor$ parts,
$$
\rightborder = \bigcup_{i=0}^{m-1} \rightborder^i \mbox{ where } \rightborder^i = \left( \{ \alpha n \} \times \left( -n + i \cdot \frac{2n}{m}, -n + (i+1) \cdot \frac{2n}{m} \right) \right).
$$
We expand the expectation,
\begin{align*}
\bbE^0 [N] & = \iint_{\substack{x \in \leftborder \\ y \in \rightborder}} \bbP^0 (x \leftrightarrow y) \dd x \dd y \\
& = \int_{x \in \leftborder} \sum_{i=0}^{m-1} \bbP^0 (x \leftrightarrow \rightborder^i) \dd x.
\end{align*}
By Proposition \ref{prop:proba_harm}, each $\bbP^0(x \leftrightarrow \rightborder^i)$ can be bounded from below by $\hmdual(w)$ where $w$ is a neighbor of $x$ which is not on the free arc, and the harmonic measure is with respect to the modified domain $\modifieddualdomain$ where the Dobrushin domain $\Dobrushindomain$ is given by $\medialdomain = R_{n, \alpha}^\delta$, $a_\delta$ and $b_\delta$ such that $\rightborder^i = (a_\delta b_\delta)$.
Moreover, from the local central limit and gambler's ruin-type estimate, we have that $\hmdual(w) \geq c (\delta / n)^2$ for a $c > 0$ uniform in $n$, $\delta$ and $i$.

Finally, we get
$$
\bbE^0 [N] \geq \int_{x \in \leftborder} m \cdot \sqrt{c} \frac{\delta}{n} \dd x \geq c' n
$$
where $c'$ is a uniform constant.
\end{proof}

To estimate $\bbE^0 [N^2]$, we need Proposition \ref{lem:two_pts}, a consequence of Lemma \ref{lem:arm_exp}.
Both of them make use of the so-called \emph{exploration path}, which is the interface between the primal wired cluster and the dual free cluster.
The proof in the discrete case \cite{RSW-harmonic} can be easily adapted to the semi-discrete case, since the interface is well-defined and we have similar estimates on harmonic functions, by means of semi-discrete Brownian motion, local central limit and gambler's ruin-type estimates.
Therefore, we will just give the proof of Lemma \ref{lem:arm_exp}.

For any given $\alpha, n$ and $\delta$, let us consider $R_{n, \alpha}$ as before and $(R_{n, \alpha}^\delta, a_\delta, b_\delta)$ the semi-discretized Dobrushin domain obtained from $R_{n, \alpha}$ with the right boundary $\rightborder = \arcab$ which is wired.

\begin{prop}[\protect{\cite[Proposition 14]{RSW-harmonic}}] \label{lem:two_pts}
There exists a constant $c > 0$ which is uniform in $\alpha$ and $n$ such that for any rectangle $R_{n, \alpha}^\delta$ and any two points $x, z \in \leftborder$, we have
$$
\bbP_{(R_{n, \alpha}^\delta, a_\delta, b_\delta)} (x, z \leftrightarrow \rightborder)
= \bbP^0_{R_{n, \alpha}^\delta} (x, z \leftrightarrow \rightborder)
\leq \frac{c}{\sqrt{|x- z| n}}
$$
\end{prop}

\begin{lem}[\protect{\cite[Lemma 15]{RSW-harmonic}}] \label{lem:arm_exp}
There exists a constant $c > 0$ which is uniform in $\alpha$, $n$, $\delta$ and $x \in \leftborder$ such that for any rectangle $R_{n, \alpha}^\delta$ and all $k \geq 0$,
$$
\bbP_{(R_{n, \alpha}^\delta, a_\delta, b_\delta)} (\ball (x, k) \leftrightarrow \arcab) \leq c \sqrt\frac{k}{n}.
$$
\end{lem}

\begin{proof}
Let $n, k, \delta, \alpha > 0$, the rectangular domain $R_{n, \alpha}^\delta$ and its semi-discrete counterpart, the Dobrushin domain $(R_{n, \alpha}^\delta, a_\delta, b_\delta)$ where $\rightborder = \arcab$ is the wired arc.
Consider $x \in \leftborder$.
For $k \geq n$, the inequality is trivial, so we can assume $k < n$.

Since the probability $\bbP_{(R_{n, \alpha}^\delta, a_\delta, b_\delta)} (\ball (x, k) \leftrightarrow \arcab)$ is non-decreasing in $\alpha$, we can bound it by above by replacing $\alpha$ by $\alpha+1$, which we bound by above by a longer wired arc $\arccd$, where $c_\delta$ and $d_\delta$ are respectively the left-bottom and the left-top points of the rectangular domain $R_{n, \alpha}^\delta$.
See Figure \ref{fig:recdoms} for notations.
\begin{align*}
\bbP_{(R_{n, \alpha}^\delta, a_\delta, b_\delta)} (\ball (x, k) \leftrightarrow \arcab)
& \leq \bbP_{(R_{n, \alpha+1}^\delta, a_\delta, b_\delta)} (\ball (x, k) \leftrightarrow \arcab) \\
& \leq \bbP_{(R_{n, \alpha+1}^\delta, c_\delta, d_\delta)} (\ball (x, k) \leftrightarrow \arccd)
\end{align*}

\begin{figure}[htb] \centering
  \begin{minipage}{.22\textwidth}
    \includegraphics[scale=0.85, page=1]{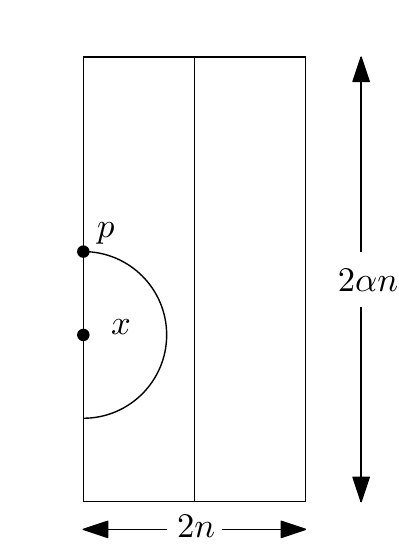}
  \end{minipage}
  \hfill
  \begin{minipage}{.22\textwidth}
    \includegraphics[scale=0.85, page=2]{images/rec_domain.pdf}
  \end{minipage}
  \hfill
  \begin{minipage}{.22\textwidth}
    \includegraphics[scale=0.85, page=3]{images/rec_domain.pdf}
  \end{minipage}
  \hfill
  \begin{minipage}{.22\textwidth}
    \includegraphics[scale=0.85, page=4]{images/rec_domain.pdf}
  \end{minipage}
  \caption{\emph{First:} The primal domain $R_{n, \alpha}^\delta$. \emph{Second:} The Dobrushin domain $(R_{n, \alpha}^\delta, a_\delta, b_\delta)$. \emph{Third:} The Dobrushin domain $(R_{n, \alpha}^\delta, c_\delta, d_\delta)$. \emph{Third:} The Dobrushin domain $(R_{n, \alpha}^\delta, c_\delta, d_\delta)$. \emph{Fourth:} The exploration path starting at $c_\delta$ and touching $\ball(x, k)$ at time $T$.}
  \label{fig:recdoms}
\end{figure}

Let $\interface$ be the interface as defined in Section \ref{sec:QI_loop} and used in Section \ref{sec:observables} to define the fermionic observable.
The definition of $\interface$ tells us that the ball $\ball(x, k)$ is connected to the wired arc if and only if $\interface$ goes through a mid-edge which is adjacent to the ball.
We parametrize $\interface$ by its length and denotes $T$ the hitting time of the set of the mid-edges adjacent to the ball $\ball(x, k)$.
Therefore, $\ball(x, k)$ is connected to the wired arc if and only if $T < \infty$.

Write $p$ for the top-most point of $\ball(x, k)$.
We can rewrite the probability of $\{ p \leftrightarrow \arccd \}$ by conditionning on $\interface[0, T]$ and using Markov domain property to obtain
\begin{align*}
\bbP_{(R_{n, \alpha+1}^\delta, c_\delta, d_\delta)} (p \leftrightarrow \arccd)
& = \bbE_{(R_{n, \alpha+1}^\delta, c_\delta, d_\delta)} [ \mathbbm{1}_{T < \infty} \bbP_{(R_{n, \alpha+1}^\delta, c_\delta, d_\delta)} (p \leftrightarrow \arccd \mid \interface[0, T]) ] \\
& = \bbE_{(R_{n, \alpha+1}^\delta, c_\delta, d_\delta)} [ \mathbbm{1}_{T < \infty} \bbP_{(R_{n, \alpha+1}^\delta \backslash \interface[0, T], \interface(T), d_\delta)} (p \leftrightarrow \arccd) ].
\end{align*}

We estimate this quantity in two different ways to get the desired inequality.
Firstly, since $p$ is at distance at least $n$ from the wired arc, we have
$$
\bbP_{(R_{n, \alpha+1}^\delta, c_\delta, d_\delta)} (p \leftrightarrow \arccd) \leq \frac{c_1}{\sqrt{n}},
$$
which follows from Proposition \ref{prop:proba_harm} and the fact that $\hmprim \leq \frac{c}{n}$.
Secondly, we can write $\interface(T)$ as $z + (s, -r, s)$ where $0 \leq s \leq k$ and $0 \leq r \leq 2k$.
Thus, the line $z + \bbZ \times \{ -r \}$ disconnects $a$ from the free arc, we estimate the harmonic function and we obtain a.s.
$$
\bbP_{(R_{n, \alpha+1}^\delta \backslash \interface[0, T], \interface(T), d_\delta)} (p \leftrightarrow \arccd)
\geq \frac{c_2}{\sqrt{r}} \geq \frac{c_2}{\sqrt{2k}}
$$
which again comes from Proposition \ref{prop:proba_harm} and the estimate $\hmdual \geq \frac{c}{r}$.
This final estimate being uniform in $\interface[0, T]$, we get
$$
\frac{c_2}{\sqrt{2k}} \bbP_{(R_{n, \alpha+1}^\delta, c_\delta, d_\delta)} ( T < \infty)
\leq \bbP_{(R_{n, \alpha+1}^\delta, c_\delta, d_\delta)} (p \leftrightarrow \arccd)
\leq \frac{c_1}{\sqrt{n}}.
$$
which implies the statement.
\end{proof}

Now, we can complete the proof of the RSW property.

\begin{proof}[Proof of Proposition \ref{prop:RSW_prop}]
From Equation \eqref{eqn:second_moment} and Lemma \ref{lem:estimate_N}, we just need to show that
\begin{equation} \label{eqn:N2}
\bbE^0 [N^2] = \iiiint \bbP^0 (x \leftrightarrow y, z \leftrightarrow t) \dd x \dd y \dd z \dd t
\end{equation}
is $\bigO{n^2}$.

Consider $x, z \in \leftborder$, $y, t \in \rightborder$ and $l$ the middle vertical line separating them, we get
\begin{align*}
\bbP^0(x \leftrightarrow y, z \leftrightarrow t) \leq 
\bbP^0(x, z \leftrightarrow l) \bbP^0(y, t \leftrightarrow l) 
\end{align*}
because the left-hand side of $l$ and the right-hand side of $l$ are independent.
Therefore, the integral in Equation \eqref{eqn:N2} can be cut into two independent parts, each of whom gives the same contribution,
\begin{align*}
& \iiiint_{\substack{x, z \in \leftborder \\ y, t \in \rightborder}} \bbP^0 (x \leftrightarrow y, z \leftrightarrow t) \dd x \dd y \dd z \dd t \\
& \qquad = \left( \iint_{x, z \in \leftborder} \bbP^0 (x, z \leftrightarrow l) \dd x \dd z \right) \left( \iint_{y, t \in \rightborder} \bbP^0 (y, t \leftrightarrow l) \dd y \dd t \right) \\
& \qquad = \left( \iint_{x, z \in \leftborder} \bbP^0 (x, z \leftrightarrow l) \dd x \dd z \right)^2 \\
& \qquad \leq \left( \iint_{x, z \in \leftborder} \frac{c}{\sqrt{|x-z| n}} \dd x \dd z \right)^2 \leq (4cn)^2
\end{align*}
where in the last line we use Proposition \ref{lem:two_pts}.

The proof is thus complete.
\end{proof}

Here we give a brief idea to the proof for the vertical crossing.
To start with, we need to establish propositions similar to Propositions \ref{prop:f_proba} and \ref{prop:proba_harm}.
We will get $\sqrt{\dy \hmprim}$ and $\sqrt{\dy \hmdual}$ in the statement.
And to estimate these harmonic functions, we can use Harnack Principle (Proposition \ref{prop:Harnack}) to get the correct orders.
Then, the end of the story is the same, since we can always define the exploration path and get the same estimates (Lemmas \ref{lem:estimate_N} and \ref{lem:two_pts}).

\subsection{Conclusion: proof of the main Theorem}

We have all the necessary ingredients to conclude the proof of the main Theorem:
\begin{enumerate}
\item The RSW property shown in the previous section gives the G2 condition mentioned in \cite{Kemp-SLE}, giving as conclusion that the family of interfaces $(\interface)$ is tight for the weak convergence.
\item The fact that the fermionic observable (seen as an exploration process) is a martingale and is conformally invariant allows us to identify the limit via Itô's formula.
More precisely, if $\gamma$ is a subsequential limit of the interface parametrized by a Löwner chain $W$, from property of martingales and Itô's formula, we prove that $(W_t)$ and $(W_t^2 - \kappa t)$ are both martingales ($\kappa = 16/3$ for quantum FK-Ising).
The computation is exactly the same as in the limit of the classical FK-Ising since we have the same Riemann-Hilbert Boundary value problem in continuum and same martingales.
Readers who are interested in more details, see \cite{DS-conf-inv, Duminil-book}.
\end{enumerate}

%% file: residue2.tex
\section{Computation of residues} \label{app:residues}

For a non-negative integer $k$ and $m \in \bbZ$, define 
\begin{align*}
g_{k, m} (z) := \frac{1}{z} \left( \frac{1}{z+1} + \frac{1}{z-1} \right)^k \left( \frac{z+1}{z-1} \right)^{2m}
 = \frac{2^k z^{k-1}}{(z-1)^{k+2m} (z+1)^{k-2m}}.
\end{align*}

\begin{lem}
When $k=0$, we have $\Res(g_{k, m}, 1) = \Res(g_{k, m}, -1) = 0$ for all $m$.
\end{lem}

\begin{proof}
When $m = 0$, the result is trivial.
Assume $m \in \bbN^*$, the function $g_{k, m}$ does not have any pole at $-1$, so it is clear that $\Res(g_{k, m}, -1) = 0$.
The residue of $g_{k, m}(z)$ at $z = 1$ is the residue of $g_{k, m}(y+1)$ at $y = 0$.
We have
\begin{align*}
g_{k, m} (y+1) & = \frac{1}{y+1} \left( 1 + \frac2y \right)^{2m} \\
& = \left[ \sum_{k \geq 0} (-y)^k \right]
\left[ \sum_{l=0}^{2m} {2m \choose l} \left( \frac2y \right)^l \right],
\end{align*}
thus the coefficient of $\frac{1}{y}$ is given by
\begin{align*}
\sum_{l=1}^{2m} {2m \choose l} 2^l (-1)^{l-1} = - [(1-2)^{2m} - 1] = 0.
\end{align*}
When $m$ is negative, the proof is similar.
\end{proof}

\begin{lem} \label{lem:sum_residues_zero}
For $k \geq 1$, we have $\Res(g_{k, m}, 1) + \Res(g_{k, m}, -1) = 0$.
\end{lem}

\begin{proof}
When $k \geq 1$, the singularity at $0$ is removable.
We observe that $|g_{k, m} (z)|$ behaves like $|z|^{-k-1} \leq |z|^{-2}$ when $|z|$ is large, giving
$$
\lim_{R \rightarrow \infty} \frac{1}{2 \pi \icomp} \int_{\partial B(0, R)} g_{k, m} (z) \dd z = 0.
$$
Moreover, when $R > 1$,
$$
\frac{1}{2 \pi \icomp} \int_{\partial B(0, R)} g_{k, m} (z) \dd z = \Res(g_{k, m}, -1)  + \Res(g_{k, m}, 1),
$$
giving us the wanted result.
\end{proof}

\begin{lem} \label{lem:certain_residues_zero}
For $k \geq 1$ and $k \leq 2 |m|$, $\Res(g_{k, m}, 1) = \Res(g_{k, m}, -1) = 0$.
\end{lem}

\begin{proof}
We use the previous lemma.
It is enough to show that the residue is zero at either $1$ or $-1$.
If $m$ is positive, we notice that $g_{k, m}(z)$ does not have any pole at $-1$, thus the residue at $-1$ is zero.
If $m$ is negative, the residue at $1$ is zero.
\end{proof}

\begin{lem} \label{lem:general_zero}
More generally, for all $k$ even integer, 
$$
\Res(g_{k, m}, 1) = \Res(g_{k, m}, -1) = 0.
$$
\end{lem}

\begin{proof}
Assume that $k = 2l > m \geq 0$ for a positive integer $l$.
Look at the residue of $g_{k, m} (z)$ around $z = 1$ is equivalent to looking at the residue of $g_{k, m} (y+1)$ around $y=0$,
$$
\Res(g_{k, m}(z), z=1) = \Res(g_{k, m}(y+1), y=0).
$$
We have the following equivalent relations
\begin{align*}
& \quad \Res(g_{2l, m}(y+1), y=0) = 0 \\
\Leftrightarrow & \quad \Res \left( \frac{(y+1)^{2l-1}}{y^{2l+2m}(y+2)^{2l-2m}}, y=0 \right) = 0 \\
\Leftrightarrow & \quad \frac{(1+y)^{2l-1}}{(1+\frac{y}{2})^{2l-2m}} [y^{2l+2m-1}]= 0
\end{align*}
We develop the rational fraction to evaluate this coefficient where the following three identities are useful,
\begin{align}
{{2m-2l \choose 2m+p}} & = (-1)^{2m+p} {{2l+p-1 \choose 2m+p}} \label{eqn:neg_binomial} \\
{{2l+p-1 \choose 2m+p}} & = \sum_{q=0}^{2l-1} {{p \choose q}}  {{2l-1 \choose 2m+p-q}}, \label{eqn:comb} \\
\sum_{k=0}^n {{n \choose k}} {{ k \choose r}} (-x)^k & = (-x)^r (1-x)^{n-r} {{n \choose r}}. \label{eqn:dev_poly}
\end{align}
Equation \eqref{eqn:neg_binomial} comes from the general definition of binomial coefficients, here $2m-2l < 0$.
Equation \eqref{eqn:comb} is an easy combinatorial identity and Equation \eqref{eqn:dev_poly} a simple development.
Then,
\begin{align*}
& \sum_{p=0}^{2l-1} {{2l-1 \choose p}} \left( \frac12 \right)^{2l+2m-1-p} {{2m-2l \choose 2l+2m-1-p}} \\
(p \leadsto 2l-1-p) & \quad = \sum_{p=0}^{2l-1} {{2l-1 \choose p}} \left( \frac12 \right)^{2m+p} {{2m-2l \choose 2m+p}} \\
& \quad = \sum_{p=0}^{2l-1} {{2l-1 \choose p}} \left( -\frac12 \right)^{2m+p} {{2l+p-1 \choose 2m+p}} \\
(\mbox{Equation } \eqref{eqn:comb}) & \quad = \sum_{p=0}^{2l-1} {{2l-1 \choose p}} \left( -\frac12 \right)^{2m+p} 
\sum_{q=0}^{2l-1} {{p \choose q}}  {{2l-1 \choose 2m+p-q}} \\
(q \leadsto p-q) & \quad = \sum_{p=0}^{2l-1} {{2l-1 \choose p}} \left( -\frac12 \right)^{2m+p} 
\sum_{q=0}^{2l-1} {{p \choose q}}  {{2l-1 \choose 2m+q}} \\
& \quad = \left( \frac12 \right)^{2m}  \sum_{q=0}^{2l-1} {{2l-1 \choose 2m+q}}
\sum_{p=0}^{2l-1} {{2l-1 \choose p}} {{p \choose q}} \left( -\frac12 \right)^p \\
(\mbox{Equation } \eqref{eqn:dev_poly}) & \quad = \left( \frac12 \right)^{2m}  \sum_{q=0}^{2l-1} {{2l-1 \choose 2m+q}}
\left( -\frac12 \right)^q \left( \frac12 \right)^{2l-1-q} {{2l-1 \choose q}} \\
& \quad = \left( \frac12 \right)^{2m+2l-1}  \sum_{q=0}^{2l-1} (-1)^q {{2l-1 \choose 2m+q}} {{2l-1 \choose 2l-1-q}} =0
\end{align*}
where the last sum in the final quantity is the coefficient in front of $x^{2m+2l-1}$ in $(1-x)^{2l-1}(1+x)^{2l-1} = (1-x^2)^{2l-1}$, which is zero because the polynomial only contains monomials of even degree.
\end{proof}

\begin{prop} \label{prop:int_zero}
For $\zeta \in \primallattice$, let $f_\zeta$ as defined in \ref{sec:Greens}.
Consider $C$ a path as defined in Proposition \ref{prop:free_Greens}.
Then,
$$
\int_C f_\zeta(z) \dd z = 0.
$$
\end{prop}

\begin{proof}
We will develop the exponential into series and show that this integral is zero for all terms.
We can do this because exponential converges uniformly on all compacts, it makes sense to exchange integral and series.
After developping, we get
$$
f_\zeta(z) = \sum_{k \geq 0} (2 \icomp t)^k g_{k, m} (z)
$$
and Lemma \ref{lem:sum_residues_zero} allows us to conclude.
\end{proof}

\begin{prop} \label{prop:Greens_Ck}
Around $\zeta = m \in \bbZ \backslash \{ 0 \}$, the function $\Greens$ is $\calC^\infty$.
\end{prop}

\begin{proof}
To show this, we need to check that for all $k \in \bbN$, we have
$$
\int_C g_{k, m}(z) \ln_1(z) \dd z = \int_C g_{k, m}(z) \ln_2(z) \dd z
$$
where $\ln_i$ is chosen such that $\ln_i(1) - \ln_i(-1) = (-1)^i \icomp \pi$.
Here, $\ln_1$ corresponds the logarithm we chose in the upper half-plan and $\ln_2$ in the lower half-plane.
From Proposition \ref{prop:int_zero}, we can fix a lift of logarithm, for example $\ln_1 (1) = \ln_2 (1) = 0$ and $\ln_1(-1) = \icomp \pi$ and $\ln_2(-1) = -\icomp \pi$.

Let $I_1$ the integral on the left-hand side and $I_2$ the one on the right-hand side.
Since $\ln_1 - \ln_2$ is constant in a small neighborhood of $1$ and $-1$, we can write
\begin{align*}
I_1 - I_2 & = \Res(g_{k, m} (z) [\ln_1(z) - \ln_2(z)], z = -1) \\
& = 2 \icomp \pi \Res(g_{k, m} (z), z=-1 ),
\end{align*}
whose value is zero according to Lemmas \ref{lem:certain_residues_zero} and \ref{lem:general_zero}.
\end{proof}

%% file: qi2.bbl
\newcommand{\etalchar}[1]{$^{#1}$}
\begin{thebibliography}{CDCH{\etalchar{+}}14}

\bibitem[AKN93]{AKN-perco-QI}
Michael Aizenman, A~Klein, and C~Newman.
\newblock Percolation methods for disordered quantum {I}sing models.
\newblock Technical report, PRE-33973, 1993.

\bibitem[BG09]{BG-QI-sharp}
J.~E. Bj{\"o}rnberg and G.~R. Grimmett.
\newblock The phase transition of the quantum {I}sing model is sharp.
\newblock {\em J. Stat. Phys.}, 136(2):231--273, 2009.

\bibitem[BG15]{Bobenko-DCA-quad}
Alexander~I Bobenko and Felix G{\"u}nther.
\newblock Discrete complex analysis on planar quad-graphs.
\newblock {\em arXiv preprint arXiv:1505.05673}, 2015.

\bibitem[BH16]{BH-Ising-CLE}
St{\'e}phane Benoist and Cl{\'e}ment Hongler.
\newblock The scaling limit of critical {I}sing interfaces is cle (3).
\newblock {\em arXiv preprint arXiv:1604.06975}, 2016.

\bibitem[Bj{\"o}13]{Bjornberg-infrared}
Jakob~E. Bj{\"o}rnberg.
\newblock Infrared bound and mean-field behaviour in the quantum {I}sing model.
\newblock {\em Comm. Math. Phys.}, 323(1):329--366, 2013.

\bibitem[Bj{\"o}16]{Bjornberg-obs}
Jakob~E Bj{\"o}rnberg.
\newblock Fermionic observables in the transverse {I}sing chain.
\newblock {\em arXiv preprint arXiv:1607.06484}, 2016.

\bibitem[BMS05]{BMS}
Alexander~I. Bobenko, Christian Mercat, and Yuri~B. Suris.
\newblock Linear and nonlinear theories of discrete analytic functions.
  {I}ntegrable structure and isomonodromic {G}reen's function.
\newblock {\em J. Reine Angew. Math.}, 583:117--161, 2005.

\bibitem[B{\"u}c08]{Bucking-circlepack}
Ulrike B{\"u}cking.
\newblock Approximation of conformal mappings by circle patterns.
\newblock {\em Geom. Dedicata}, 137:163--197, 2008.

\bibitem[CDCH{\etalchar{+}}14]{Ising-SLE-convergence}
Dmitry Chelkak, Hugo Duminil-Copin, Cl{\'e}ment Hongler, Antti Kemppainen, and
  Stanislav Smirnov.
\newblock Convergence of {I}sing interfaces to {S}chramm's {SLE} curves.
\newblock {\em C. R. Math. Acad. Sci. Paris}, 352(2):157--161, 2014.

\bibitem[CHI15]{CHI-CI-spin}
Dmitry Chelkak, Cl{\'e}ment Hongler, and Konstantin Izyurov.
\newblock Conformal invariance of spin correlations in the planar {I}sing
  model.
\newblock {\em Ann. of Math. (2)}, 181(3):1087--1138, 2015.

\bibitem[CS11]{CS-discreteca}
Dmitry Chelkak and Stanislav Smirnov.
\newblock Discrete complex analysis on isoradial graphs.
\newblock {\em Adv. Math.}, 228(3):1590--1630, 2011.

\bibitem[CS12]{CS-universality}
Dmitry Chelkak and Stanislav Smirnov.
\newblock Universality in the 2{D} {I}sing model and conformal invariance of
  fermionic observables.
\newblock {\em Invent. Math.}, 189(3):515--580, 2012.

\bibitem[DC13]{Duminil-book}
Hugo Duminil-Copin.
\newblock {\em Parafermionic observables and their applications to planar
  statistical physics models}, volume~25 of {\em Ensaios Matem\'aticos
  [Mathematical Surveys]}.
\newblock Sociedade Brasileira de Matem\'atica, Rio de Janeiro, 2013.

\bibitem[DCHN11]{RSW-harmonic}
Hugo Duminil-Copin, Cl{\'e}ment Hongler, and Pierre Nolin.
\newblock Connection probabilities and {RSW}-type bounds for the
  two-dimensional {FK} {I}sing model.
\newblock {\em Comm. Pure Appl. Math.}, 64(9):1165--1198, 2011.

\bibitem[DCS12]{DS-conf-inv}
Hugo Duminil-Copin and Stanislav Smirnov.
\newblock Conformal invariance of lattice models.
\newblock In {\em Probability and statistical physics in two and more
  dimensions}, volume~15 of {\em Clay Math. Proc.}, pages 213--276. Amer. Math.
  Soc., Providence, RI, 2012.

\bibitem[GOS08]{GOS-entanglement}
Geoffrey~R. Grimmett, Tobias~J. Osborne, and Petra~F. Scudo.
\newblock Entanglement in the quantum {I}sing model.
\newblock {\em J. Stat. Phys.}, 131(2):305--339, 2008.

\bibitem[HS13]{HS-energy-Ising}
Cl{\'e}ment Hongler and Stanislav Smirnov.
\newblock The energy density in the planar {I}sing model.
\newblock {\em Acta Math.}, 211(2):191--225, 2013.

\bibitem[Iof09]{Ioffe-QI}
Dmitry Ioffe.
\newblock Stochastic geometry of classical and quantum {I}sing models.
\newblock In {\em Methods of contemporary mathematical statistical physics},
  volume 1970 of {\em Lecture Notes in Math.}, pages 87--127. Springer, Berlin,
  2009.

\bibitem[Ken02]{Kenyon-laplacian}
R.~Kenyon.
\newblock The {L}aplacian and {D}irac operators on critical planar graphs.
\newblock {\em Invent. Math.}, 150(2):409--439, 2002.

\bibitem[KS12]{Kemp-SLE}
Antti Kemppainen and Stanislav Smirnov.
\newblock Random curves, scaling limits and loewner evolutions.
\newblock {\em arXiv preprint arXiv:1212.6215}, 2012.

\bibitem[Pfe70]{Pfeuty-QI}
Pierre Pfeuty.
\newblock The one-dimensional ising model with a transverse field.
\newblock {\em Annals of Physics}, 57(1):79--90, 1970.

\bibitem[Smi06]{Smirnov-towards}
Stanislav Smirnov.
\newblock Towards conformal invariance of 2{D} lattice models.
\newblock In {\em International {C}ongress of {M}athematicians. {V}ol. {II}},
  pages 1421--1451. Eur. Math. Soc., Z\"urich, 2006.

\bibitem[Smi10]{Smirnov-conformal}
Stanislav Smirnov.
\newblock Conformal invariance in random cluster models. {I}. {H}olomorphic
  fermions in the {I}sing model.
\newblock {\em Ann. of Math. (2)}, 172(2):1435--1467, 2010.

\end{thebibliography}
